\documentclass[12pt]{amsart}

\usepackage{latexsym, amssymb, amsmath}

\usepackage{amsfonts, graphicx}

\newcommand{\be}{\begin{equation}}
\newcommand{\ee}{\end{equation}}
\newcommand{\beq}{\begin{eqnarray}}
\newcommand{\eeq}{\end{eqnarray}}

\newtheorem{thm}{Theorem}[section]
\newtheorem{conj}{Conjecture}[section]

\newtheorem{lma}{Lemma}[section]
\newtheorem{prop}{Proposition}[section]
\newtheorem{cor}{Corollary}[section]

\theoremstyle{remark}
\newtheorem{rem}{Remark}[section]

\numberwithin{equation}{section}

\def\bee{\begin{equation*}}
\def\eee{\end{equation*}}

\def\sgn{\text{\rm sgn}}

\def\C{\mathcal{C}}
\def\R{\mathbb{R}}

\def\fm{\mathfrak{m}}

\def\lf{\left}
\def\ri{\right}
\def\e{\epsilon}
\def\ol{\overline}
\def\R{\Bbb R}
\def\wt{\widetilde}

\def\Ric{\text{\rm Ric}}

\def\Pi{\overline{\displaystyle{\mathbb{II}}}}

\def\a{\alpha}
\def\b{\beta}

\def\ii{\sqrt{-1}}

\def\K{K\"ahler }

\def\Ric{\text{Ric}}
\def\lf{\left}
\def\ri{\right}

\def\a{\alpha}
\def\ol{\overline}

\def\e{\epsilon}

\def\C{\Bbb C}
\def\R{\Bbb R}
\def\CP{{\Bbb C} P}

\def\ba{{\bar{\alpha}}}

\def\Ric{\operatorname{Ric}}

\def\K{K\"ahler\ }
\def\be{\begin{equation}}
\def\ee{\end{equation}}

\def\b{\bar}

\def\ve{\varepsilon}
\def\mf{\mathfrak}
\begin{document}
\title{ K\"ahler  $C$-spaces and quadratic bisectional curvature}

\author{Albert Chau$^1$}

\address{Department of Mathematics,
The University of British Columbia, Room 121, 1984 Mathematics
Road, Vancouver, B.C., Canada V6T 1Z2} \email{chau@math.ubc.ca}

\author{Luen-Fai Tam$^2$}
\address{The Institute of Mathematical Sciences and Department of
 Mathematics, The Chinese University of Hong Kong,
Shatin, Hong Kong, China.} \email{lftam@math.cuhk.edu.hk}
\thanks{$^1$Research
partially supported by NSERC grant no. \#327637-11}
\thanks{$^2$Research partially supported by Hong Kong  RGC General Research Fund
\#CUHK 403011}

\renewcommand{\subjclassname}{%
  \textup{2000} Mathematics Subject Classification}
  %\subjclass[2000]{ Primary 53C55}
\begin{abstract}
In this article we give necessary and sufficient conditions for an irreducible \K $C$-space with $b_2=1$ to have nonnegative or positive  quadratic bisectional curvature, assuming the space is not Hermitian symmetric.  In the cases of the five exceptional Lie groups $E_6, E_7, E_8, F_4, G_2$, the computer package MAPLE is used to assist our calculations.  The results are related to two conjectures of Li-Wu-Zheng.


\end{abstract}
%\date{November, 2012}
\maketitle
\markboth{Albert Chau and Luen-Fai Tam} {Quadratic orthogonal bisectional curvature}

\section{Introduction}\label{intro}

Let $(M^n,g)$ be a \K manifold of complex dimension $n$ and let $o\in M$.  $M$ is said to have   {\it nonnegative quadratic orthogonal bisectional curvature} at $o$ if for any unitary frame $e_i$ at $o$ and real numbers $\xi_i$ we have \begin{equation}\label{QBgeq0}
\sum_{i,j}R_{ i\b i j\b j}(\xi^i-\xi^j)^2\geq 0.
\end{equation}
Here $R_{ i\b i j\b j}=R(e_i,\b e_i,e_j,\b e_j)$. Recall that $M$ is said to have nonnegative   bisectional curvature at $o$ if for any $X,Y\in T_o^{(1,0)}(M)$, $R(X,\b X,Y,\b Y)\ge0$, and $M$ is said to have nonnegative {\it orthogonal}   bisectional curvature at $o$ if $R(X,\b X,Y,\b Y)\ge0$ for all unitary pairs  $X,Y\in T_o^{(1,0)}(M)$. Following \cite{LiWuZheng} we abbreviate by $QB\ge 0$ for nonnegative quadratic orthogonal bisectional curvature, $B\ge 0$ for nonnegative   bisectional curvature and $B^\perp\ge0$ for nonnegative   orthogonal    bisectional curvature. It is obvious that $B\ge 0\Rightarrow B^\perp\ge 0\Rightarrow QB\ge0$. Note that in dimension $n=2$, the conditions $B^\perp\ge0$ and $QB\ge 0$ are the same.

It is well-known that compact manifolds with $B\ge0$ have been completely classified by the works   \cite{M,SY,HowardSmythWu1981,Bando,Mok}. By these works, we know that any compact simply connected \K manifold with $B\ge0$ is either biholomorphic to $ \mathbb{CP}^n$ or is isometrically biholomorphic to an irreducible compact Hermitian symmetric space of rank at least 2. While the condition $B^\perp\ge0$ seems weaker, by the works of Chen \cite{Chen} (see also \cite{wilking}) and Gu-Zhang \cite{GZ} we know that a compact simply connected irreducible \K manifold with $B^\perp\ge0$ is also either biholomorphic to $ \mathbb{CP}^n$ or is isometrically biholomorphic to an irreducible compact Hermitian symmetric space of rank at least 2.  In this sense,  no new compact complex manifolds are introduced when we weaken the condition $B\ge0$ to the condition $B^\perp\ge0$.

  The condition $QB\ge0$ was first considered by Wu-Yau-Zheng \cite{WuYauZheng} where they proved that on a compact \K manifold with $QB\ge0$ any class in the boundary of the \K cone can be represented by a smooth closed  (1,1) form which is everywhere   non-negative.  There are other interesting properties satisfied by compact \K manifolds with $QB\ge0$.  A fundamental property of such manifolds, implicit from earlier works \cite{BG} (see \cite{CT} for additional references)  is that all harmonic real (1,1) forms are parallel.  Recently it is proved in   \cite{CT} that the scalar curvature of such a manifold must be nonnegative, and if the manifold is irreducible then the first Chern class is positive.

  The ultimate goal is to classify \K manifolds with $QB\ge0$.  For the compact case, a partial classification of the de Rham factors of the universal cover of such a manifold is given in \cite{CT}.  Hence it remains to study the structure of compact simply-connected  irreducible \K manifolds with $QB\ge0$. By the parallelness of real harmonic (1,1) forms mentioned above, such \K manifolds also have $b_2=1$ (see \cite{HowardSmythWu1981}).  In view of the above results for $B^\perp\ge0$, one may wonder if any new compact complex manifolds are introduced when we weaken the condition $B^\perp\ge0$ to the condition $QB\ge0$.  To address this, Li, Wu and Zheng \cite{LiWuZheng} constructed the first example of a simply connected irreducible compact \K manifold having $QB\ge0$ which does not support a \K metric having $B^\perp\ge0$. Their example is $(B_3,\a_2)$, the classical \K $C$-space with second Betti number $b_2=1$.
It was further conjectured in that all \K $C$-space with second Betti number $b_2=1$ must have $QB\geq 0$,  and the following conjectures were raised in \cite{LiWuZheng}:

\begin{conj}\label{conj-LWZ}
\

 \begin{itemize}
   \item [(1)] Any \K $C$-space with $b_2=1$ satisfies $QB\ge 0$ everywhere.
   \item [(2)] A compact simply connected irreducible \K manifold $(M^n,g)$ with $QB\ge0$ is biholomorphic to a \K $C$-space with $b_2=1$.
   \item [(3)] In (2), if the manifold is not $\mathbb{CP}^n$, then $g$ is a constant multiple of the standard metric.
 \end{itemize}
 \end{conj}

 A \K $C$-space is a compact simply connected \K manifold such that the group of holomorphic isometries acts transitively on the manifold, see \cite{Wang,Itoh}. There is a complete classification of  \K $C$-spaces with $b_2=1$, and this is associated with the classification of simple complex Lie algebras which are just $A_n= \mathfrak{sl}_{n+1}, B_n=\mathfrak{so}_{2n+1}, C_n=\mathfrak{sp}_{2n}, D_n=\mathfrak{so}_{2n}$ and the exceptional cases $E_6, E_7, E_8, F_4, G_2$.  Motivated by the work in \cite{LiWuZheng}, we establish the following Theorems related to conjectures (1) and (3). For the classical types we have the following:
\newpage
\begin{thm}\label{mainthm}
\

\begin{itemize}
\item [(i)] The \K $C$-space $(B_n,\a_p)$, $n\ge 3$, $1<p<n$ satisfies $QB\ge0$ if and only if $5p+1\le 4n$.  Moreover, $QB>0$ if and only if $5p+1< 4n$.\vskip .1cm
\item [(ii)] The \K $C$-space $(C_n,\a_p)$, $n\ge 3$, $1<p<n$ satisfies $QB\ge0$ if and only if $5p\le 4n+3$.  Moreover, $QB>0$ if and only if $5p< 4n+3$.\vskip .1cm
\item [(iii)] The \K $C$-space $(D_n,\a_p)$, $n\ge 4$, $1<p<n-1$ satisfies $QB\ge0$ if and only if $5p+3\le 4n$.  Moreover, $QB>0$ if and only if $5p+3< 4n$.
\end{itemize}
\end{thm}
For the exceptional cases, we have the following:
\begin{thm}\label{mainthmexceptional}
\

\begin{itemize}
\item [(i)] The \K $C$-space $(G_2,\a_2)$, satisfies $QB>0$. \vskip .1cm
\item [(ii)] The \K $C$-space $(F_4,\a_p)$, $1\leq p\leq 4$ satisfies $QB\geq 0$ iff $p=1,2,4$ in which cases $QB >0$.\vskip .1cm
\item [(iii)] The \K $C$-space $(E_6,\a_p)$, $2\leq p\leq 5$ satisfies $QB\geq 0$ iff $p=2,3,5$ in which cases $QB >0$.\vskip .1cm
\item [(iv)] The \K $C$-space $(E_7,\a_p)$, $1\leq p\leq 6$ satisfies $QB\geq 0$ iff $p=1,2,5$ in which cases $QB >0$.\vskip .1cm
\item [(v)] The \K $C$-space $(E_8,\a_p)$, $1\leq p\leq 8$ satisfies $QB\geq 0$ iff $p=1,2, 8$ in which cases $QB >0$.\vskip .1cm
\end{itemize}
\end{thm}

 We only consider \K $C$-spaces which are not Hermitian symmetric. According to Itoh \cite{Itoh}, Theorem \ref{mainthm} and \ref{mainthmexceptional} includes all such \K $C$-spaces with $b_2=1$.  Here $QB>0$ means that \eqref{QBgeq0} is a strict inequality unless all $\xi_i$ are the same. Note that if $QB>0$, then a small perturbation of the \K metric will still satisfy $QB>0$, see Lemma \ref{l-QB-quadractic-1} (and Remark \ref{QBperturb}). Hence conjecture (1) for the classical types is true only under some restrictions mentioned in Theorem \ref{mainthm}, while conjecture (3) is too strong. Conjecture (2) however may still be true in general.

Theorems \ref{mainthm} and \ref{mainthmexceptional} give more information on the curvature   properties   of \K $C$-spaces with $b_2=1$. It is well-known that $\CP^n$ has $B>0$ and Hermitian symmetric spaces with rank at least 2 have $B\ge0$ but not $B>0$. All other \K $C$-spaces which are non-Hermitian symmetric spaces do not have $B\ge0$ or even $B^\perp\ge0$. On the other hand, Itoh \cite{Itoh} proved that a  \K $C$-space with $b_2=1$ is Hermitian symmetric space if and only if its curvature operator has at most two distinct eigenvalues. Our results show that as far as the sign of curvature is concerned,  \K $C$-spaces with $b_2=1$ which are not Hermitian symmetric are further divided into two groups: some of them satisfy $QB\ge0$ and other do not have such property.

  We give here an idea of the proof and refer to \S \ref{section-basic} for details.  Consider a Lie algebra as above, and a corresponding system $\Delta\subset \R^n$ of root vectors in $\R^n$ where the induced Killing form is induced by the standard Euclidean inner product.  Then each associated  \K $C$-space corresponds to a certain subset of $\mathfrak{m}^+\subset \Delta$ representing a unitary frame in which curvature approximation reduces to taking sums and inner products of the vectors in $\mathfrak{m}^+\subset \Delta$ (we can calculate exact values in the case of bisectional curvatures).   We combine this with symmetry, counting and eigenvalue estimate arguments to obtain Theorem \ref{mainthm}.  For the five exceptional cases in Theorem \ref{mainthmexceptional}, the computer package MAPLE was used to assist our calculations and details are provided in the appendix.  Theorem \ref{mainthm} was proved in an earlier version of this article \cite{CT2}, and while similar in spirit, our proof here is somewhat simpler eliminating the need for many calculations from the appendix of \cite{CT2}.

The organization of the paper is as follows. In \S \ref{section-basic} we will state basic properties and formulae for \K $C$-spaces which will be used throughout the paper.  We will then discuss the condition $QB\ge0$ and $QB>0$ in general, then in relation to the \K $C$-spaces.  In \S \ref{section-classical} we prove Theorem \ref{mainthm} for the classical \K $C$-spaces; details for some of the calculations in these sections can be found in the appendices of  \cite{CT2} which is an earlier version of this article.  In \S \ref{section-exceptional} we present the details of our results on Theorem \ref{mainthmexceptional} for the exceptional \K $C$-spaces, with details of our use of MAPLE provided in the appendix.

The authors would like to thank F. Zheng for valuable comments and interest in this work.

\section{basic facts}\label{section-basic}

\subsection{The \K C-spaces and curvature formulae}\label{section-C-spaces}

 Consider a compact \K $C$-space   $(M,\omega)$ with transitive holomorphic  isometry group $G$, and suppose $b_2(M)=1$.  Then any real $(1,1)$ form $\rho$ on $M$ is given by $\rho=c\omega + \sqrt{-1}\partial\overline{\partial} f$ for some constant $c$ and function $f$ where $\omega$ is the \K form.  Now if $\rho$ is $G$ invariant then $\Delta_g f$ is also  $G$ invariant  and hence constant on $M$.  Thus $f$ is constant on $M$ and  $\rho=c\omega$.  In particular, $g$ is the unique $G$ invariant \K metric on $M$ and it is \K Einstein. For more discussions on \K $C$-space, see \cite{Besse,Itoh,Wang,LiWuZheng}.

 \K $C$-spaces with second Betti number $b_2=1$ are obtained as follows, see \cite{Borel-1, Borel-2, BH, Itoh, LiWuZheng, Wang}. Let $G$ be a simply connected, \ complex Lie group, and let $\mathfrak{g}$ be its Lie algebra with Cartan subalgebra $\mathfrak{h}$ and corresponding root system $\Delta\subset \mf h^*$. Then $\mathfrak{g}=\mathfrak{h}\oplus\bigoplus_{\a\in \Delta} \C E_\a$, where $E_\a$ is a root vector of $\a$. Let $l=\dim_\C \mathfrak{h}$ and
fix a fundamental root system $\a_1,\dots,\a_l \subset \Delta$. This gives an ordering of roots in $\Delta$. Let $\Delta^+$ and $\Delta^-$ be the set of positive and negative roots respectively. Let $K$ be the Killing form for $\mathfrak{g}$. Then we may choose root vectors  $\{E_\a\},\a \in \Delta$ such that
$$
K(E_\a,E_{-\a})=-1, \a\in \Delta^+; \ \ [E_\a,E_\beta]=n_{\a,\beta}E_{\a+\beta}
$$
 such that $n_{\a,\beta}=n_{-\a,-\beta}\in \mathbb{R}$ with $n_{\a,\beta}=0$ if $\a+\beta$ is not a root. Together with a suitable basis in $\mathfrak{h}$, they form a {\it Weyl canonical basis} for $\mathfrak{g}$.  Now for any $1\le r\le l$ and let
$$
\Delta_r^+(k)=\{\sum_i n_i\a_i\in \Delta^+|\ n_r=k\},  \,\,\, \Delta^+_r=\bigcup_{k>0}\Delta^+_r(k).
$$
Let $P$ be the subgroup whose Lie algebra is $\mathfrak{h}\oplus\bigoplus_{\a\in \Delta\setminus \Delta_r^+} \C E_\a$.  Then $G/P$ is a complex homogeneous space having $b_2 =1$.  Now let $$\mathfrak{m}_k^+=\bigoplus_{\a\in \Delta_r^+(k)}\C E_\a, \,\,\, \mathfrak{m}_k^-=\bigoplus_{\a\in \Delta_r^-(k)}\C E_\a, \,\,\, \mathfrak{t}=\mathfrak{h}\oplus\bigoplus_{\a\in \Delta^+(0)}(\C E_\a\oplus \C E_{-\a}).$$Then $\mathfrak{m}^+= \bigoplus_{k>0}\mathfrak{m}_k^+$ can be identified with the tangent space of $G/P$.  As given in \cite{Borel-1, Itoh, LiWuZheng}, the $G$-invariant \K form on $G/P$ is given by:
\begin{lma}\label{l-metric} \hspace{12pt}

 \begin{itemize}
      \item [(i)] In a Weyl canonical basis, let $\omega^\a$, $\omega^{\b\a}$  be the dual of $E_\a$ and $\b E_\a := -E_{-\alpha}$, $\a\in \Delta_r^+$. The $G$ invariant \K form on $G/P$ is

       $$
          g=2\sum_{k>0}k\sum_{\a\Delta_r^+(k)}\omega^\a\cdot \omega^{\b\a}=\sum_{k>0}(-kK)|_{\mathfrak{m}^+_k\times \mathfrak{m}^-_k}.
          $$
      \item [(ii)] $[\mathfrak{t},\mathfrak{m}_k^\pm ]\subset \mathfrak{m}_k^\pm$, $[\mathfrak{m}_k^\pm,\mathfrak{m}_l^\pm ]\subset \mathfrak{m}_{k+l}^\pm$, $[\mathfrak{m}_k^+ ,\mathfrak{m}_k^- ]\subset \mathfrak{t} $. If $k>l>0$, then
          $[\mathfrak{m}_k^+,\mathfrak{m}_l^- ]\subset \mathfrak{m}_{k-l}^+$, $[\mathfrak{m}_k^-,\mathfrak{m}_l^+ ]\subset \mathfrak{m}_{k-l}^-$.
    \end{itemize}
    \end{lma}The \K $C$ space thus obtained is denoted as $(\mathfrak{g}, \alpha_r)$.  Conversely, every \K $C$ space with $b_2 =1$ can be obtained by the construction.  Thus the set  $\{e_{\a}:=1/{\sqrt k}E_\a\}$; $\a\in \Delta_k^+$, $k\ge 1$ forms a unitary basis for the tangent space of $(\mathfrak{g},\a_r)$ in the metric $g$. We call this basis as a {\it Weyl frame}.   To compute the curvature tensor in this frame we have the following  from Li-Wu-Zheng \cite[Proposition 2.1]{LiWuZheng}, using the method in \cite{Itoh}. For the sake of completeness, we give a proof.

\begin{prop}\label{prop-LiWuZheng}[Li-Wu-Zheng]  Let $X^i \in \fm^+_i, Y^j \in \fm^+_j, Z^k \in \fm^+_k, W^l  \in \fm^+_l$.  Suppose $i+k=j+l$. Then
\be
\begin{split}
R(X^i,\b Y^j, &Z^k,\b W^l)= \lf((k-j)\xi_{k-j}-\frac{kl}{i+k}\ri)K([X^i,Z^k],[\b Y^j,\b W^l]) \\
&+\lf(- (k-j)\xi_{k-j}+k\xi_{i-j}+l\xi_{j-i}+l\delta_{ij}\delta_{kl}\ri)K([X^i,\b Y^j],[Z^k,\b  W^l]).\\
\end{split}
\ee
 $R(X^i,\b Y^{j},Z^k,\b W_l)=0$ if $i+k\neq j+l$.
\end{prop}
\begin{proof}
  Note that $g(U,\b V)=-kK(U,\b V)$, on $\fm_k^+\times   \fm_k^{-}$ etc.

Then \cite[p.43]{Itoh}
\be
\begin{split}
R(X^i,\b Y^{j},Z^k,\b W_l)=&g(R(X^i, \b Y^{j}) Z^k, \b W^l)\\
=&g([\Lambda(X^i), \Lambda(\b Y^{j})] Z^k, \b W^l)-g(\Lambda([X^i,\b Y^j]_{\fm})Z^k,\b W^l)\\
&-g([[X^i,\b Y^j]_{\mathfrak{t}},Z^k],\b W^l),
\end{split}
\ee
where $\Lambda(U)V =n'/(n+n')[U,V]_{\fm^+} $ if $U\in \fm^+_n, V\in \fm^+_{n'}$, and $\Lambda(\b U)V=[\b U,V]_{\fm^+}, $ for all $U, V\in \fm^+$, see \cite[p.45]{Itoh}.  Here $[U,V]_{\fm^+}$ is the component of $[U,V]$ in $\fm^+$.
Now if $i+k=j+l$,
\be
\begin{split}
 [\Lambda(X^i), \Lambda(\b Y^{j})] Z^k=&\lf(\Lambda(X^i)\Lambda(\b Y^{j})- \Lambda(\b Y^{j})\Lambda(X^i)\ri)Z^k\\
 =&\Lambda(X^i)([\b Y^j,Z^k]_{\fm^+})- \frac k{i+k}\Lambda(\b Y^{j})([X^i,Z^k])\\
 =&\frac{(k-j)}l\xi_{k-j} [X^i,[\b Y^j,Z^k]]-\frac k{i+k}[\b Y^{j},[X^i,Z^k]]_{\fm^+} \ \
\end{split}
\ee
Here   each term is in $\fm^+_l$.
Hence
\be
\begin{split}
g([\Lambda(X^i)&, \Lambda(\b Y^{j})] Z^k, \b W^l)\\
=& -(k-j) \xi_{k-j}K ([X^i,[\b Y^j,Z^k]],\b W^l)+\frac {kl}{i+k}K([\b Y^{j},[X^i,Z^k]],\b W^l)\\
=& -(k-j) \xi_{k-j}K( X^i,[[\b Y^j,Z^k] ,\b W^l])+\frac {kl}{i+k}K([\b Y^{j},[X^i,Z^k]],\b W^l)\\
=& (k-j) \xi_{k-j}K( X^i,[ [ \b W^l,\b Y^j ],Z^k ]+[[ Z^k,\b W^l]],\b Y^j)\\ &+\frac {kl}{i+k}K([ \b W^l,\b Y^{j}],[X^i,Z^k]])\\
=& -(k-j) \xi_{k-j}K([X^i,\b Y^j],[ Z^k,\b W^l]) \\&
+\lf((k-j) \xi_{k-j}-\frac {kl}{i+k}\ri)K( [X^i,  Z^k],[\b Y^j  ,\b W^l] ).
\end{split}
\ee
Now
$
[X^i,\b Y^j]_{\fm}$ is in $\fm_{i-j}^+$ if $i>j$, and $\fm_{j-i}^-$ if $j>i$, and is 0 if $i=j$. So
$$
g(\Lambda([X^i,\b Y^j]_{\fm})Z^k,\b W^l)=-\lf(k\xi_{i-j}+l\xi_{j-i}\ri)K\lf([X^i,\b Y^j],[Z^k,\b W^l]\ri).
$$
Also, $[X^i,\b Y^j]_{\mathfrak{t}}=0$ unless $i=j$. If $i=j$, then $[[X^i,\b Y^j]_{\mathfrak{t}},Z^k]\in \fm_k^+$. Hence
$$
g([[X^i,\b Y^j]_{\mathfrak{t}},Z^k],\b W^l)= \delta_{ij}\delta_{kl}g([[X^i,\b Y^j] ,Z^k],\b W^l)=-l\delta_{ij}\delta_{kl}K\lf([X^i,\b Y^j],[Z^k,\b W^l]\ri).
$$
Also, $R(X^i,\b Y^{j},Z^k,\b W_l)=0$ if $i+k\neq j+l$.
\end{proof}

\begin{lma}\label{l-Weylframe} Same notations as in Proposition \ref{prop-LiWuZheng}.   Assume also that $X, Y, Z, W$ are canonical Weyl basis vectors. Suppose $X=Y$, and $Z\neq W$, then $$R(X,\b X, Z,\b W)=0.$$
\end{lma}
\begin{proof}

Since $i=j$, the lemma is true if $k\neq l$ by Proposition \ref{prop-LiWuZheng}. Hence we assume $k=l$. We first  assume that $k\le i$. Then
$$
R(X,\b X, Z,\b W)=-\frac{k^2}{i+k}K([X,Z],[\b X,\b W])+kK([X,\b X],[Z,\b W]).
$$
Now let $X=E_{\a}$, $Z=E_\beta$, $W=E_\gamma$ with $\beta\neq \gamma$. Note that $\b E_\a=-E_{-\a}$.  Then $[X,Z]=n_{\a,\beta}E_{\a+\beta}$, $[  X,  W]=n_{\a,\gamma}E_{ \a+\gamma}$. Hence
 $$K([X,Z],[\b X,\b W])=-n_{\a,\beta}n_{\a,\gamma}K(E_{\a+\beta},E_{-\a-\beta}).$$
 If $ \a+\beta $ or $ \a+\gamma $ is not a root, then $n_{\a,\beta}=0$ or $n_{\a,\gamma}=0$ and $ K([X,Z],[\b X,\b W])=0$. Otherwise,
 both $E_{\a+\beta}$ and $E_{\a+\gamma}$ are canonical Weyl basis vectors and  are in $ \mathfrak{m}^+_{i+k}$ by Lemma \ref{l-metric}. Since $\beta\neq \gamma$,  and $K$ is proportional to $g$ on $ \mathfrak{m}^+_{i+k}\times \mathfrak{m}^-_{i+k}$, we also have $K([X,Z],[\b X,\b W])=0.$

On the other hand, by the fact that $K([x,y],z)=K(z,[y,z])$, we have
\be\label{Killing_jacobi}
K([X,\b X],[Z,\b W])=K( X,[\b X,[Z,\b W]]).
\ee
Now $[Z,\b W]=n_{\beta,-\gamma}E_{\beta-\gamma}$, $[\b X,[Z,\b W]]=n_{-\a, \beta-\gamma}n_{\beta,-\gamma}E_{-\a+\beta-\gamma}  $. If $\beta-\gamma$ or $-\a+\beta-\gamma$ is not a root, then as before we have $K( X,[\b X,[Z,\b W]])=0$. Otherwise, by Lemma \ref{l-metric}, $[Z,\b W]\in \mathfrak{t}$ and  $[\b X,[Z,\b W]]\in \mathfrak{m}_i^-$.
Since $-\a+\beta-\gamma\neq -\a$, so as before
$$
K([X,\b X],[Z,\b W])=K([X,[\b X,[Z,\b W])=0.
$$
Hence the lemma is true when $k\le i$.

Suppose $i<k$. Then it is equivalent to prove $R(X,\b Y,Z,\b Z)=0$ but assuming $i>k$ and $X\neq Y$.  In this  case,
$$
R(X,\b Y, Z,\b Z)=-\frac{k^2}{i+k}K([X,Z],[\b Y,\b Z])+kK([X,\b Y],[Z,\b Z]).
$$
The previous argument implies the lemma is true in this case as well.
\end{proof}

 To use the formula in Proposition \ref{prop-LiWuZheng} we need to compute the Lie bracket and Killing form in the given Weyl basis.  Now the Killing form $K$ is negative definite on $\mathfrak{h}$ and thus induces a positive definite bilinear form, denoted also by $K$, on the dual $\mathfrak{h}^*$.  We can then identify $\mathfrak{h}^*$ with $\R^l$ (or a subspace of some $\R^n$) so that $K$ becomes the standard Euclidean inner product and the root system is represented by a subset $\Delta\subset\R^l$.   It turns out that a corresponding Weyl basis $\{E_{\a}\}_{\a \in \Delta^+}$ exists in which the Lie bracket and Killing form are computed in terms of Euclidean inner products and addition of the vectors $\a$.  We describe this in more detail below.

Let $\mathfrak{g}$ be a semi simple Lie algebra, and let $\Delta=\{\a,\beta,...\}\subset \R^n$ be a corresponding root system with standard inner product $(\cdot, \cdot)$ corresponding to the induced Killing form $K$.   To the positive roots there corresponds a $Chevalley$ $basis$ $\{X_\a, X_{-\a}, H_\a\}_{\a\in \Delta^+}$ where for each $\a$, $X_\a, X_{-\a}$ are root vectors for $\a, -\a$ respectively, $H_\a \in \mathfrak{h}$,  and the following relations are satisfied (see \cite[ p.51]{Serre}):
\be\label{liebracketeuclid}
\begin{split}
[X_\a,X_{-\a}]=&H_\a,\\
[X_\a,X_\beta]= &\left\{
                  \begin{array}{ll}
                   N_{\a,\beta}X_{\a+\beta},   & \hbox{if $\a+\beta$ is a root, $N_{\a,\beta}=-N_{-\a,-\beta}$
;} \\
                    0, & \hbox{if $\a+\beta\neq0$ is not a root.}
                  \end{array}
                \right.\\
                N_{\a,\beta}=&\pm (p+1), \text{\ $p$ is the largest integer so that $\beta-p\a$ is a root}.\\
                [H_\a,X_{\beta}]=&\beta(H_\a)X_\beta.
                 \end{split}
                \ee
We also have:
 \begin{enumerate}
           \item [(a)] $\beta(H_\a)=2\frac{(\beta,\a)}{(\a,\a)}$,
           \text{\rm \cite[p.337]{Fulton-Harris}};
           \item [(b)] $K(H_\a,H_\a)=2K(X_\a,X_{-\a})$,  \text{\rm\cite[p.207]{Fulton-Harris}}.
         \end{enumerate}
By these and \cite[p.207-208]{Fulton-Harris}, we have the formulas

\be\label{killingeuclid}
\begin{split}
K(H_\a,H_{-\a})&=\frac4{(\a,\a)},\\
K(X_\a,X_{-\a})&=\frac2{(\a,\a)},\\
K(H_\a,H_\beta)&=\frac{4(\a,\beta)}{|\a|^2|\beta|^2}.
\end{split}
\ee
Where in the last equation we have used the first equation of \eqref{liebracketeuclid}, \eqref{Killing_jacobi},  (a),  and the second formula in \eqref{killingeuclid}.

\begin{lma}\label{l-Weyl basis}
For positive roots $\a$, let
\be
E_\a=\frac {|\a|}{\sqrt 2}X_\a, E_{-\a}=-\frac {|\a|}{\sqrt 2}X_{-\a}.
\ee
Then for positive roots $\a,\beta$
\be
\begin{split}
K(E_\a,E_{-\a})=&-1;\\
[E_\a,E_\beta]=&n_{\a,\beta}E_{\a+\beta};\\
[E_{-\a},E_{-\beta}]=& n_{-\a,-\beta}E_{-\a-\beta};\\
[E_\a,E_{-\beta}]=&n_{\a,-\beta}E_{\a-\beta}, \text{\rm if $\a-\beta\neq0$},
\end{split}
\ee
where

$$
n_{-\a,-\beta}=n_{\a,\beta}
= \left\{\begin{array}{ll}\frac {|\a||\beta|}{ \sqrt2|\a+\beta|}N_{\a,\beta}, & \hbox{if $\a+\beta$ is a root;} \\
0, & \hbox{if $\a+\beta$ is not a root;}
                                    \end{array}
                                  \right.
$$
and
$$
n_{\a,-\beta}= \left\{
                                    \begin{array}{ll}
                                      \frac {|\a||\beta|}{  \sqrt2|\a-\beta|}N_{\a,-\beta}, & \hbox{if $\a-\beta$ is a positive  root;}\\
  -\frac {|\a||\beta|}{  \sqrt2|\a-\beta|}N_{\a,-\beta}, & \hbox{if $\a-\beta$ is a negative  root;}\\
0, & \hbox{if $\a-\beta$ is not a root.}
                                    \end{array}
                                  \right.
$$
Hence $\{E_\a\}$ form a Weyl canonical basis.
\end{lma}
\begin{proof}
It is easy to see that $K(E_\a,E_{-\a})=-1$. If $\a+\beta$ is a root, then
\bee
\begin{split}
[E_\a,E_{\beta }]=&\frac {|\a||\beta|}{  2}[X_\a,X_\beta]\\=&N_{\a,\beta}\frac {|\a||\beta|}{  2} X_{\a+\beta}\\=&\frac {|\a||\beta|}{  \sqrt2|\a+\beta|}N_{\a,\beta}E_{\a+\beta}=n_{\a,\beta}E_{\a+\beta}
\end{split}
\eee
\bee
\begin{split}
[E_{-\a},E_{-\beta }]=&\frac {|\a||\beta|}{  2}[X_{-\a},X_{-\beta}]\\
=&N_{-\a,-\beta}\frac {|\a||\beta|}{  2} X_{-\a-\beta}\\
=&N_{\a,\beta}\frac {|\a||\beta|}{ \sqrt 2|\a+\beta|} E_{-\a-\beta}=n_{\a,\beta}E_{-\a-\beta}
\end{split}
\eee
where $n_{\a,\beta}=\frac {|\a||\beta|}{  \sqrt2|\a+\beta|}N_{\a,\beta}$. Here we have used the fact that $N_{\a,\beta}=-N_{-\a,-\beta}$ and
$X_{-\a-\beta}=-\frac{\sqrt 2}{|\a+\beta|}  E_{-\a-\beta}$. If $\a+\beta$ is not a root, then $[E_,a,E_\beta]=0$.

If $\a-\beta\neq 0$ and is a positive root, then

\bee
[E_\a,E_{-\beta }]=\frac {|\a||\beta|}{  2}[X_\a,X_{-\beta}]=N_{\a,-\beta}\frac {|\a||\beta|}{  2} X_{\a-\beta}=\frac {|\a||\beta|}{  \sqrt2|\a-\beta|}N_{\a,-\beta}E_{\a-\beta}.
\eee
 If $\a-\beta\neq 0$ and is a negative root, then
\bee
[E_\a,E_{-\beta }]=\frac {|\a||\beta|}{  2}[X_\a,X_{-\beta}]=N_{\a,-\beta}\frac {|\a||\beta|}{  2} X_{\a-\beta}=-\frac {|\a||\beta|}{  \sqrt2|\a-\beta|}N_{\a,-\beta}E_{\a-\beta}.
\eee
\end{proof}

Now let $\eta\in \Delta^+$, and consider the \K $C$-space $(\mathfrak{g},\eta)$ with corresponding  Weyl frame (unitary frame for $(\mathfrak{g},\eta)$)  $e_\a=\frac1{\sqrt k}E_\a$ for $\a\in \Delta^+(k)$.  For any positive roots $\a, \beta$, define
\be\label{wtN}
\left\{
  \begin{array}{ll}
    \wt N_{\a,\beta}= \frac{|\a||\beta|}{|\a+\beta|}N_{\a,\beta};& \\
\wt N_{\a,-\beta}= \frac{|\a||\beta|}{|\a-\beta|}N_{\a,-\beta}, &\hbox{\rm if $\a-\beta\neq0$}.
 \end{array}
\right.
\ee

We can now combine Lemma \ref{l-Weylframe} and Proposition \ref{prop-LiWuZheng} with Lemma \ref{l-Weyl basis}, \eqref{liebracketeuclid} and \eqref{killingeuclid} to obtain the following from \cite[Proposition 2.4]{Itoh}:
 \begin{lma}\label{l-curvature-formula-2} Let $\a\in \Delta^+(i)$, $\beta\in \Delta^+(j)$, with $i\le j$ and let $R_{\a\ba\beta\b\beta}=R(e_\a,\b e_\a,e_\beta,\b e_\beta)$. Then
\be
R_{\a\ba\beta\b\beta}=\frac1j\lf((\a,\beta)+\frac12\frac{i}{i+j}\wt N_{\a,\beta}^2\ri).
\ee
\end{lma}
Next let us consider $R(\a,\b\beta,\gamma,\b\delta)=R(e_\a,\b e_\beta,e_\gamma,\b e_\delta)$ with $\a-\beta,\gamma-\delta\neq0$.

\begin{lma}\label{l-curvature-formula-3}
Let $e_\a \in \Delta^+(i)$, $e_\beta \in \Delta^+(j)$, $e_\gamma \in \Delta^+(k)$ and $e_\delta\in \Delta^+(l)$.
\begin{enumerate}
  \item If $\a-\beta\neq \delta-\gamma$, then $R(\a,\b\beta,\gamma,\b\delta)=0.$
  \item If $\a-\beta= \delta-\gamma\neq 0$, then
\be
\begin{split}
R(\a,\b\beta,\gamma,\b\delta)=&
 -\frac{1}{2\sqrt{ijkl}}\lf[ \lf((k-j)\xi_{k-j}-\frac{kl}{i+k}\ri)\wt N_{\a,\gamma}\wt N_{ \beta, \delta}\ri] \\
 &+   \frac{1}{2\sqrt{ijkl}}\lf[\lf(- (k-j)\xi_{k-j}+k\xi_{i-j}+l\xi_{j-i}+l\delta_{ij}\delta_{kl}\ri)\wt N_{\a,-\beta}\wt N_{\gamma,-\delta}.                                \ri]\\
=:&R_1(\a,\b\beta,\gamma,\b\delta)+R_2(\a,\b\beta,\gamma,\b\delta).\end{split}
\ee
\end{enumerate}
\end{lma}
\begin{proof}
(1) follows from Lemma \ref{l-Weylframe}, and the fact that $K(E_\a,E_\beta)=0$ unless $\a+\beta=0$.

(2) Note that $\b e_\a=-e_{-\a}$, etc. First assume that $\a+\gamma$ and $\a-\beta$ are both roots. By Lemma \ref{l-Weyl basis}
\bee
\begin{split}
[e_\a,e_\gamma]=&\frac1{\sqrt{ik}}[E_\a,E_\gamma]\\
=&n_{\a,\gamma}\frac1{\sqrt{ik}}E_{\a+\gamma}\\
=&N_{\a,\gamma}\frac{|\a||\gamma|}{\sqrt{2ik}}E_{\a+\gamma}.
\end{split}
\eee
Similarly,
\bee
\begin{split}
[e_{-\beta} ,e_{-\delta} ]
=&N_{ \beta, \delta}\frac{|\gamma||\delta|}{\sqrt{2jl}}E_{-\beta-\delta},
\end{split}
\eee

We may assume that $\a-\beta$ is a positive root, then $\gamma-\delta$ is a negative root.
\bee
\begin{split}
[e_{\a} ,e_{-\beta} ]=&N_{ \a,-\beta}\frac{|\a||\beta|}{\sqrt{2ij}}E_{\a-\beta},
\end{split}
\eee
\bee
\begin{split}
[e_{\gamma} ,e_{-\delta} ]=&-N_{ \gamma,-\delta}\frac{|\gamma||\delta|}{\sqrt{2kl}}E_{\gamma-\delta}.
\end{split}
\eee
Since $\a+\gamma=\beta+\delta$ and $K(E_\sigma,E_{-\sigma})=-1$, we see that (2) is true by Proposition \ref{prop-LiWuZheng}. The cases that $\a+\gamma$ or $\a-\beta$ is not a root can be proved similarly.

\end{proof}

\subsection{The condition $QB \geq 0$}\label{section-quadraticbisectional}
We  first discuss the condition $QB \geq 0$ on a \K manifold $(M,\omega)$ with \K form $\omega$.  We will also consider the condition $QB > 0$ at $p$ which we define as:   $QB \geq 0$ at $p$ with strict inequality in \eqref{QBgeq0} provided not all $\xi_i 's$ are the same.  Now define the following bilinear forms on the space $\Omega^{1,1}_{\R} (M)$ of real $(1,1)$ forms on $M$:
$$
F(\eta,\sigma)=\sum_{i,j,k,l}R_{i\b jk\b l}\rho^{i\b l} \sigma^{k\b j} =\sum_{i,j,k,l}R_{i\b l k\b j}\rho^{i\b l} \sigma^{k\b j} $$

$$
G(\eta,\sigma)=\frac12(R_{i\b j}g_{k\b l}+R_{k\b l}g_{i\b j} )\rho^{i\b l} \sigma^{k\b j}.$$
where $\rho^{i\b l},  \sigma^{k\b j}$ are the local components of $\rho, \sigma$ with indices raised.  Clearly, $G$ and $F$ are well defined real symmetric  bilinear forms on $\Omega^{1,1}_{\R} (p)$ for any $p$.  Now let $\theta_A$ be a unitary frame at any $p$ with co-frame $\eta_A$ and let $a_A$ be real numbers.  Take $X=\sum_{A}\ii a_A \eta_A\wedge \overline{\eta_A} \Omega^{1,1}_{\R} (p)$.  Then a simple calculation gives
\begin{equation}\label{bochnerQB}
 G(X,X)-F(X,X)=  \sum_A R_{A\b A}a_A^2-\sum_{A,B}R_{A\b AB\b B}a_Aa_B = \frac12\sum_{A,B}R_{A\b AB\b B}(a_A-a_B)^2.
 \end{equation}

The following was observed by Yau \cite{Zheng}.

\begin{lma}\label{l-QB-quadractic-1}
At any point $p$ we have
\begin{enumerate}
\item [(a)] $QB\ge0$ if and only if  $G-F\ge 0$.
\item [(b)] $QB>0$ if and only if  $G-F> 0$ on $\Omega^{1,1}_{\R} (p) \setminus \R\omega(p)$.
\end{enumerate}
Here $\Omega^{1,1}_{\R} (p) \setminus \R\omega(p)$ is the real $(1,1)$ forms at $p$ which are not a multiple of the \K form.
\end{lma}

\begin{proof}  We first prove (a).  The fact that  $G-F\ge0$ implies $QB\ge0$ follows immediately from \eqref{bochnerQB} and the fact that $\theta_A$ and $a_A$ are arbitrary.   Conversely, suppose $QB\geq 0$ and let $X$ be any real $(1,1)$ form at $p$.  Then we can always diagonalize $X$.  Namely, there exists a unitary frame $e_A$ with co-frame $\eta_A$ such that $X= \sum_{A}\ii a_A \eta_A\wedge \overline{\eta_A}$.  Now \eqref{bochnerQB} and the assumption $QB \geq 0$ immediately implies $G(X,X)-F(X,X) \geq 0$.

Now we prove (b).  The proof is basically the same in part (a) once we observe that $X\in \R\omega(p)$  if and only if: for every  unitary frame $e_A$ at $p$ with co-frame $\eta_A$ we have $X=c\sum_{A}\ii \eta_A\wedge \overline{\eta_A}$ for some real constant $c$.  The fact that  $G-F>0$ on $\Omega^{1,1}_{\R} (p) \setminus \R\omega(p)$ implies $QB>0$ now follows immediately from \eqref{bochnerQB} and the fact that $\theta_A$ and $a_A$ are arbitrary.   Conversely, suppose $QB> 0$ and let $X\in \Omega^{1,1}_{\R} (p) \setminus \R\omega(p)$.  Then there exists a unitary frame $e_A$ with co-frame $\eta_A$ such that $X= \sum_{A}\ii a_A \eta_A\wedge \overline{\eta_A}$ with $a_A 's$ not all the same.  Now \eqref{bochnerQB} and the assumption $QB >0$ immediately implies $G(X,X)-F(X,X) > 0$.

 This concludes the proof of the Lemma.
 \end{proof}

\begin{rem}\label{QBperturb}
 Thus $QB>0$ if and only if $G-F$ is positive in the orthogonal complement of $\R\omega$. In particular, if $(M,g)$ is a compact \K manifold with $QB>0$ then a \K metric which is a small perturbation of $g$ will also satisfy $QB>0$.
\end{rem}

 \begin{rem} Viewed as an endomorphism on $\Omega^{1,1}_{\R} (M)$, $G-F$ is in fact the curvature term in the Weitzenb\"{o}ck identity for real $(1,1)$ forms: $\Delta_g- \Delta$ is given by $G-F$ up to a positive constant multiple where $\Delta_g$ is the Bochner Laplacian with respect to $g$ and $\Delta$ is the Laplace-Beltrami operator.  The standard Bochner technique and Lemma \ref{l-QB-quadractic-1}
 then gives: {\sl all real harmonic (1,1) forms on $M$ are parallel provided $QB\geq 0$}, moreover, {\sl $dim(H^{1,1}_{\R} (M))=1$ provided $QB>0$} where $H^{1,1}_{\R} (M)$ is the space of real harmonic (1,1) forms on $M$.  See \S 1 for a reference to these facts and their implicit appearance in earlier works.
\end{rem}

By Lemma \ref{l-QB-quadractic-1} to check whether $QB\ge0$ (or $QB>0$) it is sufficient to consider $G-F \geq 0$ in a unitary frame of our choice. In the case of \K $C$-spaces, the natural choice is a Weyl frame. By Lemmas \ref{l-Weylframe} and \ref{l-QB-quadractic-1}, we have:

\begin{cor}\label{c-QB}
 On a \K $C$-space, let $\Ric=\mu g $ and let $e_A$ be a Weyl frame.  Then $QB\ge 0$ if and only if the largest eigenvalues of the quadratic forms $\sum_{A,B}R_{A\b AB\b B}x_Ax_B$, with $x_A$'s real, and $\sum_{{A,B,C,D;}\atop{  A\neq B, C\neq D}}  R_{A\b B C\b D}x_{AB}x_{CD}$, with $\ol{x_{AB}}=x_{BA}$, are at most $\mu$. $QB>0$   if $QB\ge0$ and the eigenvalue $\mu$ of  $\sum_{A,B}R_{A\b AB\b B}x_Ax_B$ is simple and the largest eigenvalue of $\sum_{{A,B,C,D;}\atop{  A\neq B, C\neq D}}  R_{A\b B C\b D}x_{AB}x_{CD}$ is less than $\mu$.
\end{cor}

 The following simple fact will be used throughout the paper to estimate the largest eigenvalue of a quadratic form.

\begin{lma}\label{l-eigenvalue}[row sums]  Let $ x_1,\dots, x_n$, $a_1,\dots,a_n$, and $\lambda$ be real or complex numbers. Suppose $|x_k|=\max\{|x_i|\ 1\le i\le n\}>0$ and
$$
\lambda x_k=\sum_{i=1}^n a_ix_i.
$$
Then
$$
|\lambda|\le \sum_{j=1}^n|a_i|.
$$
In particular, if   $\lambda$ is an eigenvalue of an $n\times n$ matrix $(a_{ij})$, then
$$
|\lambda|\le \max_i\lf(\sum_{j=1}^n|a_{ij}|\ri).
$$
\end{lma}

 We also note the following modification of Lemma \ref{l-eigenvalue} which will only be needed in a few exceptional cases.

\begin{lma}\label{weightedrowsumsZ}[weighted row sums] Let   $\lambda$ is an eigenvalue of an $n\times n$ matrix $A=(a_{ij})$ such that $|\lambda|>0$. Let $\mu>0$ be a positive number. Define $b^{s}_j$ inductively: $b^{(0)}_j=1$, and
$$
b^{(s+1)}_j=\min(1,\sum_{l}|a_{jl}|\frac{b^{(s)}_l} \mu).
$$

Then for all $s\ge0$,
\be\label{2}
|\lambda|\le \max\{\max_i\lf(\sum_{j=1}^n|a_{ij}|b_j^{(s)}\ri), \mu\}.
\ee
In particular, if for some $s\ge0$,
\be\label{3}
\max_i\lf(\sum_{j=1}^n|a_{ij}|b_j^{(s)}\ri)<\mu,
\ee
then $\lambda<\mu$.
\end{lma}
\begin{proof}
First we show that $b^{(s+1)}_j\le b^{(s)}_j$ for all $j$. Note that by definition $1\ge b^{(s)}_j\ge0$. It is obviously true that  $b^{(1)}_j\le 1= b^{(0)}_j$. Suppose $b^{(s+1)}_j\le b^{(s)}_j$ for all $j$, then
$$
b^{(s+2)}_j=\min(1,\sum_{l}|a_{jl}|b^{(s+1)}_l/\mu)\le \min(1,\sum_{l}|a_{jl}|b^{(s)}_l/\mu) =b^{(s+1)}_j.
$$

To prove the lemma. If $\lambda\le \mu$, then the lemma is true. Suppose $\lambda>\mu$.
Let $x_i$ be the components of an eigenvector of $A$ with eigenvalue $\lambda$.  Suppose, without loss of generality, that $\max_i|x_i|=1$. We claim that for all $s\ge0$,
$$
|x_i|\le b^{(s)}_i
$$
for all $i\ge1$. For $s=1$, then for any $j$
\bee
\lambda x_j=\sum_{l}a_{jl}x_l.
\eee
and
\bee
|x_j|\le \frac{1}{|\lambda|}\sum_{l}|a_{jl}x_l|\le \frac{1}{|\lambda|}\sum_{l}|a_{jl}|.
\eee
So $|x_j|\le b^{(1)}_j$ because $\lambda>\mu$ and $|x_j|\le 1$. Now suppose $|x_j|\le b^{(s)}_j$ for all $j$. Then as before,

 \bee
|\lambda| |x_j|\le \sum_{l}|a_{jl}| |x_l|\le \sum_{l}|a_{jl}| b^{(s)}_l,
\eee
and
\bee
|x_j|\le \sum_{l}|a_{jl}|\frac{ b^{(s)}_l}\mu.
\eee
Hence $|x_j|\le b^{(s+1)}_j$. Hence the claim is true.

 Now we may assume without loss of generality that $|x_1|=1$. \eqref{2} is true for $s=0$ by the previous lemma. For $s\ge1$,
 \bee
 |\lambda|=|\lambda x_1|\le \sum_{l}|a_{1l}|x_l|\le \sum_{l}|a_{1l}b^{(s)}_l.
 \eee
 Hence \eqref{2} is also true in this case.

 and $|x_i|\leq 1$ for all $i$.  Then we have

\be\label{1}
|\lambda|=|\lambda  x_1|\le  \sum_{j}|a_{1j}| |x_j|.
\ee

If \eqref{3} is true, then it is still true if $\mu$ is replaced by $\mu-\e$ for $\e>0$ small enough. Then by \eqref{2}, $|\lambda|\le \mu-\e<\mu$.

\end{proof}

\section{\K $C$-spaces of classical type}\label{section-classical}

According to \cite{Itoh}, the \K $C$-spaces with $b_2=1$ of classical type which are not Hermitian symmetric spaces are $(B_n,\a_p)$, with $n\ge 3$, $1<p<n$, $(C_n,\a_p)$, with $n\ge 3$, $1<p<n$, and $(D_n,\a_p)$, with $n\ge 4$, $1<p<n-1$.  For each Lie algebra $B_n, C_n, D_n$ below, we assume an identification has been made between $\mathfrak{h}^*$, the dual Cartan subalgebra, and $V=\R^n$ so that the induced Killing form corresponds to the Euclidean inner product $(\cdot, \cdot)$.  We will then present the corresponding root system $\Delta$ as a set of vectors in $V=\R^n$.  We refer to \cite{Bourbaki} for details.

\subsection{The spaces $(B_n,\a_p)$}\label{section-B}

We first consider the space $(B_n,\a_p)$, with $n\ge 3$, $1<p<n$. Let $V=\R^n$ and let $\ve_i$ be the standard basis on $V$.  The root system for $B_n$ is:

\be\label{e-root-so2n-1}
\Delta=\{\pm \ve_i\pm \ve_j| 1\le i,j\le n,i\neq j  \}\cup\{\pm \ve_i|\ 1\le i\le n\}
\ee
Simple positive roots are:

\be
\label{e-root-B-2}
\a_1=\ve_1-\ve_2, \a_2=\ve_2-\ve_3,\dots, \a_{n-1}=\ve_{n-1}-\ve_n, \a_n=  \ve_n.
\ee
Positive roots are:
\be\label{e-root-B-1}
\Delta^+=\{ \ve_i+ \ve_j\}_{i<j}\cup\{\ve_i-\ve_j\}_{i<j}\cup\{\ve_i\}.
\ee
In terms of the $\a_i$'s the positive roots are
\be\label{e-B}
\begin{split}
\ve_i=&\a_i+\dots+\a_n\\
\ve_i+\ve_j=&\a_i+\dots+\a_{j-1}+2\a_j+\dots+2\a_n, \ i<j\\
\ve_i-\ve_j=&\a_i+\dots+\a_{j-1}, \ i<j.
\end{split}
\ee

Let $1<p<n$. Recall that
$$
\Delta_p^+(k)=\{\a\in \Delta^+|\ \a=k\a_p+\sum_{i\neq p}m_i\a_i, m_i\ge0, m_i\in Z\}.
$$
By \eqref{e-root-B-1} and \eqref{e-B}, we have

\be
\label{e-root-B-3}
\begin{split}
\Delta^+_p(1)
=& \{\ve_a|\ 1\le a\le p \}   \bigcup\{\ve_a+\ve_i| 1\le a\le p, p+1\le i\le  n\}\\
 & \bigcup\{\ve_a-\ve_i| 1\le a\le p, p+1\le i\le n   \},
\end{split}
\ee

 \be
\label{e-root-B-5}
\begin{split}
\Delta^+_p(2)=&\{\ve_a+\ve_b| 1\le a<b\le p\}.
\end{split}
\ee
 \be
\label{e-root-6}
\Delta^+_p(k)=\emptyset,
\ee
for $k\ge 3$.   The dimension of $(B_m,\a_p)$ is $\frac12p(4n-3p+1)$.  We denote the elements of the $\Delta^+_p(k)$'s by:  $X_{a i}=\ve_a-\ve_i$, $Y_{a i}=\ve_a+\ve_i$, $1\le a\le p$, $p+1\le i\le n$; $U_a=\ve_a$, $W_{a b}=\ve_a+\ve_b$, $1\le a, b\le p$. In the following $a, b,\dots$ will range from $1$ to $p$ and $i, j,\dots$ will range from $p+1$ to $n$. Thus
\bee
\begin{split}
\Delta^+_p(1)=&\{X_{ai}\}_{1\le a\le p;p+1\le i\le n}\bigcup \{Y_{a i}\}_{1\le a\le p;p+1\le i\le n}\bigcup\{U_a\}_{1\le a\le p};
\\
\Delta^+_p(2)=&\{W_{a b}\}_{1\le a<b\le p}.
\end{split}
\eee

Now recall that $N_{\a,\beta}=\pm (p+1)$ were $p$ is the largest integer so that $\beta-p\a$ is a root, and also the definition of $\wt N_{\a,\beta}$ in \eqref{wtN}.

\begin{lma}\label{l-N-B} Let $\a, \beta$ be positive roots in $(B_n,\a_p)$,  then $\wt N_{\a,\beta}=\sqrt2\,\sgn(N_{\a,\beta})$. If $\a-\beta\neq0$, then $\wt N_{\a,-\beta}=\sqrt2\,\sgn(N_{\a,-\beta})$.

\end{lma}
\begin{proof}
Note that if $\sigma$ is a root, then either $|\sigma|^2=1$ or $|\sigma|^2=2$.  We begin by proving the first part of the Lemma.  Let  $\a, \beta$ be positive roots.  We may assume $\a+\beta$ is a root, otherwise the first part of the lemma is obviously true.

Suppose $|\a|^2=|\beta|^2=1$ and suppose $|\a+\beta|^2=1$, then $(\a,\beta)=-\frac12$, and this is impossible because one can see that $(\a,\beta)$ is an integer. Hence $|\a+ \beta|^2=2$ and $(\a,\beta)=0$. So  $\a-\beta$ is also a root \cite[p.324]{Fulton-Harris}. $ \a-2\beta$ is not a root because $|\a-2\beta|^2=5$.  Hence $N_{\a,\beta}=\pm2$. Therefore, by the definition of $\wt N_{\a,\beta}$ in \eqref{wtN}, $\wt N_{\a,\beta}=\sqrt 2\,\sgn(N_{\a,\beta}).$

Suppose $|\a|^2=1$, and $|\beta|^2=2$. As before, one can prove that $(\a,\beta)=-1$ and $N_{\a,\beta}=\pm1$. Hence  $\wt N_{\a,\beta}=\sqrt 2\,\sgn(N_{\a,\beta})$.

Suppose $|\a|^2=|\beta|^2=2$.  As before, one can prove that $(\a,\beta)=-1$, $N_{\a,\beta}=\pm1$  and hence  $\wt N_{\a,\beta}=\sqrt 2\,\sgn(N_{\a,\beta})$.

The case for $\wt N_{\a,-\beta}$ can be proved similarly.
\end{proof}

By Lemmas \ref{l-curvature-formula-2}, \ref{l-curvature-formula-3} and \ref{l-N-B} and  the fact that $R(\a,\b\beta,\gamma,\b\delta)=R(\a,\b\delta,\gamma,\b\beta)$ we have:
\begin{cor}\label{cor-curvature-classical-1}

Let $\a \in \Delta^+(i)$, $\beta \in \Delta^+(j)$, $\gamma \in \Delta^+(k)$ and $\delta\in \Delta^+(l)$.

\begin{enumerate}
\item
\bee
R_{\a\ba\beta\b\beta}=\left\{
                        \begin{array}{ll}
                           (\a,\beta)+\frac 12(\sgn(N_{\a,\beta})^2), & \hbox{$i=j=1$;} \\
  \frac12 (\a,\beta)                        , & \hbox{$i=1, j=2$;}\\
\frac1 2 (\a,\beta), & \hbox{$i=j=2$;}.
                        \end{array}
                      \right.
\eee

  \item If $\a-\beta\neq \delta-\gamma$ then $R(\a,\b\beta,\gamma,\b\delta)=0.$
  \item   If $\a-\beta=\delta-\gamma\neq0$, then for $(i,j,k,l)=(1,1,1,1)$,
  \bee
R(\a,\b\beta,\gamma,\b\delta)=
\left\{
\begin{array}{ll}
\frac12
\sgn(N_{\a,\gamma})\sgn(N_{\beta, \delta}), & \hbox{if $\a-\beta$ is not a root;} \\
\sgn(N_{\a,-\beta })\sgn(N_{ \gamma, -\delta}), & \hbox{if $\a+\gamma$ is not a root; } \\
-\frac12
\sgn(N_{\a,\gamma})\sgn(N_{ \delta, \beta}), & \hbox{if $\beta-\gamma\neq0$   is not a root. }\end{array}
 \right.
\eee
For other cases,
 \bee
R(\a,\b\beta,\gamma,\b\delta)=
\left\{
\begin{array}{ll}
  \frac12
\sgn(N_{\a,-\beta })\sgn(N_{ \gamma, -\delta}), & \hbox{if $(i,j,k,l)=(1,1,2,2)$; } \\
\frac12
\sgn(N_{\a,-\beta })\sgn(N_{ \gamma, -\delta}), & \hbox{if $(i,j,k,l)=(2,2,2,2)$; }\\
 \frac12\sgn(N_{\a,-\beta })\sgn(N_{ \gamma, -\delta})& \hbox{if $(i,j,k,l)=(1,2,2,1)$; }\end{array}
 \right.
\eee

\end{enumerate}
\end{cor}

 To compute the Ricci curvature, we know that $\Ric=\mu g$ and thus
\bee
\begin{split}
\mu=&\Ric(W_{12}, \b W_{12})\\
=&\sum_{a,i}\lf[R(W_{12}, \b W_{12} , X_{ai},\b X_{ai})+R(W_{12}, \b W_{12}, Y_{ai},\b Y_{ai})\ri]\\
&+\sum_a R(W_{12}, \b W_{12}, U_{a},\b U_{a})+\sum_{a<b}R(W_{12}, \b W_{12}, W_{ab},\b W_{ab})\\
=&\frac12\lf[2(n-p)+2(n-p)\ri]+1+\frac12\lf(p+ (p-2)\ri)\\
=&2n-p.
\end{split}
\eee

\begin{lma}\label{l-B-eigenvalue-1} Let $\lambda$ be the largest eigenvalue of the quadratic form $$\sum_{A,B }R_{A\b AB\b B}x_Ax_B$$  in the Weyl frame, where $x_A$ are reals.
\begin{itemize}
  \item [(a)] $\lambda\le 2n-p$ if and only if $5p+1\le 4n$.
  \item [(b)] If $5p+1< 4n$, then $\lambda=(2n-p)$ iff the corresponding eigenvector satisfies $x_A=x_B$ for all $A,B$.
  \item [(c)] If $5p+1=4n$, then there is an eigenvector with eigenvalue $(2n-p)$ such that $x_A\neq x_B$ for some $A\neq B$.
\end{itemize}
\end{lma}
\begin{proof}  We begin with the proof of $(a)$.  Let $v=(x_A)$ be an eigenvector corresponding to the largest eigenvalue $\lambda$ for the quadratic form.  Assume the components satisfy $max_A |x_A|=1$.  Let us denote the components  $x_A$ more specifically by $x_{ai}, y_{ai}, a\le p<i; u_a, a\le p; t_{ab}, a<b\le p$, and let us denote $R(X_{ai},\b X_{ai},  X_{bj}, \b X_{bj})$ by $R(X_{ai},X_{bj})$ etc.  Then  $P(v)=\sum_{A,B }R_{A\b AB\b B}x_Ax_B$ is equal to:

\be\label{l-B-eigenvalue-1-proof-e1}
\begin{split}
P(v)=&\sum_{a,b\le p<i, j}R(X_{ai},X_{bj})x_{ai}x_{bj}+\sum_{a,b\le p<i, j}R(Y_{ai},Y_{bj})y_{ai}y_{bj}\\
&+\sum_{a,b\le p<i, j}R(X_{ai},Y_{bj})x_{ai}y_{bj}+\sum_{a,b\le p<i, j}R(Y_{ai},X_{bj})y_{ai}x_{bj}\\
&+2\sum_{a,c\le p<i }R(X_{ai},U_c )x_{ai}u_c+2\sum_{a,c\le p<i }R(Y_{ai},U_c )y_{ai}u_c\\
&+2\sum_{a<b, c\le p<i }R(w_{ab},X_{ci} )t_{ab}x_{ci}+2\sum_{a<b, c\le p<i }R(w_{ab},Y_{ci} )t_{ab}y_{ci}\\
&+\sum_{a,b\le p}R(U_a,U_b)u_au_b+2\sum_{a<b, c\le p}R(w_{ab},U_c)t_{ab}u_c\\
&+\sum_{a<b\le p, c<d\le p}R(w_{ab},w_{cd})t_{ab}t_{cd}.
\end{split}
\ee

From Corollary \ref{cor-curvature-classical-1}, it is easy to see that $R(X_{ai},X_{bj})=R(Y_{ai},Y_{bj})$, $R(X_{ai},Y_{bj})$ $=R(Y_{ai},X_{bj})$, $R(X_{ai},U_{b})=R(Y_{ai},U_{b})$,
$R(X_{ai},W_{bc})=R(Y_{ai},W_{bc})$. We see that if we interchange $x_{ai}$ and $y_{ai}$, for all $a, i$ and obtain a vector $w$, then $P(v)=P(w)$ and $|v|=|w|$.  We may then assume that either $x_{ai}=y_{ai}$ for all $a,i$, or by considering $v-w$, that $x_{ai}=-y_{ai}$ and $u_a=t_{ab}=0 $ for all $a,b$.

Suppose $|u_a|=1$ for some $a$. We may assume that $u_a=1$.  By Corollary \ref{cor-curvature-classical-1}, $R(U_a,\b U_a, x,\b x)\ge 0$ because $(U_a,x)\ge0$ for all $x\in \Delta^+_p(k), k=1,2$.

 \be\label{e-urowB}
\begin{split}\lambda u_a=&\sum_{b\le p}R(U_a,U_b)u_b +\sum_{b\le p<i}R(X_{bi},U_a)x_{bi}+ \sum_{b\le  p<i}R(Y_{bi},U_a)y_{bi}\\
& +\sum_{c<d\le p}R(w_{cd},U_a) t_{cd}.
\end{split}
\ee
Notice that the coefficients are all non-negative and the sum is just $\Ric(U_a,\b U_a)=2n-p$. Hence $\lambda\le 2n-p$.
Moreover, if $\lambda=2n-p$ then   we must in fact have

\be\label{e-urowB-equal} x_{a,i}=y_{a,i}=u_b=t_{cd}=1
\ee
for all $a,b\le p<i$ and $c<d\le p$.

Since $(W_{ab},x)\ge0$ for all $x\in \Delta^+_p(k), k=1,2$. We have similar result.

Suppose $x_{ai}=1$ for some $a, i$.

{\bf Case 1} ($x_{bj}=y_{bj}$ for all $b,j.$):   As above, we have
\be
\begin{split}\lambda x_{ai} =&\sum_{b,j}R(X_{ai},X_{bj})x_{bj} +\sum_{b,j}R(X_{ai}, Y_{bj})y_{bj} + \sum_{b}R(X_{ai},U_b)u_b\\
& + \sum_{c<d}R(X_{ai},w_{cd}) w_{cd}.
\end{split}
\ee
Since $x_{bj}=y_{bj}$, this equation is the same as:
\be
\begin{split}\lambda x_{ai} =&\sum_{b,j}\frac12(R(X_{ai},X_{bj})+R(X_{ai}, Y_{bj})) x_{bj} +\sum_{b,j}\frac12(R(X_{ai},X_{bj})\\
&+R(X_{ai}, Y_{bj}))y_{bj}
 + \sum_{b}R(X_{ai},U_b)u_b + \sum_{c<d}R(X_{ai},w_{cd}) w_{cd}.
\end{split}
\ee
By Corollary \ref{cor-curvature-classical-1}, $R(X_{ai},X_{bj})+R(X_{ai}, Y_{bj})\ge0$ since $(X_{ai},X_{bj}+Y_{bj})=2\delta_{ab}\ge0$. Hence the coefficients are all nonnegative.  Also, the sum of the coefficients is still the Ricci curvature $2n-p$. Hence we have $\lambda\le 2n-p$ as before, and if equality holds then \eqref{e-urowB-equal} is true.

{\bf Case 2} ($x_{bj}=-y_{bj}$ and $u_c=w_{cd}=0 $ for all $b, c ,d ,j.$): Suppose  $x_{ai}=1$. Then
\be
\begin{split}\lambda x_{ai} =&\sum_{b,j}R(X_{ai},X_{bj})x_{bj} +\sum_{b,j}R(X_{ai}, Y_{bj})y_{bj}\\
=&\sum_{b,j}\lf(R(X_{ai},X_{bj})-R(X_{ai}, Y_{bj})\ri)x_{bj}
\end{split}
\ee
  By Corollary \ref{cor-curvature-classical-1}, $R(X_{ai},X_{bj})-R(X_{ai}, Y_{bj})\geq 0$ because $(X_{ai},X_{bj}-Y_{bj})=2\delta_{ij}$. Hence the coefficients are all non negative. The sum of the coefficients is:
\be
\sum_{b,j}\lf( (\delta_{ab}+\delta_{ij})-(\delta_{ab}-\delta_{ij}+\frac12\delta_{ij}(1-\delta_{ab}))\ri)
=p+\frac12(p+1)=\frac12(3p+1).
\ee
Here we have used the fact that $X_{ai}+X_{bj}$ is not a root, and $X_{ai}+Y_{bj}$ is a root if and only if $b\neq a$ and $j=i$.
Hence if $5p+1\le 4n$  then $\lambda\le 2n-p$.  Moreover, if  $5p+1< 4n$  then $\lambda< 2n-p$

Now suppose  $5p+1> 4n$. Let $v$ be such that $x_{ai}=-y_{ai}=1$, $u_a=w_{ab}=0$ for all $a, b$. Then
$$
P(v)=2\sum_{a, b\le p<i, j} (R(X_{ai},X_{bj})-R(X_{ai}, Y_{bj}))=p(n-p)(3p+1).
$$
On the other hand, $|v|^2=2p(n-p)$. Hence $P(v)>(2n-p)|v|^2$ because $5p+1> 4n$.

 The case that $y_{ai}=1$ for some $a,i$ is similar. This completes the proof of (a).

 To prove (b), suppose $5p+1< 4n$. Then $\lambda\le 2n-p$, and as $(2n-p)$ is always an eigenvalue we have $\lambda=2n-p$. Let $v$ be the corresponding eigenvector with components $x_{ai}, y_{ai},u_a, t_{cd}$.  Thus $P(v)=\lambda|v|^2$.   The above proof then shows if $x_{ai}=y_{ai}$ for all $a,i$, then \eqref{e-urowB-equal} must be true, while if $x_{ai}\neq y_{ai}$ for some $a,i$ then we must have $\lambda< 2n-p $ which is impossible by our assumption.
Hence (b) is true.

To prove (c), suppose $5p+1= 4n$.  Then $\lambda=2n-p$ in this case too.  Let $v$ be such that $x_{ai}=-y_{ai}=1$, $u_a=w_{ab}=0$ for all $a, b$. Then the computations above give $P(v)=\lambda|v|^2$. Since $x_{ai}\neq y_{ai}$, (c) is true.

\end{proof}

\begin{lma}\label{l-B-eigenvalue-2}    Let $\lambda$ be the largest eigenvalue of the quadratic form $$\sum_{A,B,C,D; A\neq B, C\neq D}R_{A\b B C\b D}x_{A  B}x_{C  D}$$ in the Weyl frame, where $x_{A  B}=\ol{x_{B  A}} $.
\begin{itemize}
    \item [(a)] If   $5p+1\le 4n$, then
  $\lambda\le  2n-p $ .
    \item [(b)] If $5p+1< 4n$, then $\lambda< 2n-p $.
  \end{itemize}
\end{lma}
\begin{proof}
We want to estimate \be\label{SAB} S_{AB}=\sum_{x\neq y}|R_{A\b B y\b x}|\ee for each case of $A,B$.  Note that $S_{AB}=S_{BA}$. Recall the following properties of the curvature from Corollary \ref{cor-curvature-classical-1} which we repeat here for convenience of reference
\begin{enumerate}
\item[(C1)]  If $A-B \neq x-y$ then $R_{A\b B y\b x}=0$.
\item[(C2)] If neither $A-B$ nor $A+y$ are roots then  $R_{A\b B y\b x}=0$.
\end{enumerate}
In each case we will use these to reduce the terms in \eqref{SAB} as much as possible.  Then Corollary \ref{cor-curvature-classical-1} will be used to calculate the absolute values of the remaining curvature terms.

\vspace{12pt}

CASE (i) $A=X_{ai}, B=X_{bj}$ with $(a,i)\neq (b,j)$.

Note that $A,B\in \Delta^+_p(1)$.  By  (C1) we may assume that $x,y\in\Delta^+_p(1)$ or $x,y\in\Delta^+_p(2)$.  Note that the sum of the coordinates of $X$'s is 0, that the sum of the coordinates of $Y$'s is 2, that the sum of the coordinates of $U$'s is 1 and that the sum of the coordinates of $W$'s is 2.  Thus by  (C1),  \eqref{SAB}  reduces to:

\be\label{SABi} \begin{split}S_{AB}=&\sum_{c,k,d,l}|R(X_{ai}, \b X_{bj}, X_{c k},\b X_{d l})|+\sum_{c,k,d,l}|R(X_{ai}, \b X_{bj}, Y_{c k},\b Y_{d l})|\\&+\sum_{c,d}|R(X_{ai}, \b X_{bj}, U_c,\b U_d)|+\sum_{c,d,e,f}|R(X_{ai}, \b X_{bj}, W_{c d}, \b W_{e f})|\\=:&I+II+III+IV
\end{split}
\ee

If $a\neq b$, $i\neq j$ then: All terms in $III, IV$ are zero by (C1).  All terms in $I$ are zero by (C1) except for $|R(X_{ai}, \b X_{bj}, X_{bj},\b X_{ai})|$  which is zero by $(C2)$.  By (C1), the only non-zero term in $II$ is $|R(X_{ai}, \b X_{bj}, Y_{bi},\b Y_{aj})|=\frac12$.  Hence $S_{AB}=1/2<2n-p$ because $p<n$.
\vspace{12pt}

If $a=b$, $i\neq j$ then: All terms in $III, IV$ are zero by (C1).  By (C1), the only non-zero terms in $I$ are $|R(X_{ai}, \b X_{aj}, X_{cj},\b X_{ci})|$ for any $c$, leaving $I=p$.  By (C1), the only non-zero terms in $II$ are $|R(X_{ai}, \b X_{aj}, Y_{cj},\b Y_{ci})|$ for any $c$, giving a contribution of $1$ from the case $c=a$ and  $1/2(p-1)$ from the cases  $c\neq a$.   Hence $S_{AB}=p+1+\frac{1}{2}(p-1)=3p/2 + 1/2 \leq 2n-p$ if and only if $5p+1\le 4n$, and $S_{AB}< 2n-p$ if and only if $5p+1< 4n$.
\vspace{12pt}

If $a\neq b$, $i= j$, we may assume that $a<b$, then:  By (C1), the only non-zero terms in $I$ are $|R(X_{ai}, \b X_{bi}, X_{bk},\b X_{ak})|$ for any $k$, leaving $I=n-p$.   By (C1), the only non-zero terms in $II$ are $|R(X_{ai}, \b X_{bi}, Y_{bk},\b Y_{ak})|$ leaving a contribution of $1/2$ when $k=i$ and a contribution of $n-p-1$ for the cases when $k\neq i$.  By (C1), the only non-zero term in $III$ is $|R(X_{ai}, \b X_{bi}, U_b,\b U_a)|$ leaving $III=1$. By (C1), the only non-zero terms in $IV$ are $|R(X_{ai}, \b X_{bi}, W_{bc},\b W_{ac})|$ for $c>b$,  or $|R(X_{ai}, \b X_{bi}, W_{cb},\b W_{ca})|$ for $c<a$, or $|R(X_{ai}, \b X_{bi}, W_{bc},\b W_{ac})|$ for $a<c<b$ in which cases the respective contributions to $IV$ are $1/2(p-b), 1/2(a-1), 1/2(b-a-1)$ respectively.   Hence $S_{AB}=(n-p)+\frac12+(n-p-1)+1+ \frac12(p-b) +\frac12(a-1)+ \frac12(b-a-1)=2n-\frac32p-\frac12<2n-p$
\vspace{12pt}

CASE (ii) $A=X_{ai}, B=Y_{bj}$.

Note that $A,B\in \Delta^+_p(1)$.  By (C1), $x,y\in\Delta^+_p(1)$ or $x,y\in\Delta^+_p(2)$ and \eqref{SAB}  reduces to:

\be\label{SABii} \begin{split}S_{AB}=&\sum_{c,k,d,l}|R(X_{ai}, \b Y_{bj}, Y_{c k},\b X_{d l})|
\end{split}
\ee

If $a\neq b$, $i\neq j$:  then by (C1), the only non-zero term in \eqref{SABii} is given by

$|R(X_{ai}, \b Y_{bj}, Y_{bi},\b X_{aj})|=\frac12$.  Thus
$S_{AB}=\frac12<2n-p$.
\vspace{12pt}

If $a=b$, $i\neq j$: then by (C1), the only non-zero terms in \eqref{SABii} are $|R(X_{ai}, \b Y_{aj}, Y_{cj},\b X_{ci})|$ or $|R(X_{ai}, \b Y_{aj}, Y_{ci},\b X_{cj})|$,  for any $c$.  In the first case the contribution to \eqref{SABii} is $p$ and in the second case the contribution to \eqref{SABii} is $1$ when $c=a$ and $\frac12(p-1)$ from the cases $c\neq a$.  Thus  $S_{AB}=p+1+\frac12 (p-1)=\frac32 p +\frac12 \leq 2n-p$ if and only if $5p+1\le 4n$, and $S_{AB}< 2n-p$ if and only if $5p+1< 4n$.
\vspace{12pt}

If $a\neq b$, $i= j$ then by (C1), the only non-zero term in \eqref{SABii} is given by $|R(X_{ai}, \b Y_{bi}, Y_{bi},\b X_{bi})|=\frac12$.  Thus $S_{AB}=\frac12<2n-p$.
\vspace{12pt}

If $a=b$, $i= j$ then by (C1), the only non-zero terms in \eqref{SABii} are $|R(X_{ai}, \b Y_{ai}, Y_{ci},\b X_{ci})|$ for any $c$.   Thus , $S_{AB}=\frac12(p-1)<2n-p$.
\vspace{12pt}

CASE (iii) $A=X_{ai}, B=U_b$.

Note that $A,B\in \Delta^+_p(1)$.  By  (C1), $x,y\in\Delta^+_p(1)$ or $x,y\in\Delta^+_p(2)$ and \eqref{SAB}  reduces to:

\be\label{SABiii} \begin{split}S_{AB}=& \sum_{c,d,j}|R(X_{a i}, \b U_b, U_c,\b X_{dj})| + \sum_{c,d,j}|R(X_{a i}, \b U_b, Y_{cj},\b U_d|)=:I+II
\end{split}
\ee

If $a\neq b$ then:  All terms in $I$ are zero by (C2).   By (C1), the only non-zero term in $II$ is $|R(X_{a i}, \b U_b, Y_{bi},\b U_a|$  leaving $II=1$ .  Thus $S_{AB}=1 <2n-p$.
\vspace{12pt}

If $a=b$ then:  By (C1), the only non-zero terms in $I$ are $|R(X_{a i}, \b U_a, U_c,\b X_{ci})|$ for any $c$, leaving $I=p$.  By (C1), the only non-zero terms in $II$ are given by $|R(X_{a i}, \b U_a, Y_{ci},\b U_c)|$ for any $c$, and the contribution to $II$ is $1$ when $c=a$ and is $\frac12(p-1)$ from the cases $c\neq a$.  Thus $S_{AB}=\frac32p+\frac12<2n-p$
\vspace{12pt}

CASE (iv) $A=X_{ai}, B=W_{bc}$.

Note that $A\in \Delta^+_p(1), B\in \Delta^+_p(2)$.  By  (C1), $x\in\Delta^+_p(1)$ and $y\in\Delta^+_p(2)$ and  \eqref{SAB}  reduces to:

\be\label{SABiv}
\begin{split}
S_{AB}=&  \sum_{d,j,e,f} |R(X_{a i}, \b W_{bc}, W_{ef},\b X_{dj})|\\
\end{split}
\ee

If $a=b$ then by (C1), the only non-zero terms in  \eqref{SABiv} are  given by $ |R(X_{a i}, \b W_{ac}, W_{dc},\b X_{di})|$ for $d<c$, or $ |R(X_{a i}, \b W_{ac}, W_{cf},\b X_{fi})|$ for $f>c$.  In the first case the contribution is $c-1$ and the contribution in the second case is $p-1+c$.  Thus $S_{AB}=p-1<2n-p$.
\vspace{12pt}

If $a\neq b$ then by (C1) and (C2) the only non-zero terms  in \eqref{SABiv} are when $a=c$ in which case we get as above, that  $S_{AB}=p-1<2n-p$.
\vspace{12pt}

CASE (v) $A=Y_{ai}, B=Y_{bj}$, $(a,i)\neq (b,j)$. Similar to (i)

CASE (vi) $A=Y_{ai}, B=U_{b}$. Similar to (iii).

CASE (vii) $A=Y_{ai}, B=W_{bc}$. Similar to (iii).

CASE (viii) $A=U_a$, $B=U_b$, $a<b$.

From (C1) it is not hard to see here that \eqref{SAB}  reduces to:

\bee
\begin{split}
S_{AB}=&  \sum_{c,d,k,l} |R(U_a,\b U_b, X_{ck},\b X_{dl})| +\sum_{c,d,k,l} |R(U_a,\b U_b, Y_{ck},\b Y_{dl})|\\
&+\sum_{c,d} |R(U_a,\b U_b, U_c,\b U_d)|+\sum_{c,d,e,f} |R(U_a,\b U_b, W_{cd},\b W_{ef})|\\
=&I+II+III+IV\\
\end{split}
\eee

Now by (C1) the only non-zero terms in $I$ are $|R(U_a,\b U_b, X_{bi},\b X_{ai})|$ for any $i$ leaving $I=n-p$.  We similarly get $II=n-p$.  By (C1), the only non-zero term in $III$ is  $|R(U_a,\b U_b, U_b,\b U_a)|$ leaving $III=\frac12$.  By (C1), the only non-zero terms in $IV$ are  $|R(U_a,\b U_b, W_{bc},\b W_{ac})|$ for $c>b$, or $|R(U_a,\b U_b, W_{cb},\b W_{ca})|$ for $c<a$, or $|R(U_a,\b U_b, W_{bc},\b W_{ca})|$ for $a<c<b$,  in which cases the respective contributions to $IV$ are $1/2(p-b), 1/2(a-1), 1/2(b-a-1)$ respectively.
Thus $S_{AB}=(n-p)+(n-p)+\frac12+\frac12\lf[(p-b)+(a-1)+(b-a-1)\ri]=2n-\frac32p-\frac12<2n-p$.
\vspace{12pt}

CASE (ix) $A=U_a$, $B=W_{bc}$, $b<c$.

Note that $A\in\Delta^+_p(1)$ and  $B\in \Delta^+_p(2)$.  By  (C1),  $x\in\Delta^+_p(1)$ and  $y\in \Delta^+_p(2)$  and  \eqref{SAB}  reduces to:

\be\label{SABix}
\begin{split}
S_{AB}=&  \sum_{d,e,f} |R(U_a,\b W_{bc}, W_{de},\b U_{f})|
\end{split}
\ee

Note that $A+x$ is never a root  and thus by (C2), the only non-zero terms are when $a=b$ or $a=c$.  If $a=b$ then by (C1) the only non-zero terms are $|R(U_a,\b W_{ac}, W_{dc},\b U_{d})|$ for $d<c$, or $|R(U_a,\b W_{ac}, W_{cd},\b U_{d})|$ for $c<d$, in which cases the respective contributions to $S_{AB}$ are $1/2(c-1), 1/2(p-c)$ respectively .  Thus when $a=b$, and similarly when $a=c$, we have $S_{AB}=\frac12(c-1)+\frac12(p-c)=\frac12(p-1)<2n-p$.
\vspace{12pt}

CASE (x) $A=W_{ab}, B=W_{cd}$, $(a,b)\neq (c,d)$. We may assume that $a\le c$.

Note that $A,B \in \Delta^+_p(2), B\in \Delta^+_p(2)$.  By  (C1), $x,y\in\Delta^+_p(1)$ or $x,y\in\Delta^+_p(2)$ and  \eqref{SABix}  reduces to:

\bee\label{SABx}
\begin{split}
S_{AB}=&  \sum_{e,f,i,j} (|R(W_{ab},\b W_{cd}, X_{ei},\b X_{fj})|+  \sum_{e,f,i,j} (|R(W_{ab},\b W_{cd}, Y_{ei},\b Y_{fj})|\\
&+\sum_{e,f} |R(W_{ab},\b W_{cd}, U_e ,\b U_{f})|+\sum_{e,f,g,h}|R(W_{ab},\b W_{cd}, W_{ef},\b W_{gh })|\\
 =&I+II+III+IV\\
\end{split}
\eee

Note that $A+y$ is never a root for any positive root $y$.  Thus by (C2), all terms in $I, II, III, IV$ are zero unless either $a=c$ or $b=c$ or $b=d$.  Assume that $a=c$, and without loss of generality that $b<d$.  By (C1), the only non-zero terms in $I$ are  $|R(W_{ab},\b W_{ad}, X_{di},\b X_{bi})|$ for any $i$,  leaving $I=\frac12 (n-p)$.   We similarly get $II=\frac12 (n-p)$. By (C1), the only non-zero term in $III$ is$ |R(W_{ab},\b W_{cd}, U_d ,\b U_{b})|$ leaving $III=\frac12$. By (C1), the only non-zero terms in $IV$ are $|R(W_{ab},\b W_{ad}, W_{de},\b W_{be })|$ for $e>d$, or $|R(W_{ab},\b W_{ad}, W_{ed},\b W_{ eb})|$ for $e<b$, or $|R(W_{ab},\b W_{ad}, W_{ed},\b W_{be })|$ for $b<e<d$, in which cases the respective contributions to $IV$ are $1/2(p-d), 1/2(b-1), 1/2(d-b-1)$. Thus when $a=c$, and similarly when  $b=c$ or $b=d$, we have
$S_{AB} =\frac12(n-p)+\frac12(n-p)+\frac12+\frac12\lf[(p-d)+(b-1)+(d-b-1)\ri]=n-\frac12 p-\frac12 <2n-p$
\vspace{12pt}

%=&\frac12(n-p)+\frac12(n-p)+\frac12+\frac12\lf[(p-d)+(b-1)+(d-b-1)
 %\ri]\\
 %=&n-\frac12 p-\frac12

 This completes the proof of the Lemma.
\end{proof}
\begin{thm}\label{t-B} The \K $C$-space $(B_n,\a_p)$, $n\ge 3$, $1<p<n$ satisfies $QB\ge0$ if and only if $5p+1\le 4n$.  Moreover, $QB>0$ if and only if $5p+1< 4n$.
 \end{thm}
 \begin{proof} The first statement follows from part (a) of Lemmas \ref{l-B-eigenvalue-1} and \ref{l-B-eigenvalue-2} and  Corollary \ref{c-QB}.   For the second statement, note that  $QB>0$ iff $G-F>0$  on $\Omega^{1,1}_{\R} (p) \setminus \R\omega(p)$ by Lemma \ref{l-QB-quadractic-1}. On the other hand, here $G=(2n-p)Id$ on $\Omega^{1,1}_{\R} (p) \setminus \R\omega(p)$, and thus by parts (b) and (c) of Lemma \ref{l-B-eigenvalue-1} and part (b) of Lemma \ref{l-B-eigenvalue-2} we have $G-F>0$ iff $5p+1< 4n$.  Thus we have  $QB>0$ if and only if $5p+1< 4n$.
\end{proof}

\subsection{The spaces $(D_n,\a_p)$}\label{section-D}

In this section we will consider the space $(D_n,\a_p)$, with $n\ge 4$, $1<p<n-1$.
Let $V=\R^n$ and $\ve_i$ be as before.  The root system for $D_n$ is:
\be\label{e-root-D-1}
\Delta=\{\pm \ve_i\pm \ve_j| 1\le i,j\le n,i\neq j\}.
\ee
Positive roots are:
\be\label{e-root-D-2}
\Delta^+=\{ \ve_i+ \ve_j\}_{i<j}\cup\{\ve_i-\ve_j\}_{i<j}.
\ee
Simple positive roots are:

\be
\label{e-root-D-3}
\a_1=\ve_1-\ve_2,\a_2=\ve_2-\ve_3,\dots,\a_{n-1}=\ve_{n-1}-\ve_n,\a_n=\ve_{n-1}+\ve_n.
\ee
In terms of the  $\a_i$'s the positive roots are
\be\label{e-D}
\begin{split}
\ve_i+\ve_j=&\a_i+\dots+\a_{j-1}+2\a_j+\dots+2\a_{n-2}+\a_{n-1}+\a_n, \ i<j\le n-2\\
\ve_i+\ve_{n-1}=&\a_i+\dots+\a_n, \ i<n-1\\
\ve_i+\ve_n=&\a_i+\dots+\a_{n-2}+\a_n, \ i<n-1\\
\ve_{n-1}+\ve_n=&\a_n\\
\ve_i-\ve_j=&\a_i+\dots+\a_{j-1}, \ i<j.
\end{split}
\ee

Let $1<p<n-1$.   By \eqref{e-root-D-2} and \eqref{e-D} we have

\be
\label{e-root-D-4}
\begin{split}
\Delta^+_p(1)=
& \{\ve_a-\ve_i| 1\le a\le p, p+1\le i\le n \}\bigcup\{\ve_a+\ve_i| 1\le a\le p, p+1\le i\le n\},
\end{split}
\ee

 \be
\label{e-root-D-5}
\begin{split}
\Delta^+_p(2)=&\{\ve_a+\ve_b| 1\le a<b \le p\}.
\end{split}
\ee
\be\Delta^+_p(k)=\emptyset\ee for $k\geq 0$.  The dimension is $\frac12p(4n-3p-1)$. The structure of the roots are similar to $(B_n,\a_p)$ except $U_a$'s do not appear. Hence the computations are the basically the same. In this case $\Ric=(2n-p-1)$.

 \begin{thm}\label{t-D} The \K $C$-space $(D_n,\a_p)$, $n\ge 4$, $1<p<n-1$ satisfies $QB\ge0$ if and only if $5p+3\le 4n$.  Moreover, $QB>0$ if and only if $5p+3< 4n$.
 \end{thm}

\begin{rem} As in the $B$   cases, one can see that $(D_n,\a_p)$ does not satisfy $B^\perp\ge0$.
\end{rem}

\subsection{The spaces $(C_n, \alpha_p)$}\label{section-C}

We will consider the space $(C_n,\a_p)$, with $n\ge 3$, $1<p<n$. Let $V=\R^n$ and $\ve_i$ be as before.  The root system for $C_n$ is:

  \be\label{e-C-root-1}
\Delta=\{\pm \ve_i\pm \ve_j| 1\le i,j\le n\}.
\ee
Positive roots are:
\be\label{e-root-C-1}
\Delta^+=\{ \ve_i+ \ve_j\}_{i\le j}\cup\{\ve_i-\ve_j\}_{i<j}.
\ee
Simple positive roots are:
\be
\label{e-root-C-2}
\a_1=\ve_1-\ve_2,\a_2=\ve_2-\ve_3,\dots,\a_{n-1}=\ve_{n-1}-\ve_n,\a_n=2\ve_n.
\ee
In terms of the  $\a_i$'s the positive roots are
\be\label{e-C}
\begin{split}
\ve_i-\ve_j=&\a_i+\dots+\a_{j-1}, \ i<j\le n,\\
2\ve_i=&2(\a_i+\dots+\a_{j-1}+ \a_j+\dots+ \a_{n-1})+\a_n, i< n, 2\ve_n=\a_n,\\
\ve_i+\ve_j=&\a_i+\dots+\a_{j-1}+2(\a_j+\dots+
+ \a_{n-1})+\a_n, \ i<j\le n.\\
\end{split}
\ee

Let $1< p < n$.  By \eqref{e-root-C-1} and \eqref{e-C} we have

\be
\label{e-root-4}
\begin{split}
\Delta^+_p(1)=&  \{\ve_a\pm \ve_i| 1\le a\le p, p+1\le i\le n \}
\end{split}
\ee

 \be
\label{e-C-root-5}
\begin{split}
\Delta^+_ p(2)=&\{2\ve_a |  1\le a\le   p\}\cup \{\ve_a +\ve_b|1\le a<b\leq p\}
\end{split}
\ee
 \be
\label{e-C-root-6}
\Delta^+_p(k)=\emptyset,
\ee
for $k\ge 3$. The dimension is $\frac12p(4n-3p+1)$.

As before, $a, b,\dots$ will range from 1 to $p$, $i,j,\dots$ range from $p+1$ to $n$. Let $X_{ai}=\ve_a-\ve_i$, $Y_{ai}=\ve_a+\ve_i$, $U_a=2\ve_a$, $W_{ab}=\ve_a+\ve_b$, $a<b$. Then as in Lemma \ref{l-N-B} and Corollary \ref{l-N-B}, we have the following:
\begin{lma}\label{l-curvature-classical-2}
 \begin{enumerate}
\item Let $\a,\beta$ be positive roots in $\Delta^+_p(k)$, $k=1, 2$.
  $$
  \wt N_{\a,\pm\beta}=
  \left\{
    \begin{array}{ll}
      2\,\sgn(N_{\a,\pm\beta}), & \hbox{if $\{\a,\beta\}=\{X_{ai}, Y_{ai}\}$ for some $a, i$, or}\\ &\hbox {\   one of $\a$, $\beta$ is $U_a$ for some $a$;}\\
      \sqrt 2\,\sgn(N_{\a,\pm\beta}), & \hbox{otherwise.}  \\
    \end{array}
  \right.
  $$
  Here in case of $\a-\beta$, we assume in addition that $\a-\beta\neq0$.
\item  $R(X_{ai}, \b X_{aj}, Y_{ai},\b Y_{aj})=0$  for any $a$ if $i\neq j$, and $R(X_{ai}, \b X_{ci}, Y_{ci},\b Y_{ai})=\pm \frac12$ for any $i$ if $a\neq c$.
      \end{enumerate}
\end{lma}
\begin{proof} (1)
 Suppose $\a=X_{ai}$, $\beta=Y_{ai}$, then $\a+\beta$ and $\a-\beta$ are both roots. It is easy to see that $N_{\a,\beta}$ and $N_{\a,-\beta}$ are equal to $\pm2$. Moreover, $|\a|^2=|\beta|^2=2$ and $|\a\pm\beta|^2=4.$ Hence
$ \frac{|\a||\beta|}{|\a\pm\beta|}N_{\a,\pm\beta}=2\sgn(N_{\a,\pm \beta})$. If $\a=U_a$, say, then $U_a+y$ and $U_a-U_b$ are not roots for any $y\in \Delta^+_p(k)$, $k=1,2$. Moreover, if $\beta= X_{bi}$, then $\a-\beta$ is a root if and only if $b=a$. In this case, $N_{\a,-\beta}=\pm 1$, $|\a|^2=4$, $|\beta|^2=2$ and $|\a-\beta|^2=2$. Again $ \frac{|\a||\beta|}{|\a\pm\beta|}N_{\a,\pm\beta}=2\sgn(N_{\a,\pm \beta})$.

If $\a,\beta$ are not as above, and if $\a+\beta$ is a root then $|\a+\beta|^2=2$. In this case, one can see that $N_{\a,\beta}=\pm 1$. So $ \frac{|\a||\beta|}{|\a\pm\beta|}N_{\a,\pm\beta}=\sqrt 2\sgn(N_{\a,\pm \beta})$. The case for $\a-\beta$ is similar.

(2) Let $\a=X_{ai}, \beta=X_{aj}, \gamma= Y_{ai}, \delta= Y_{aj}$.  It is easy to see that
$$R(X_{ai}, \b X_{aj}, Y_{ai},\b Y_{aj})=-\frac12\cdot\frac{|\a||\beta||\gamma||\delta|}{|\a+\gamma|^2}N_{\a,\gamma}N_{-\beta,-\delta}
+ \frac{|\a||\beta||\gamma||\delta|}{|\a-\beta|^2}N_{\a,-\beta}N_{\gamma,-\delta}.
$$
On the other hand, since $\a-\beta+\gamma-\delta=0$,
$$
\frac{N_{\gamma,\a}N_{-\beta,-\delta}}{|\a+\gamma|^2}+
\frac{N_{\a,-\beta}N_{\gamma,-\delta}}{|\a-\beta|^2}
+\frac{N_{-\beta,\gamma}N_{\a,-\delta}}{|\gamma-\beta|^2}=0
$$
By (1), we have
$$
\frac{N_{\gamma,\a}N_{-\beta,-\delta}}4+\frac{N_{\a,-\beta}N_{\gamma,-\delta}}2
=\pm \frac12
$$
because $i\neq j$. Squaring the above equality, noting that $ N_{\gamma,\a}^2=N_{-\beta,-\delta}^2=4$, $ N_{\a,-\beta}^2=N_{\gamma,-\delta}^2=1$, we have
$$
N_{\gamma,\a}N_{-\beta,-\delta}N_{\a,-\beta}N_{\gamma,-\delta}=-4.
$$
Hence
$$
N_{\a,\gamma}N_{-\beta,-\delta}N_{\a,-\beta}N_{\gamma,-\delta}>0
$$
because $N_{\a,\gamma}=-N_{\gamma,\a}$. From this it is easy to see that $R(X_{ai}, \b X_{aj}, Y_{ai},\b Y_{aj})=0$. The other part can be proved similarly.

\end{proof}

To compute the Ricci curvature, we know that $\Ric=\mu g$ and thus
\bee
\begin{split}
\mu=&\sum_{a,i}\lf[R(U_1,\b U_1,X_{ai},\b X_{ai})+R(U_1,\b U_1,Y_{ai},\b Y_{ai})\ri]+\sum_{a}R(U_1,\b U_1,U_a,\b U_a)\\
&+\sum_{a<b}R(U_1,\b U_1,W_{ab},\b W_{ab})\\
=&\frac12\lf(2(n-p)+2(n-p)\ri)+2+ (p-1)\\
=&2n-p+1.
\end{split}
\eee

\begin{lma}\label{l-C-eigenvalue-1} Let $\lambda$ be the largest eigenvalue of the quadratic form $$\sum_{A,B }R_{A\b AB\b B}x_Ax_B$$  in the Weyl frame, where $x_A$ are reals.
\begin{itemize}
  \item [(a)] $\lambda\le 2n-p+1$ if and only if $5p\le 4n+3$.
  \item [(b)] If $5p< 4n+3$, then $\lambda=(2n-p-1)$ iff the corresponding eigenvector has $x_A=x_B$ for all $A,B$.
  \item [(c)] If $5p=4n+3$, then there is an eigenvector with eigenvalue $(2n-p-1)$ such that $x_A\neq x_B$ for some $A\neq B$.
\end{itemize}
\end{lma}
\begin{proof}
Part $(a)$: the argument is identical to the proof of Lemma \ref{l-B-eigenvalue-1} (a) except that: in Case 1 we use that for any $B$ the coefficients in $\sum_{A}R(A,B)x_A$ must add to $2n-p+1$ (instead of $2n-p$), in Case 2 we use that $$\sum_{b,j}\lf(R(X_{ai},X_{bj})-R(X_{ai}, Y_{bj})\ri)=\frac32 p -\frac12$$ (instead of $\frac32 p +\frac12$).

 Parts  (b) and (c): the argument is similar to the corresponding proofs for Lemma \ref{l-B-eigenvalue-1}.
\end{proof}
\begin{lma}\label{l-C-eigenvalue-2}
 Let $\lambda$ be the largest eigenvalue of the quadratic form \begin{equation}\label{SABC}\sum_{A,B,C,D; A\neq B, C\neq D}R_{A\b B C\b D}x_{A  B}x_{C  D}\ee in the Weyl frame, where $x_{A  B}=\ol{x_{B  A}} $.
\begin{itemize}
    \item [(a)]
  $\lambda\le  2n-p+1 $ if and only if $5p\le 4n+3$.
    \item [(b)] If $5p< 4n+3$, then $\lambda< 2n-p+1 $.
  \end{itemize}
\end{lma}
\begin{proof}
We want to estimate \be\label{SABC} S_{AB}=\sum_{x\neq y}|R_{A\b B y\b x}|\ee for each case of $A,B$.  Note that $S_{AB}=S_{BA}$. Recall the following properties
\begin{enumerate}
\item[(C1)]  If $A-B \neq x-y$ then $R_{A\b B y\b x}=0$.
\item[(C2)] If neither $A-B$ nor $A+y$ are roots then  $R_{A\b B y\b x}=0$.
\end{enumerate}
\vspace{12pt}

In each case we will use these to reduce the terms in \eqref{SABC} as much as possible.  Then Lemmas \ref{l-curvature-formula-3} and \ref{l-curvature-classical-2} will be used to calculate the absolute values of the remaining curvature terms.

CASE (i) $A=X_{ai}, B=X_{bj}$ with $(a,i)\neq (b,j)$.

Note that $A,B\in \Delta^+_p(1)$.  By  (C1) we may assume that $x,y\in\Delta^+_p(1)$ or $x,y\in\Delta^+_p(2)$.   Note that the sum of the coordinates of each $X$  is 0,  the sum of the coordinates of each $Y$  is 2,  the sum of the coordinates of each $U$  is 2 and  sum of the coordinates of each  $W$ is 2. Thus by  (C1),  \eqref{SABC}  reduces to:

\be\label{SABCi} \begin{split}S_{AB}=&\sum_{c,k,d,l}|R(X_{ai}, \b X_{bj}, X_{c k},\b X_{d l})|+\sum_{c,k,d,l}|R(X_{ai}, \b X_{bj}, Y_{c k},\b Y_{d l})|\\&+\sum_{c,d,e,f}|R(X_{ai}, \b X_{bj}, W_{c d}, \b W_{e f})|\\&+\sum_{c,d,e}|R(X_{ai}, \b X_{bj}, U_c,W_{de})|+\sum_{c,d,e}|R(X_{ai}, \b X_{bj}, W_{de}, U_c)|\\=:&I+II+III+IV+V
\end{split}
\ee
\vspace{12pt}

If $a\neq b$, $i\neq j$ then: All terms in $III, IV, V$ are zero by (C1).  By (C1), the only non-zero term in $I$ is $|R(X_{ai}, \b X_{bj}, X_{bj},\b X_{ai})|$ leaving $I=\frac14$.  We get $II=\frac14$ in the same way.   Hence $S_{AB}=p+0+\frac{1}{2}(p-1)=\frac14<2n-p+1$.
\vspace{12pt}

If $a=b$, $i\neq j$ then: All terms in $III, IV, V$ are zero by (C1).   By (C1), the only non-zero term in $I$ is $|R(X_{ai}, \b X_{aj}, X_{cj},\b X_{ci})|$ for any $c$, leaving $I=p$.  By (C1), the only non-zero term in $II$ is $|R(X_{ai}, \b X_{aj}, Y_{cj},\b Y_{ci})|$ which is $0$ by Lemma \ref{l-curvature-classical-2} if $a=c$ leaving $II=\frac12 (p-1)$.  Hence $S_{AB}=p+0+\frac{1}{2}(p-1)=3p/2 - 1/2 \leq 2n-p+1$ if and only if $5p\le 4n+3$, and $S_{AB}< 2n-p+1$ if and only if $5p< 4n+3$.
\vspace{12pt}

If $a\neq b$, $i= j$, we may assume that $a<b$, then: By (C1), the only non-zero terms in $I$ are $|R(X_{ai}, \b X_{bi}, X_{bk},\b X_{ak})|$ for any $k$, leaving $I=n-p$. By (C1), the only non-zero terms in $II$ are $|R(X_{ai}, \b X_{bi}, Y_{bk},\b Y_{ak})|$ for any $k$, leaving a contribution of $1/2$ when $k=i$ by Lemma \ref{l-curvature-classical-2} and a contribution of $n-p-1$ for the cases when $k\neq i$.   By (C1), the only non-zero terms in $III$ are $|R(X_{ai}, \b X_{bi}, W_{bc},\b W_{ac})|$ for $c>b$,  or $|R(X_{ai}, \b X_{bi}, W_{cb},\b W_{ca})|$ for $c<a$, or $|R(X_{ai}, \b X_{bi}, W_{bc},\b W_{ca})|$ for $a<c<b$ in which cases the respective contributions to $III$ are $1/2(p-b), 1/2(a-1), 1/2(b-a-1)$.    By (C1), the only non-zero term in $IV$  is $|R(X_{ai}, \b X_{bi}, U_b,W_{ab})|$ leaving $IV=\sqrt{2}/2$.   By (C1), the only non-zero term in $V$  is $|R(X_{ai}, \b X_{bi}, W_{ab}, U_a)|$ leaving $V=\sqrt{2}/2$. Hence $S_{AB}= (n-p)+   \frac12 +(n-p-1) +\sqrt2+\frac12\lf[(p-b)+ (a-1)+ (b-a-1)\ri]=2n-\frac32p-\frac12+\sqrt 2<2n-p+1.$

\vspace{12pt}

CASE (ii) $A=X_{ai}$, $Y_{bj}$.

Note that $A,B\in \Delta^+_p(1)$.  By (C1), $x,y\in\Delta^+_p(1)$ or $x,y\in\Delta^+_p(2)$ and \eqref{SABC}  reduces to:

\be\label{SABCii} \begin{split}S_{AB}=&\sum_{c,k,d,l}|R(X_{ai}, \b Y_{bj}, Y_{c k},\b X_{d l})|
\end{split}
\ee

If $a\neq b$, $i\neq j$:  then by (C1), the only possible non-zero terms in \eqref{SABCii} are $|R(X_{ai}, \b Y_{bj}, Y_{bj},\b X_{ai})|$, which is zero by (C2) and $|R(X_{ai}, \b Y_{bj}, Y_{bi},\b X_{aj})|$ which is $\frac12$.  Hence
$S_{AB}=\frac12<2n-p+1$.
\vspace{12pt}

If $a=b$, $i\neq j$: then by (C1), the only non-zero terms in \eqref{SABCii} are $|R(X_{ai}, \b Y_{aj}, Y_{cj},\b X_{ci})|$ or $|R(X_{ai}, \b Y_{aj}, Y_{ci},\b X_{cj})|$,  for any $c$.  In the first case the contribution to  \eqref{SABCii} is $p$ and in the second case the contribution to  \eqref{SABCii} is $0$ when $c=a$ and $\frac12$ when $c\neq a$.  Thus  $S_{AB}=p+\frac12 (p-1)=\frac32 p-\frac12 <2n-p+1$
\vspace{12pt}

If $a\neq b$, $i= j$ then by (C1), the only non-zero term in \eqref{SABCii} is given by  $|R(X_{ai}, \b Y_{bi}, Y_{bi},\b X_{bi})|=1/2$.  Thus  $S_{AB}=\frac12<2n-p+1$.
\vspace{12pt}

If $a=b$, $i= j$ then by (C1), the only non-zero terms in \eqref{SABCii} are $|R(X_{ai}, \b Y_{ai}, Y_{ci},\b X_{ci})|$ and the contribution to \eqref{SABCii} is $1$ when $c=a$ and $\frac12(p-1)$ from the cases $c\neq a$ by Lemma   \ref{l-curvature-classical-2}.  Thus $S_{AB}=1+ \frac12(p-1)<2n-p+1$.
\vspace{12pt}

CASE (iii) $A=X_{ai}, B=U_b$.

Note that $A\in \Delta^+_p(1)$ and $B\in \Delta^+_p(2)$.  By  (C1), $x\in \Delta^+_p(1)$ and $x\in \Delta^+_p(2)$ and \eqref{SABC}  reduces to:

\be\label{SABiii} \begin{split}S_{AB}=& \sum_{c,d,j}|R(X_{a i}, \b U_b, U_c,\b X_{dj})| + \sum_{c,d,e,j}|R(X_{a i}, \b U_b, W_{cd},\b X_{ej})|=:I+II
\end{split}
\ee
\vspace{12pt}

If $a\neq b$ then:  All terms in $II$ are zero by (C1).   By (C1), the only possible non-zero term in $I$ is $|R(X_{a i}, \b U_b, U_b,\b X_{ai})|$  which in turn is zero by (C2).  \vspace{12pt}

If $a=b$ then:  By (C1), the only non-zero terms in $I$ is $|R(X_{a i}, \b U_a, U_a,\b X_{ai})|$  leaving $I=1$.  By (C1), the only non-zero terms in $II$ are $|R(X_{a i}, \b U_a, W_{ac},\b X_{ci})|$ for $c>a$, and $|R(X_{a i}, \b U_a, W_{ca},\b X_{ci})|$ for $c<a$ in which cases the respective contributions to $II$ are $\frac{\sqrt{2}}{2}(p-a), \frac{\sqrt{2}}{2}(a-1)$.  Thus $S_{AB}=1+\frac{\sqrt 2}2\lf((p-a)+(a-1)\ri)=1+\frac{\sqrt 2}2(p-1)<2n-p+1$.
\vspace{12pt}

CASE (iv) $A=X_{ai}, B=W_{bc}$.

Note that $A\in \Delta^+_p(1), B\in \Delta^+_p(2)$.  By  (C1), $x\in\Delta^+_p(1)$ and $y\in\Delta^+_p(2)$ and  \eqref{SABC}  reduces to:

\bee\label{SABCiv}
\begin{split}
S_{AB}=&  \sum_{d,j,e,f} |R(X_{a i}, \b W_{bc}, W_{ef},\b X_{dj})|+ \sum_{d,j,e} |R(X_{a i}, \b W_{bc}, U_e,\b X_{dj})|:=I+II \\
\end{split}
\eee

If $a=b$ then by (C1), the only non-zero terms in  $I$ are  $ |R(X_{a i}, \b W_{ac}, W_{dc},\b X_{di})|$ for $d<c$, or $ |R(X_{a i}, \b W_{ac}, W_{cf},\b X_{fi})|$ for $f>c$.  In the first case the contribution is $\frac12 (c-1)$ and the contribution in the second case is $\frac12 (p-c)$.    By (C1), the only non-zero term in $II$ is   $ |R(X_{a i}, \b W_{ac}, U_c,\b X_{ci})|$ which is $\sqrt{2}$.  Thus $S_{AB}=\frac12 (p-1)+\sqrt{2}<2n-p+1$.
\vspace{12pt}

If $a\neq b$ then by (C1) and (C2) the terms in $I, II$ are zero unless $a=c$ in which case we get as above, that  $S_{AB}=\frac12 (p-1)+\sqrt{2}<2n-p+1$.
\vspace{12pt}

CASE (v) $A=Y_{ai}, B=Y_{bj}$, $(a,i)\neq (b,j)$. Similar to (i)

CASE (vi) $A=Y_{ai}, B=U_{b}$. Similar to (iii).

CASE (vii) $A=Y_{ai}, B=W_{bc}$. Similar to (iv).

CASE (viii) $A=U_a$, $B=U_b$, $a<b$.

From (C1) it is not hard to see here that \eqref{SAB}  reduces to:
\vspace{12pt}

\bee
\begin{split}
S_{AB}=&\sum_{c,d} |R(U_a,\b U_b, U_b,\b U_a)|=0\\
\end{split}
\eee
Where the last equality follows by (C2).
\vspace{12pt}

CASE (ix) $A=U_a$, $B=W_{bc}$, $b<c$.
\vspace{12pt}

Note that $A, B\in \Delta^+_p(2)$.  By  (C1), $x, y\in \Delta^+_p(1)$ or $A, B\in \Delta^+_p(2)$ and  \eqref{SABC}  reduces to:

\bee
\begin{split}
S_{AB}=&  \sum_{d,e,i} |R(U_a,\b W_{bc}, X_{di},\b X_{ei})| +\sum_{d,e,i} |R(U_a,\b W_{bc}, Y_{di},\b Y_{ei})|\\
&+\sum_{d,e,f,g} |R(U_a,\b W_{bc}, W_{de},\b W_{fg})|+\sum_{d,e,f} |R(U_a,\b W_{bc}, U_d,\b W_{ef})|\\&+\sum_{d,e,f} |R(U_a,\b W_{bc}, W_{de},\b U_f)|\\
=&I+II+III+IV+V\\
\end{split}
\eee

Note that $A+x$ is never a root  and thus by (C2), the only non-zero terms are when $a=b$ or $a=c$.  If $a=b$ then by (C1) the only non-zero terms in $I$ are $ |R(U_a,\b W_{ac}, X_{ci},\b X_{ai})|$ for any $i$, leaving $I=\frac{\sqrt{2}}{2} (n-p)$.  Similarly, we get $II=\frac{\sqrt{2}}{2} (n-p)$.  By (C1) the only non-zero terms in $III$ are $|R(U_a,\b W_{ac}, W_{ce},\b W_{ae})|$ for $e>c$, and $|R(U_a,\b W_{ac}, W_{dc},\b W_{da})|$ for $d<a$, and $|R(U_a,\b W_{ac}, W_{dc},\b W_{ad})|$ for $a<d<a$ in which cases the respective contributions to $III$ are $\frac{\sqrt{2}}{2}(p-c), \frac{\sqrt{2}}{2}(a-1), \frac{\sqrt{2}}{2}(c-a-1)$.  By (C1) the only non-zero term in $IV$ is  $|R(U_a,\b W_{ac}, U_c,\b W_{ac})|$ leaving  $IV=1$.  By (C1) the only non-zero term in $V$ is  $|R(U_a,\b W_{ac}, W_{ac},\b U_a)|$ leaving  $V=1$.
Thus when $a=b$, and similarly when $a=c$, $S_{AB}=\sqrt2(n-p)+\frac{\sqrt 2}2(p-2)+2<2n-p+1$.

CASE (x) $A=W_{ab}, B=W_{cd}$, $(a,b)\neq (c,d)$. We may assume that $a\le c$.

Note that $A,B \in \Delta^+_p(2)$.  By  (C1), $x,y\in\Delta^+_p(1)$ or $x,y\in\Delta^+_p(2)$ and  \eqref{SABix}  reduces to:

\bee
\begin{split}
S_{AB}=&  \sum_{e,f,i} |R(W{ab},\b W_{cd}, X_{ei},\b X_{fi})| +\sum_{e,f,i} |R(W{ab},\b W_{cd}, Y_{ei},\b Y_{fi})|\\
&+\sum_{e,f,g,h} |R(W{ab},\b W_{cd}, W_{ef},\b W_{gh})|+\sum_{e,f,g} |R(W{ab},\b W_{cd}, W_{ef},\b U_g)|\\&+\sum_{e,f,g} |R(W{ab},\b W_{cd}, U_e,\b W_{fg})|\\
=&I+II+III+IV+V\\
\end{split}
\eee

Note that $A+y$ is never a root for any positive root $y$.  Thus by (C2), all terms in $I, II, III, IV, V$ are zero unless either $a=c$ or $b=c$ or $b=d$. Assume that $a=c$, and without loss of generality that $b<d$. By (C1), the only non-zero terms in $I$ are  $|R(W_{ab},\b W_{ad}, X_{di},\b X_{bi})|$  leaving $I=\frac12 (n-p)$.   We similarly get $II=\frac12 (n-p)$. By (C1), the only non-zero terms in $III$ are $|R(W_{ab},\b W_{ad}, W_{de},\b W_{be })|$ for $e>d$, or $|R(W_{ab},\b W_{ad}, W_{ed},\b W_{ eb})|$ for $e<b$, or $|R(W_{ab},\b W_{ad}, W_{ed},\b W_{be })|$ for $b<e<d$, in which cases the respective contributions to $III$ are $1/2(p-d), 1/2(b-1), 1/2(d-b-1)$.  By (C1) the only non-zero term in $IV$ is  $|R(U_a,\b W_{ac}, W_{bd},\b U_b)|$ leaving $V=\frac{\sqrt{2}}{2}$.  Similarly, we get $V=\frac{\sqrt{2}}{2}$.  Thus when $a=c$, and similarly when $b=c$ or $b=d$, $S_{AB}=(n-p)+\frac12(n-p)+\frac12\lf[(p-d)+(b-1)+(d-b-1)
 \ri]+\sqrt 2=n-1-\frac12 p +\sqrt 2<2n-p+1$

This completes the proof of the Lemma.
\end{proof}
By  Lemmas \ref{l-C-eigenvalue-1} and \ref{l-C-eigenvalue-2}, we can proceed as in the $B$ cases to obtain:

\begin{thm}\label{t-C} The \K $C$-space $(C_n,\a_p)$, $n\ge 3$, $1<p<n$ satisfies $QB\ge0$ if and only if $5p\le 4n+3$.  Moreover, $QB>0$ if and only if $5p< 4n+3$.
 \end{thm}
\begin{rem} As in the $B$   cases, one can see that $(C_n,\a_p)$ does not satisfy $B^\perp\ge0$.
\end{rem}

\section{\K $C$-spaces of exceptional type}\label{section-exceptional}

For each of the exceptional Lie algebras $G_2, F_4, E_6, E_7, E_8$, we will establish whether or not the corresponding \K C-spaces with $b_2=1$ have $QB\geq 0$ or not. For each case, we define the following quadratic forms with respect to the Weyl frames:
$$M_1:=\sum_{A,B}R_{A\b AB\b B}x_Ax_B$$
$$M_2:=\sum_{{A,B,C,D;}\atop{  A\neq B, C\neq D}}  R_{A\b B C\b D}x_{AB}x_{CD},$$   where the $x_A$'s are real and $\ol{x_{AB}}=x_{BA}$. We will study the largest eigenvalues of these two quadratic forms. By Lemma \ref{c-QB}, these will tell us whether the space satisfies $QB\ge0$, or $QB>0$.

For each exceptional Lie algebra $\mathfrak{g}$, we will first present an explicit root system (in some Euclidean space $\R^n$) and fundamental set of roots $\{\alpha_1,..\}$.   Then for each corresponding \K $C$-space $(\mathfrak{g}, \alpha_k)$ we present a Weyl frame.   Lemma \ref{l-curvature-formula-2}  then allows explicit calculation of the matrix for $M_1$.  The main point here is to determine $\wt N_{\a,\pm\beta}$.   From this point, while eigenvalue estimates are possible by row sum and symmetry arguments, as in $M_1$ in the classical cases, we compute the eigenvalues and Ricci curvature (row sum) of $M_1$ directly using MAPLE.

 To estimate the largest eigenvalue of $M_2$, we will use Lemmas \ref{l-curvature-formula-2} and \ref{l-curvature-formula-3} to compute the curvature tensor. However, in this case the lemmas allow only an upper estimate for the absolute value of the entries of $M_2$,  since we can only calculate the $N_{\alpha\beta}$'s appearing there up to a sign.  By the same reason, this can  only be estimated from above by $|R_1|+|R_2|$ in  the formula Lemma \ref{l-curvature-formula-3}(2).  In some cases, this becomes too large, and we cannot get a good estimate. Hence, for $A, B, C, D$ corresponding to positive roots $\a\in \Delta^+(i),\beta\in \Delta^+(j),\gamma\in \Delta^+(k),\delta\in \Delta^+(l)$ where $A\neq B$, define  $\wt R(A,\b B, C,\b D)$ as follows:
 \vspace{20pt}

\be\label{eq-tR}
\wt R(A,\b B, C,\b D)=\left\{
  \begin{array}{ll}
    0, & \hbox{if $\a-\beta\neq\delta-\gamma$;}\\|R(A,\b A, B,\b B)|, &\hbox{if $B=C$ (i.e. if $\beta=\gamma$);}\\
m, & \hbox{if $\a-\beta=\delta-\gamma$, and $\beta-\gamma\neq0$,}
  \end{array}
\right.
\ee
 where $m=$

 $\min\{|R_1(A,\b B, C,\b D)|+|R_2(A,\b B, C,\b D)|,
|R_1(C,\b B, A,\b D)|+|R_2(C,\b B, A,\b D))|\}$.

 \vspace{20pt}

Note that in the last case, we also have $\a-\delta\neq0$.  Since $R(A,\b B, C,\b D)=R(C,\b B, A,\b D)$ by symmetries of the curvature tensor, we have $|R(A,\b B, C,\b D)|\le  \wt  R(A,\b B, C,\b D),$ for all $A\neq B, C\neq D$.

\begin{rem} In many cases, $|R(A,\b B, C,\b D)|$ is exactly equal to the quantity $m$ in \eqref{eq-tR}. For example, this is the case if one of $\a+\gamma, \a-\beta, \beta-\gamma$ is not a root.
\end{rem}
Consider the following matrix $Z_{AB, CD}$, which is defined for all pairs $AB, CD$:

 \be\label{Z}
Z_{AB, CD} = \left\{
   \begin{array}{ll}
    0, & \hbox{if $A=B$ or $C=D$;} \\
     \wt R(A,\b B, C,\b D), & \hbox{otherwise;} \\
     \end{array}
 \right.
\ee
 Recall that $(M_2)_{AB, CD}=R(A, \b B, C \b D)$ is only defined for pairs with $A\neq B$, $C\neq D$.  For any $A$ the $AA$th row
and column of $Z$ has zero in every entry, and removing these rows and columns leaves a symmetric matrix with the same dimension as $M_2$, bounding $M_2$ from above, entry-wise in absolute values.  The following simple Lemma justifies estimating the largest eigenvalue of $M_2$ by the largest absolute eigenvalue of $Z$

\begin{lma}\label{eigenvalueZ} Let $N,M$ be a real symmetric $n\times n$ matrices such that
  $N_{ij}\geq |M_{ij}|$ for all $i,j.$
Then spectral radius (the maximal absolute value of eigenvalues) $\lambda_N$ of $N$ is greater than or equal to the spectral radius $\lambda_M$ of $M$.
\end{lma}
\begin{proof}
  Let $x=(x_1,..,x_n)$ be a unit eigenvector of $M$ for which $|Mx|=\lambda_M$.  Note that $|x|=(|x_1|,..,|x_n|)$ is also a unit vector.  Then we have $\lambda_M=|Mx|=\sqrt{|\sum_j M_{ij}x_i|^2}\leq \sqrt{|\sum_j N_{ij}|x_i||^2}=N|x|\leq \lambda_N$.
  \end{proof}

 We will calculate $Z$ in each case using MAPLE.  From this point, we can of course compute the eigenvalues of $Z$ directly using MAPLE, so obtaining an eigenvalue estimate for $M_2$.  However, we will use Lemmas \ref{l-eigenvalue} and \ref{weightedrowsumsZ} here as they are elementary and similar to our methods for $M_2$ in the classical case.  In fact, most of the terms in $Z$ are zero and one may be able to decompose $Z$ into quadratic forms
of much smaller size so that Lemmas \ref{l-eigenvalue} and
 \ref{weightedrowsumsZ} can be applied without using a computer.  In all cases other than $(G_2,\a_2)$, $(E_7,\a_5)$ and $(F_4,\a_2)$, the estimate provided by Lemma \ref{l-eigenvalue} will be sufficient while in the cases of  $(G_2,\a_2)$, $(E_7,\a_5)$ and  $(F_4,\a_2)$ Lemma \ref{weightedrowsumsZ} is used to estimate the eigenvalue of $Z$.

  In the subsections below we just present the results of the MAPLE calculations, and we indicate the algorithms used in the appendix.  For each Lie algebra below, the dual Cartan subalgebra $\mathfrak{h}^*$ is associated to some Euclidean subspace $V$, and the root system is given as a set of vectors in $V$. We refer to \cite{Bourbaki} for details.  We will use $\xi_1,...., \xi_n$ to denote the standard coordinates on $\R^n$ and $\ve_1,...,\ve_n$ to denote the standard basis vectors of $\R^n$.

\subsection{The space  $(G_2,\a_2)$}

 Let $V$ be the hyperplane in $\R^3$ with $\xi_1+\xi_2+\xi_3=0$.   The positive roots in $V$ are $$\ve_1-\ve_2, -2\ve_1+\ve_2+\ve_3, -\ve_1 +   \ve_3, -\ve_2+\ve_3, \ve_1-2\ve_2+\ve_3, -\ve_1-\ve_2+2\ve_3.$$  Simple positive roots are $\a_1=\ve_1-\ve_2,\a_2=-2\ve_1+\ve_2+\ve_3$ with respect to which the positive roots are

$$\a_1, \a_2, \a_1+\a_2, 2\a_1+\a_2, 3\a_1+\a_2, 3\a_1+2\a_2.$$

Now $(G_2,\a_1)$ is Hermitian symmetric, so we only consider $(G_2,\a_2)$ for which we have

\be
\begin{split}
\Delta_2^+(1)=&\{\a_1+\a_2, 2\a_1+\a_2; \a_2,3\a_1+\a_2\} \\
\Delta_2^+(2)=&\{3\a_1+2\a_2\}.
\end{split}
\ee
$\Delta_2^+(k)=\emptyset$, for $k\ge3$.

\bee
\left\{
  \begin{array}{ll}
    \dim=5, \Ric=9 g, &   \\
    \text{\rm 4 largest eigenvalues of $M_1$ are 1.5000,1.5000, 8.5000, 9.000}, & \\
\text{\rm eigenvalues of $M_2$ are less than $9$.}\\   \end{array}
\right.
\eee
(the estimate for $M_2$ is obtained by using $\mu=9$ and $s=1$ in Lemma \ref{weightedrowsumsZ} in which case the maximum weighted row sum is $8.6309$).  Thus the space has $QB>0$.

\subsection{The spaces $ (F_4,\a_i) $}\vskip .2cm

\subsubsection{Root system}  Let $V=\R^4$.  The positive roots in $V$ are
$$
\{\ve_i\}_{1\le i\le 4}\cup\{\ve_i+\ve_j\}_{1\le i<j\le 4}\cup\{\ve_i-\ve_j\}_{1\le i<j\le 4}\cup\{\frac12(\varepsilon_1\pm\varepsilon_2\pm\varepsilon_3\pm\varepsilon_4)\}.
$$
A total of $4+6+6+8=24$ positive roots. Let
\be
\begin{split}
A=&\left(
    a_i
  \right)_{i=1}^{12}=
 \left(
     \begin{array}{c}
      \frac{1}{2}(\ve_1+\ve_2+\ve_3+\ve_4)\\
 \frac{1}{2}(\ve_1+\ve_2-\ve_3+\ve_4)\\
  \frac{1}{2}(\ve_1+\ve_2+\ve_3-\ve_4)\\
\frac{1}{2}(\ve_1+\ve_2-\ve_3-\ve_4)\\
\frac{1}{2}(\ve_1-\ve_2+\ve_3+\ve_4)\\
 \frac{1}{2}(\ve_1-\ve_2-\ve_3+\ve_4)\\
 \frac{1}{2}(\ve_1-\ve_2+\ve_3-\ve_4)\\
\frac{1}{2} (\ve_1-\ve_2-\ve_3-\ve_4)\\
\ve_1\\
\ve_2\\  \ve_3\\  \ve_4\\
\end{array}
   \right)\\
   \end{split},
\begin{split}
B=&\left(
    b_i
  \right)_{i=1}^{12}=
 \left(
     \begin{array}{c}
     \ve_1+\ve_2\\  \ve_1+\ve_3\\ \ve_1+\ve_4\\ \ve_2+\ve_3\\ \ve_2+\ve_4\\ \ve_3+\ve_4\\
\ve_1-\ve_2\\  \ve_1-\ve_3\\ \ve_1-\ve_4\\ \ve_2-\ve_3\\ \ve_2-\ve_4\\ \ve_3-\ve_4\\
     \end{array}
   \right)\\
   \end{split}
   \ee

The simple positive roots are $\a_1 =b_{10}, \a_2=b_{12}, \a_3=a_{12}, \a_4=
a_8$.
The matrix for $(\a_i )$ is  $$g=\left(
                                  \begin{array}{cccc}
                                    0 & 1 & -1 & 0 \\
                                    0 & 0 &  1 & -1 \\
                                    0 & 0 & 0 & 1 \\
                                    1/2 & -1/2 &-1/2 & -1/2 \\
                                  \end{array}
                                \right).
                                $$
The coordinates of $(a_i)$ with respect to the ordered basis $\{\a_1,\dots,\a_4\}$ are given by the columns of
\be (gg^t)^{-1}gA^t=\left(\begin {array}{cccccccccccc} 1&1&1&1&0&0&0&0&1&1&0&0
\\ \noalign{\medskip}2&1&2&1&1&0&1&0&2&1&1&0\\ \noalign{\medskip}3&2&2
&1&2&1&1&0&3&1&1&1\\ \noalign{\medskip}1&1&1&1&1&1&1&1&2&0&0&0
\end {array} \right)\ee
The coordinates of $(b_i)$ with respect to the ordered basis $\{\a_1,\dots,\a_4\}$ are given by the columns of

\be  (gg^t)^{-1}gB^t =\left(\begin {array}{cccccccccccc} 2&1&1&1&1&0&0&1&1&1&1&0
\\ \noalign{\medskip}3&3&2&2&1&1&1&1&2&0&1&1\\ \noalign{\medskip}4&4&4
&2&2&2&2&2&2&0&0&0\\ \noalign{\medskip}2&2&2&0&0&0&2&2&2&0&0&0
\end {array} \right) \ee
That is $a_1=\a_1+2\a_2+3\a_3+\a_4$, etc.  From this it is easy to write down the $\Delta^+(k)'s$ for $(F_4,\a_i)$ for $1\le i\le 4$.  Next, let us determine $\wt N_{\a,\pm\beta}$.

\begin{lma}\label{l-N-F} Let $\a$, $\beta$ be positive roots.
\be\label{eq-N-F-1}
\wt N_{\a,\beta}= \left\{
 \begin{array}{ll}
    \sgn(N_{\a,\beta}), & \hbox{if $\a,\beta\in A$, and $|\a+\beta|^2=1$;} \\
    \sqrt2\, \sgn(N_{\a,\beta}), & \hbox{otherwise.}
  \end{array}
\right.
\ee
If $\a-\beta\neq0$, then
\be\label{eq-N-F-2}
\wt N_{\a,-\beta}= \left\{
 \begin{array}{ll}
    \sgn(N_{\a,-\beta}), & \hbox{if $\a,\beta\in A$, and $|\a-\beta|^2=1$;} \\
    \sqrt2 \,\sgn(N_{\a,-\beta}), & \hbox{otherwise.}
  \end{array}
\right.
\ee

\end{lma}
\begin{proof} First note that a root $\a$ is in $A$ or $-\a$ is $A$ if and only if $||\a||^2=1$ and it is in $B$ or $-\a$ is in $B$ if and only if $||\a||^2=2$.

To prove \eqref{eq-N-F-1}, it is sufficient to consider the case that   $\a+\beta$ is a root. Suppose $\a, \beta\in A$,  and  $||\a+\beta||^2=1$ then $(\a,\beta)=-\frac12$. Suppose $\a-\beta$ is also a root, then $||\a-\beta||^2=1,$ or 2 and $(\a,\beta)= \frac12$ or 0, which is impossible. Hence $N_{\a,\beta}=\pm1$. In this case,
$$
\wt N_{\a,\beta}=\frac{|\a||\beta|}{|\a+\beta|}N_{\a,\beta}=\sgn(N_{\a,\beta}).
$$
Suppose $||\a+\beta||^2=2$, then $(\a,\beta)=0$, and $\a-\beta$ is also a root, see \cite[p.324]{Fulton-Harris}. Moreover, $||\a-2\beta||^2=5$ and so $\a-2\beta$ is not a root. Hence $N_{\a,\beta}=\pm 2$. Then $\wt N_{\a,\beta}=\sqrt2\,\sgn(N_{\a,\beta}).$

Suppose $\a\in A,  \beta\in B$,  and  $||\a+\beta||^2=1$ then $(\a,\beta)=-1$. Suppose $\a-\beta$ is also a root, then $||\a-\beta||^2=1,$ or 2, and $(\a,\beta)= 1$ or $\frac12$. Hence $\a-\beta$ is not a root. In this case, $\wt N_{\a,\beta}=\sqrt2\,\sgn(N_{\a,\beta}).$ Suppose $||\a+\beta||^2=2$, then $(\a,\beta)=-\frac12$.
 that $(\a,\beta)$ is an integer. Hence this is impossible.

Suppose  $\a, \beta\in B$,  and  $||\a+\beta||^2=1$ then $(\a,\beta)=-\frac32$. This is impossible. Hence $||\a+\beta||^2=2$ and $(\a,\beta)=-1$. As before, we can prove that $\a-\beta$ is not a root. So $N_{\a,\beta}=\pm1$ and $\wt N_{\a,\beta}=\sqrt2\, \sgn(N_{\a,\beta})$. This completes the proof of \eqref{eq-N-F-1}. The proof of \eqref{eq-N-F-2} is  similar.

\end{proof}

\subsubsection{The space $(F_4,\a_1)$}
\be
\begin{split}
\Delta_1^+(1)=&\{a_1, a_2, a_3, a_4, a_9, a_{10}, b_2, b_3, b_4, b_5, b_8, b_9, b_{10}, b_{11}\} \\
\Delta_1^+(2)=&\{b_1\} \\
\end{split}
\ee
$\Delta_1^+(k)=\emptyset$, for $k\ge3$.

\bee
\left\{
  \begin{array}{ll}
    \dim=15, \Ric=8g, &   \\
    \text{\rm 4 largest eigenvalues of $M_1$ are 8, 4.5, 4.5, 4.5}, & \\
\text{\rm eigenvalues of $M_2$ are at most $4.9142$  (using Lemma \ref{l-eigenvalue})}.
  \end{array}
\right.
\eee
Thus the space has $QB> 0$.

\subsubsection{The space $(F_4,\a_2)$}
\be
\begin{split}
\Delta_2^+(1)=&\{a_2,a_4,a_5,a_7, a_{10},a_{11}; b_5,b_6,b_7,b_8,b_{11}, b_{12}\}\\
\Delta_2^+(2)=& \{a_1  ,a_3 ,a_9, b_3 , b_4 , b_9\}\\
\Delta_2^+(3)=&\{ b_1, b_2\}.
\end{split}
\ee
$\Delta_2^+(k)=\emptyset$, for $k\ge4$.

\bee
\left\{
  \begin{array}{ll}
    \dim=20, \Ric=5g, &   \\
    \text{\rm 4 largest eigenvalues of $M_1$ are 5, 4.8941, 4.8941, 4.6543}, & \\
\text{\rm eigenvalues of $M_2$ are less than $5$}.
  \end{array}
\right.
\eee
(the estimate for $M_2$ is obtained by using $\mu=5$ and $s=4$ in Lemma \ref{weightedrowsumsZ} in which case the maximum weighted row sum is $4.9822$.  When we take $s=10$, then the maximal weighted row sum actually becomes $4.8070$).

 Thus the space has $QB> 0$.

\subsubsection{The space $(F_4,\a_3)$}
\be
\begin{split}
\Delta_3^+(1)=&\{a_4, a_6, a_7, a_{10},a_{11},a_{12}\}\\
\Delta_3^+(2)=& \{a_2,a_3,a_5, b_4,..,b_9\}\\
\Delta_3^+(3)=& \{a_1, a_9\}\\
\Delta_3^+(4)=& \{b_1, b_2,b_3\}
\end{split}
\ee
$\Delta_3^+(k)=\emptyset$, for $k\ge5$.

\bee
\left\{
  \begin{array}{ll}
    \dim=20 ,  \Ric=7/2g, &   \\
    \text{\rm 4 largest eigenvalues of $M_1$ are 3.6888, 3.5, 2.4137, 2.4137}, & \\
  \end{array}
\right.
\eee
 Thus the space does not have  $QB\geq  0$.

\subsubsection{The space $(F_4,\a_4)$}
\be
\begin{split}
\Delta_4^+(1)=&\{a_1,..,a_8\}\\
\Delta_4^+(2)=&\{a_9; b_1,b_2,b_3,b_7,b_8,b_9\}.
\end{split}
\ee
$\Delta_4^+(k)=\emptyset$, for $k\ge3$.

\bee
\left\{
  \begin{array}{ll}
    \dim=15, \Ric=11/2 g, &   \\
    \text{\rm 4 largest eigenvalues of $M_1$ are 5.5, 2.1328, 2.1328, 2.1328}, & \\
\text{\rm eigenvalues of $M_2$ are at most $3.9571$ (using Lemma \ref{l-eigenvalue})}.
  \end{array}
\right.
\eee
  Thus the space has $QB> 0$.

\subsection{The spaces $  (E_6,\a_i)$}

\subsubsection{Root system}

Consider the subspace $V$ of $\R^8$ such that $\xi_6=\xi_7=-\xi_8.$  The positive roots in $V$ are $\pm \ve_i+\ve_j, 1\le i<j\le 5$ (total 20), and
$$
\frac12(\ve_8-\ve_7-\ve_6+\sum_{i=1}^5(-1)^{\nu(i)}\ve_i)%=\frac12(\ve_8-\ve_7-\ve_6+\sum_{i=1}^5\pm\ve_i)
$$
so that $\sum_1^5\nu(i)$ is even, i.e. the number of minus sign is even (total 16).
Let
\be
\begin{split}
A=&\left(
    a_i
  \right)_{i=1}^{10}=
 \left(
     \begin{array}{c}
       \ve_1+\ve_2 \\
        \ve_1+\ve_3 \\
      \ve_1+\ve_4 \\
      \ve_1+\ve_5 \\
       \ve_2+\ve_3 \\
       \ve_2+\ve_4 \\
       \ve_2+\ve_5\\
       \ve_3+\ve_4 \\
       \ve_3+\ve_5 \\
       \ve_4+\ve_5 \\
     \end{array}
   \right)\\
   \end{split},
\begin{split}
B=&\left(
    b_i
  \right)_{i=1}^{10}=
 \left(
     \begin{array}{c}
       -\ve_1+\ve_2 \\
        -\ve_1+\ve_3 \\
      -\ve_1+\ve_4 \\
      -\ve_1+\ve_5 \\
       -\ve_2+\ve_3 \\
       -\ve_2+\ve_4 \\
       -\ve_2+\ve_5\\
       -\ve_3+\ve_4 \\
       -\ve_3+\ve_5 \\
       -\ve_4+\ve_5 \\
     \end{array}
   \right)\\
   \end{split}
   \ee
   \be
   \begin{split}
C=&\left(
    c_i
  \right)_{i=1}^{10}=
 \left(
     \begin{array}{c}
      \frac12(\ve_8-\ve_7-\ve_6-\ve_1-\ve_2+\ve_3+\ve_4+\ve_5) \\
       \frac12(\ve_8-\ve_7-\ve_6-\ve_1+\ve_2-\ve_3+\ve_4+\ve_5) \\
     \frac12(\ve_8-\ve_7-\ve_6-\ve_1+\ve_2+\ve_3-\ve_4+\ve_5) \\
      \frac12(\ve_8-\ve_7-\ve_6-\ve_1+\ve_2+\ve_3+\ve_4-\ve_5) \\
       \frac12(\ve_8-\ve_7-\ve_6+\ve_1-\ve_2-\ve_3+\ve_4+\ve_5) \\
       \frac12(\ve_8-\ve_7-\ve_6+\ve_1-\ve_2+\ve_3-\ve_4+\ve_5) \\
      \frac12(\ve_8-\ve_7-\ve_6+\ve_1-\ve_2+\ve_3+\ve_4-\ve_5)\\
      \frac12(\ve_8-\ve_7-\ve_6+\ve_1+\ve_2-\ve_3-\ve_4+\ve_5)\\
       \frac12(\ve_8-\ve_7-\ve_6+\ve_1+\ve_2-\ve_3+\ve_4-\ve_5) \\
       \frac12(\ve_8-\ve_7-\ve_6+\ve_1+\ve_2+\ve_3-\ve_4-\ve_5) \\
     \end{array}
   \right)\\
   \end{split}
   \ee
    \be
   \begin{split}
D=&\left(
    \begin{array}{c}
      d_1 \\
      d_2 \\
      d_3 \\
      d_4 \\
      d_5 \\
      d_6 \\
    \end{array}
  \right)=
 \left(
     \begin{array}{c}
      \frac12(\ve_8-\ve_7-\ve_6+\ve_1+\ve_2+\ve_3+\ve_4+\ve_5) \\
       \frac12(\ve_8-\ve_7-\ve_6+\ve_1-\ve_2-\ve_3-\ve_4-\ve_5) \\
     \frac12(\ve_8-\ve_7-\ve_6-\ve_1+\ve_2-\ve_3-\ve_4-\ve_5) \\
      \frac12(\ve_8-\ve_7-\ve_6-\ve_1-\ve_2+\ve_3-\ve_4-\ve_5) \\
       \frac12(\ve_8-\ve_7-\ve_6-\ve_1-\ve_2-\ve_3+\ve_4-\ve_5) \\
       \frac12(\ve_8-\ve_7-\ve_6-\ve_1-\ve_2-\ve_3-\ve_4+\ve_5) \\
     \end{array}
   \right)\\
   \end{split}
   \ee

Simple positive roots are: $\a_1=d_2, \a_2=a_1,\a_3=b_1,\a_4=b_5,\a_5=b_8,\a_6=b_{10}.$ The matrix for $(\a_i)$ is
\be
g=\left(
    \begin{array}{cccccccc}
      \frac12 & -\frac12 & -\frac12 & -\frac12 & -\frac12 &-\frac12 &-\frac12 & \frac12 \\
      1 & 1 & 0 & 0 & 0 & 0 & 0 & 0 \\
      -1 & 1 & 0 & 0 & 0 & 0 & 0 & 0 \\
      0 & -1 & 1 & 0 & 0 & 0 & 0 & 0 \\
      0 & 0 & -1 & 1 & 0 & 0 & 0 & 0 \\
      0 & 0 & 0 & -1 & 1 & 0 & 0 & 0 \\
    \end{array}
  \right)
\ee
The coordinates of $a_i$ relative the ordered basis $\{\a_1,\dots,\a_6\}$ are the columns  of
\be (gg^t)^{-1}gA^t =
\left(
    \begin{array}{cccccccccc}
     0  &        0 &   0  &   0&   0  &  0& 0& 0& 0&0\\
 1& 1& 1& 1& 1& 1& 1&1 &1&1\\
  0&  0 &   0& 0& 1& 1& 1& 1&1 &1 \\
  0&1& 1& 1& 1& 1& 1& 2 &2&2\\
  0&  0&1& 1&   0&1& 1& 1&1 &2\\
  0&  0   &0& 1&   0   &0& 1& 0&1&1 \\
\end{array}
  \right)
\ee
The coordinates of $b_i$ relative the ordered basis $\{\a_1,\dots,\a_6\}$ are the columns  of
 \be
(gg^t)^{-1}gB^t =\left(
    \begin{array}{cccccccccc}

       0    &      0   & 0 & 0 & 0& 0 & 0 &0 & 0 & 0 \\
 0 &0 & 0&0 &0 & 0 & 0 & 0 & 0 &0 \\
  1 & 1 & 1 & 1 & 0&0 & 0 &0 & 0 & 0 \\
  0& 1 & 1 & 1 & 1 & 1 & 1 & 0 & 0 &0 \\
  0& 0& 1 & 1 & 0& 1 & 1 & 1 & 1 & 0 \\
  0& 0&0 & 1 & 0&0 & 1 &0 & 1 & 1 \\
\end{array}
  \right)
\ee
The coordinates of $c_i$ relative the ordered basis $\{\a_1,\dots,\a_6\}$ are the columns  of

\be(gg^t)^{-1}gC^t =
 \left(
    \begin{array}{cccccccccc}
  1  &   1 &1 &1 &1 &1 &1 &1 &1 &1 \\
 1 &1 &1 &1 &1 &1 &1 &1 &1 &1  \\
 2 &2 &2 &2 &1 &1 &1 &1 &1 &1 \\
 3 &2 &2 &2 &2 &2 &2 &1 &1 &1 \\
 2 &2 &1 &1 &2 &1 &1 &1 &1 &  0\\
 1 &1 &1 &   0 &1 &1  &   0 &1 & 0 &  0\\
 \end{array}
  \right)
  \ee
 The coordinates of $d_i$ relative the ordered basis $\{\a_1,\dots,\a_6\}$ are the columns  of
  \be(gg^t)^{-1}gD^t =
 \left(
    \begin{array}{cccccc}
   1  &   1 &1 &1 &1 &1 \\
 2 &0  &  0 &   0 &0  &  0 \\
 2  &  0 &1 &1 &1 &1 \\
 3 &   0  &  0 &1 &1 &1 \\
 2 &  0&  0&  0 &1 &1 \\
 1 &  0&  0&  0  & 0 &1 \\
\end{array}
  \right)
 \ee

We can determine $\wt N_{\a,\beta}$ as before.
\begin{lma}\label{l-N-E} Let $\a, \beta$ be positive roots, then
\be
\wt N_{\a,\beta}=\sqrt 2\,\sgn(N_{\a,\beta}).
\ee
If $\a-\beta\neq 0$, then
\be
\wt N_{\a,-\beta}=\sqrt 2\,\sgn(N_{\a,-\beta}).
\ee
\end{lma}
\begin{proof} The proof is  similar to the proof of Lemma \ref{l-N-F} using the fact that if $\a$ is a root then $|\a|^2=2$.
\end{proof}

Since $(E_6, \a_1)$ and $(E_6, \a_6)$ are Hermitian symmetric space, we only consider $(E_6, \a_i)$, $2\le i\le 5$ below.

\subsubsection{The space $(E_6,\a_2)$}

\be
\begin{split}
  \Delta_2^+(1) =& \{a_1,\dots,a_{10};c_1,\dots,c_{10}\} \\
  \Delta_2^+(2) =&  \{d_1\}
\end{split}
\ee
$\Delta_2^+(k)=\emptyset$, for $k\ge3$.

\bee
\left\{
  \begin{array}{ll}
    \dim=21, \Ric=11g, &   \\
    \text{\rm 4 largest eigenvalues of $M_1$ are 5.5000,
    5.5000,
    5.5000,
   11.0000}, & \\
\text{\rm eigenvalues of $M_2$ are at most $5.5$ (using Lemma \ref{l-eigenvalue})}.
  \end{array}
\right.
\eee

  Thus the space    has $QB>0$.

\subsection{The space  $(E_6,\a_3)$}

\be
\begin{split}
  \Delta_3^+(1) =& \{a_5,\dots,a_{10};b_1,\dots,b_{4};c_5,\dots,c_{10};d_3,\dots,d_6\} \\
 \Delta_3^+(2) =&  \{c_1,\dots,c_4;d_1\}
\end{split}
\ee
$\Delta_3^+(k)=\emptyset$, for $k\ge3$.
\bee
\left\{
  \begin{array}{ll}
    \dim=25, \Ric=9g, &   \\
    \text{\rm 4 largest eigenvalues of $M_1$ are 5.3117,
    5.3117,
    8.5000,
    9.0000}, & \\
\text{\rm eigenvalues of $M_2$ are at most $8.5$ (using Lemma \ref{l-eigenvalue})}.
  \end{array}
\right.
\eee

Thus the space    has $QB>0$.

\subsubsection{The space  $(E_6,\a_4)$}

\be
\begin{split}
  \Delta_4^+(1) =& \{a_2,\dots,a_{7};b_2,\dots,b_{7};c_8,\dots,c_{10};d_4,\dots,d_6\} \\
  \Delta_4^+(2) =&  \{a_8,\dots,a_{10};c_2,\dots,c_7\}\\
  \Delta_4^+(3)=&\{c_1,d_1\}.
\end{split}
\ee
$\Delta_4^+(k)=\emptyset$, for $k\ge4$.
\bee
\left\{
  \begin{array}{ll}
    \dim=29, \Ric=7g, &   \\
    \text{\rm 4 largest eigenvalues of $M_1$ are 5.8226,
    5.8226,
    7.0000,
    7.1468}.
  \end{array}
\right.
\eee

 The space    does not have $QB\ge0$.

\subsubsection{The space  $(E_6,\a_5)$}

\be
\begin{split}
  \Delta_5^+(1) =& \{a_3,a_4,a_6,\dots,a_9;b_3,b_{4}, b_6,\dots,b_9;c_3,c_4, c_6,\dots,c_{9};d_5,d_6\} \\
  \Delta_5^+(2) =&  \{a_{10};c_1,c_2,c_5;d_1\}
\end{split}
\ee
$\Delta_5^+(k)=\emptyset$, for $k\ge3$.

\bee
\left\{
  \begin{array}{ll}
    \dim=25, \Ric=9g, &   \\
    \text{\rm 4 largest eigenvalues of $M_1$ are 5.3117,
    5.3117,
    8.5000,
    9.0000}, & \\
\text{\rm eigenvalues of $M_2$ are at most $ 8.5$ (using Lemma \ref{l-eigenvalue})}.
  \end{array}
\right.
\eee

  Thus the space    has $QB>0$.

\subsection{The spaces $ (E_7,\a_i)$}

\subsubsection{Root system}
Consider the subspace $V$ of $\R^8$, orthogonal to $\ve_7+\ve_8$.
The positive roots in $V$ are $\pm \ve_i+\ve_j, 1\le i<j\le 6$ (total 30), $-\ve_7+\ve_8$ and
$$
\frac12(\ve_8-\ve_7+\sum_{i=1}^6(-1)^{\nu(i)}\ve_i)%=\frac12(\ve_8-\ve_7+\sum_{i=1}^6\pm\ve_i)
$$
so that $\sum_1^6\nu(i)$ is odd, i.e.   the number of minus sign is odd (total 32).  Let
\be
\begin{split}
A=&\left(a_i
  \right)_{i=1}^{15}=
 \left(
     \begin{array}{c}
       \ve_1+\ve_2 \\
        \ve_1+\ve_3 \\
      \ve_1+\ve_4 \\
      \ve_1+\ve_5 \\
      \ve_1+\ve_6 \\
       \ve_2+\ve_3 \\
       \ve_2+\ve_4 \\
       \ve_2+\ve_5\\
       \ve_2+\ve_6 \\
       \ve_3+\ve_4 \\
       \ve_3+\ve_5 \\
       \ve_3+\ve_6 \\
       \ve_4+\ve_5 \\
       \ve_4+\ve_6 \\
       \ve_5+\ve_6 \\
     \end{array}
   \right)
   \end{split},
\begin{split}
B=&\left(b_i
  \right)_{i=1}^{16}=
 \left(
     \begin{array}{c}
       -\ve_1+\ve_2 \\
       - \ve_1+\ve_3 \\
      -\ve_1+\ve_4 \\
      -\ve_1+\ve_5 \\
     - \ve_1+\ve_6 \\
       -\ve_2+\ve_3 \\
       -\ve_2+\ve_4 \\
       -\ve_2+\ve_5\\
       -\ve_2+\ve_6 \\
       -\ve_3+\ve_4 \\
       -\ve_3+\ve_5 \\
       -\ve_3+\ve_6 \\
       -\ve_4+\ve_5 \\
       -\ve_4+\ve_6 \\
       -\ve_5+\ve_6 \\
        -\ve_7+\ve_8 \\
     \end{array}
   \right)\\
   \end{split}
   \ee
    \be
   \begin{split}
C=&\left(
    \begin{array}{c}
      c_1 \\
      c_2 \\
      c_3 \\
     c_4 \\
      c_5 \\
      c_6 \\
    \end{array}
  \right)=
 \frac12\left(
     \begin{array}{c}
      \ve_8-\ve_7-\ve_1+\ve_2+\ve_3+\ve_4+\ve_5+\ve_6\\
         \ve_8-\ve_7+\ve_1-\ve_2+\ve_3+\ve_4+\ve_5+\ve_6\\
       \ve_8-\ve_7+\ve_1+\ve_2-\ve_3+\ve_4+\ve_5+\ve_6\\
        \ve_8-\ve_7+\ve_1+\ve_2+\ve_3-\ve_4+\ve_5+\ve_6\\
        \ve_8-\ve_7+\ve_1+\ve_2+\ve_3+\ve_4-\ve_5+\ve_6\\
         \ve_8-\ve_7+\ve_1+\ve_2+\ve_3+\ve_4+\ve_5-\ve_6\\
     \end{array}
   \right )
   \end{split}
   \ee
%\vspace{200pt}

 \be
   \begin{split}
D=&\left(
    \begin{array}{c}
      d_i
    \end{array}
  \right)_{i=1}^{20}=
 \frac12\left(
     \begin{array}{c}
        \ve_8-\ve_7-\ve_1-\ve_2-\ve_3+\ve_4+\ve_5+\ve_6  \\
      \ve_8-\ve_7-\ve_1-\ve_2+\ve_3-\ve_4+\ve_5+\ve_6\\
       \ve_8-\ve_7-\ve_1-\ve_2+\ve_3+\ve_4-\ve_5+\ve_6  \\
        \ve_8-\ve_7-\ve_1-\ve_2+\ve_3+\ve_4+\ve_5-\ve_6 \\
        \ve_8-\ve_7-\ve_1+\ve_2-\ve_3-\ve_4+\ve_5+\ve_6  \\
        \ve_8-\ve_7-\ve_1+\ve_2-\ve_3+\ve_4-\ve_5+\ve_6  \\
        \ve_8-\ve_7-\ve_1+\ve_2-\ve_3+\ve_4+\ve_5-\ve_6  \\
        \ve_8-\ve_7-\ve_1+\ve_2+\ve_3-\ve_4-\ve_5+\ve_6  \\
         \ve_8-\ve_7-\ve_1+\ve_2+\ve_3-\ve_4+\ve_5-\ve_6  \\
          \ve_8-\ve_7-\ve_1+\ve_2+\ve_3+\ve_4-\ve_5-\ve_6  \\
        \ve_8-\ve_7+\ve_1-\ve_2-\ve_3-\ve_4+\ve_5+\ve_6 \\
         \ve_8-\ve_7+\ve_1-\ve_2-\ve_3+\ve_4-\ve_5+\ve_6 \\
          \ve_8-\ve_7+\ve_1-\ve_2-\ve_3+\ve_4+\ve_5-\ve_6 \\
          \ve_8-\ve_7+\ve_1-\ve_2+\ve_3-\ve_4-\ve_5+\ve_6 \\
          \ve_8-\ve_7+\ve_1-\ve_2+\ve_3-\ve_4+\ve_5-\ve_6 \\
          \ve_8-\ve_7+\ve_1-\ve_2+\ve_3+\ve_4-\ve_5-\ve_6 \\
          \ve_8-\ve_7+\ve_1+\ve_2-\ve_3-\ve_4-\ve_5+\ve_6 \\
          \ve_8-\ve_7+\ve_1+\ve_2-\ve_3-\ve_4+\ve_5-\ve_6 \\
          \ve_8-\ve_7+\ve_1+\ve_2-\ve_3+\ve_4-\ve_5-\ve_6 \\
          \ve_8-\ve_7+\ve_1+\ve_2+\ve_3-\ve_4-\ve_5-\ve_6 \\
     \end{array}
   \right )
   \end{split}
   \ee
  \be
   \begin{split}
E=&\left(
    \begin{array}{c}
      e_1 \\
      e_2 \\
      e_3 \\
     e_4 \\
      e_5 \\
      e_6 \\
    \end{array}
 \right) =
 \frac12\left(
     \begin{array}{c}
      \ve_8-\ve_7+\ve_1-\ve_2-\ve_3-\ve_4-\ve_5-\ve_6\\
        \ve_8-\ve_7-\ve_1+\ve_2-\ve_3-\ve_4-\ve_5-\ve_6 \\
      \ve_8-\ve_7-\ve_1-\ve_2+\ve_3-\ve_4-\ve_5-\ve_6 \\
       \ve_8-\ve_7-\ve_1-\ve_2-\ve_3+\ve_4-\ve_5-\ve_6  \\
        \ve_8-\ve_7-\ve_1-\ve_2-\ve_3-\ve_4+\ve_5-\ve_6 \\
        \ve_8-\ve_7 -\ve_1-\ve_2-\ve_3-\ve_4-\ve_5+\ve_6 \\
     \end{array}
  \right)
   \end{split}
   \ee

 Simple positive roots are: $\a_1=e_1,\a_2=a_1, \a_3=b_1,\a_4=b_6, \a_5=b_{10}, \a_6=b_{13}, \a_7=b_{15}$.  The matrix for $(\alpha_i )$ is:
   \be
   g= \left(
      \begin{array}{cccccccc}
        1/2 & -1/2 & -1/2 &-1/2 & -1/2&-1/2   &-1/2 & 1/2\\
        1 & 1 & 0 & 0 & 0 & 0 & 0 & 0 \\
        -1 & 1 & 0 & 0 & 0 & 0 & 0 & 0 \\
        0 & -1 & 1 & 0 & 0 & 0 & 0 & 0 \\
        0 & 0 & -1 & 1 & 0 & 0 & 0 & 0 \\
        0 & 0 & 0 & -1 & 1 & 0 & 0 & 0 \\
        0 & 0 & 0 & 0 & -1 & 1 & 0 & 0 \\
      \end{array}
    \right)
   \ee

   The coordinates of $(a_i)$ with respect of the ordered basis $\{\a_1,\dots,\a_7\}$ are given by the columns of
   \be(gg^t)^{-1}gA^t=
 \left( \begin {array}{ccccccccccccccc} 0&0&0&0&0&0&0&0&0&0&0&0&0&0&0
\\ \noalign{\medskip}1&1&1&1&1&1&1&1&1&1&1&1&1&1&1
\\ \noalign{\medskip}0&0&0&0&0&1&1&1&1&1&1&1&1&1&1
\\ \noalign{\medskip}0&1&1&1&1&1&1&1&1&2&2&2&2&2&2
\\ \noalign{\medskip}0&0&1&1&1&0&1&1&1&1&1&1&2&2&2
\\ \noalign{\medskip}0&0&0&1&1&0&0&1&1&0&1&1&1&1&2
\\ \noalign{\medskip}0&0&0&0&1&0&0&0&1&0&0&1&0&1&1\end {array}
 \right)
   \ee
The coordinates of $(b_i)$ with respect of the ordered basis $\{\a_1,\dots,\a_7\}$ are given by the columns of

\be(gg^t)^{-1}gB^t=
 \left( \begin {array}{cccccccccccccccc} 0&0&0&0&0&0&0&0&0&0&0&0&0&0&0
&2\\ \noalign{\medskip}0&0&0&0&0&0&0&0&0&0&0&0&0&0&0&2
\\ \noalign{\medskip}1&1&1&1&1&0&0&0&0&0&0&0&0&0&0&3
\\ \noalign{\medskip}0&1&1&1&1&1&1&1&1&0&0&0&0&0&0&4
\\ \noalign{\medskip}0&0&1&1&1&0&1&1&1&1&1&1&0&0&0&3
\\ \noalign{\medskip}0&0&0&1&1&0&0&1&1&0&1&1&1&1&0&2
\\ \noalign{\medskip}0&0&0&0&1&0&0&0&1&0&0&1&0&1&1&1\end {array}
 \right)
\ee The coordinates of $(c_i)$ with respect of the ordered basis $\{\a_1,\dots,\a_7\}$ are given by the columns of

\be(gg^t)^{-1}gC^t=
 \left( \begin {array}{cccccc} 1&1&1&1&1&1\\ \noalign{\medskip}2&2&2&2
&2&2\\ \noalign{\medskip}3&2&2&2&2&2\\ \noalign{\medskip}4&4&3&3&3&3
\\ \noalign{\medskip}3&3&3&2&2&2\\ \noalign{\medskip}2&2&2&2&1&1
\\ \noalign{\medskip}1&1&1&1&1&0\end {array} \right)
\ee
The coordinates of $(d_i)$ with respect of the ordered basis $\{\a_1,\dots,\a_7\}$ are given by the columns of $(gg^t)^{-1}gD^t$ which is:

\be
 \left(\begin {array}{cccccccccccccccccccc} 1&1&1&1&1&1&1&1&1&1&1&1&1
&1&1&1&1&1&1&1\\ \noalign{\medskip}1&1&1&1&1&1&1&1&1&1&1&1&1&1&1&1&1&1
&1&1\\ \noalign{\medskip}2&2&2&2&2&2&2&2&2&2&1&1&1&1&1&1&1&1&1&1
\\ \noalign{\medskip}3&3&3&3&2&2&2&2&2&2&2&2&2&2&2&2&1&1&1&1
\\ \noalign{\medskip}3&2&2&2&2&2&2&1&1&1&2&2&2&1&1&1&1&1&1&0
\\ \noalign{\medskip}2&2&1&1&2&1&1&1&1&0&2&1&1&1&1&0&1&1&0&0
\\ \noalign{\medskip}1&1&1&0&1&1&0&1&0&0&1&1&0&1&0&0&1&0&0&0
\end {array} \right)
\ee
The coordinates of $(e_i)$ with respect of the ordered basis $\{\a_1,\dots,\a_7\}$ are given by the columns of

\be(gg^t)^{-1}gE^t=
 \left(\begin {array}{cccccc} 1&1&1&1&1&1\\ \noalign{\medskip}0&0&0&0
&0&0\\ \noalign{\medskip}0&1&1&1&1&1\\ \noalign{\medskip}0&0&1&1&1&1
\\ \noalign{\medskip}0&0&0&1&1&1\\ \noalign{\medskip}0&0&0&0&1&1
\\ \noalign{\medskip}0&0&0&0&0&1\end {array} \right)
\ee

Since $\a$ is a root implies $|\a|^2=2$, it is easy to see that Lemma \ref{l-N-E} is still true in this case. Note that $(E_7,\a_7)$ is Hermitian symmetric.

\subsubsection{The space $(E_7,\a_1)$}

\be
\begin{split}
\Delta_1^+(1)=&\{c_1,..,c_6, d_1,..,d_{20}, e_1,..,e_6\}\\
\Delta_1^+(2)= &\{b_{16}\}\\
\end{split}
\ee
$\Delta_1^+(k)=\emptyset$ for $k\ge 3.$

\bee
\left\{
  \begin{array}{ll}
    \dim=33, \Ric=17 g, &   \\
    \text{\rm 4 largest eigenvalues of $M_1$ are 17, 7.5, 7.5, 7.5}, & \\
\text{\rm eigenvalues of $M_2$ are at most $7.5$ (using Lemma \ref{l-eigenvalue})}.
  \end{array}
\right.
\eee
 Thus the space has $QB> 0$.

 \subsubsection{The space $(E_7,\a_2)$}
 \be
\begin{split}
\Delta_2^+(1)=&\{a_1,..,a_{15}, d_1,..,d_{20}\}\\
\Delta_2^+(2)= &\{c_1,..,c_6, b_{16}\},\\
\end{split}
\ee
$\Delta_2^+(k)=\emptyset$ for $k\ge 3.$
\bee
\left\{
  \begin{array}{ll}
    \dim=42, \Ric=14 g, &   \\
    \text{\rm 4 largest eigenvalues of $M_1$ are 14, 8.012, 8.012, 8.012}, & \\
\text{\rm eigenvalues of $M_2$ are at most $ 9$ (using Lemma \ref{l-eigenvalue})}.
  \end{array}
\right.
\eee

  Thus the space has $QB> 0$.

\subsubsection{The space $(E_7,\a_3)$}

 \be
\begin{split}
\Delta_3^+(1)=&\{a_6,..,a_{15}, b_1,..,b_5, d_{11},..,d_{20}, e_2,..,e_6\}\\
\Delta_3^+(2)=&\{d_1,..,d_{10}, c_2,..,c_6\}\\
\Delta_3^+(3)=&\{b_{16}, c_1\}\\
\end{split}
\ee
$\Delta_3^+(k)=\emptyset$ for $k\ge 4.$

\bee
\left\{
  \begin{array}{ll}
    \dim=47, \Ric=11 g, &   \\
    \text{\rm 4 largest eigenvalues of $M_1$ are 12.1411, 11, 7.6829, 7.6829}.
  \end{array}
\right.
\eee
  Thus the space does not have $QB\geq 0$.

\subsubsection{The space $(E_7,\a_4)$}

 \be
\begin{split}
\Delta_4^+(1)=&\{a_2,..,a_9, b_2,..,b_9, d_{17},..,d_{20}, e_3,..,e_6\}\\
\Delta_4^+(2)= &\{a_{10},..,a_{15}, d_5,..,d_{16}\}\\
\Delta_4^+(3)= &\{c_3,..,c_6, d_1,..,d_4\}\\
\Delta_4^+(4)= &\{b_{16},c_1, c_2\}\\
\end{split}
\ee
$\Delta_4^+(k)=\emptyset$ for $k\ge 5.$

\bee
\left\{
  \begin{array}{ll}
    \dim=53,\ \Ric=8 g, &   \\
    \text{\rm 4 largest eigenvalues of $M_1$ are 9.5692, 8.1727, 8.1727, 8}.
  \end{array}
\right.
\eee
  Thus the space does not have $QB\geq 0$

\subsubsection{
 $(E_7,\a_5)$}

 \be
\begin{split}
\Delta_5^+(1)= &\{a_3,..,a_5, a_7,..,a_{12}, b_3,..,b_5, b_7,..,b_{12}, d_8,..,d_{10} , d_{14},..,d_{19}, e_4,..,e_6\}\\
\Delta_5^+(2)= &\{a_{13},..,a_{15}, c_4,..,c_6, d_2,..,d_7, d_{11},..,d_{13}\}\\
\Delta_5^+(3)= &\{b_{16}, c_1, c_2, c_3, d_1\}\\
\end{split}
\ee
$\Delta_5^+(k)=\emptyset$ for $k\ge 4.$

\bee
\left\{
  \begin{array}{ll}
    \dim=50, \Ric=10 g, &   \\
    \text{\rm 4 largest eigenvalues of $M_1$ are 10, 9.7882, 9.7882, 8.0097}, & \\
\text{\rm eigenvalues of $M_2$ are less than $10$}.
  \end{array}
\right.
\eee
(the estimate for $M_2$ is obtained by using $\mu=10$ and $s=1$ in Lemma \ref{weightedrowsumsZ} in which case the maximum weighted row sum is $9.9806$).

  Thus the space has $QB> 0$.
\subsubsection{The space $(E_7,\a_6)$}
\be
\begin{split}
\Delta_6^+(1)=&\{ a_4, a_5, a_8, a_9, a_{11},..,a_{14}, b_4, b_5, b_8, b_9, b_{11},..,b_{14}, c_5, c_6, d_3, d_4, d_6,..,d_9, \\
&d_{12},..,d_{15}, e_5, e_6, d_{17}, d_{18}\}\\
\Delta_6^+(2)=& \{a_{15}, b_{16}, c_1, c_2, c_3, c_4, d_1, d_2, d_5, d_{11} \} \\
\end{split}
\ee
$\Delta_6^+(k)=\emptyset$ for $k\ge 3.$
 \bee
\left\{
  \begin{array}{ll}
     \dim=42,\ \Ric=13 g , &   \\
    \text{\rm 4 largest eigenvalues of $M_1$ are 13.5, 13, 7.1504, 7.1504}.
  \end{array}
\right.
\eee
Thus the space does not have $QB\geq 0$.

\subsection{The spaces $(E_8,\a_i)$}

\subsubsection{Root system}
Let $V=\R^8$.  The positive roots in $V$ are $\pm \ve_i+\ve_j, 1\le i<j\le 8$ (total 56), and
$$
\frac12(\ve_8+\sum_{i=1}^7(-1)^{\nu(i)}\ve_i)%=\frac12(\ve_8 +\sum_{i=1}^7\pm\ve_i)
$$
so that $\sum_{i=1}^7 \nu(i)$ is even, i.e. the number of minus sign is even (total 21+35+8=64).  Let
%\vspace{200pt}
\be
\begin{split}
A=&\left(
    \begin{array}{c}
      a_i
    \end{array}
  \right)_{i=1}^{28}=
 \left(
     \begin{array}{c}
       \ve_1+\ve_2 \\
        \ve_1+\ve_3 \\
      \ve_1+\ve_4 \\
      \ve_1+\ve_5 \\
      \ve_1+\ve_6 \\
      \ve_1+\ve_7 \\
      \ve_1+\ve_8 \\
       \ve_2+\ve_3 \\
       \ve_2+\ve_4 \\
       \ve_2+\ve_5\\
       \ve_2+\ve_6 \\
       \ve_2+\ve_7\\
       \ve_2+\ve_8 \\
       \ve_3+\ve_4 \\
       \ve_3+\ve_5 \\
       \ve_3+\ve_6 \\
       \ve_3+\ve_7\\
       \ve_3+\ve_8 \\
       \ve_4+\ve_5 \\
       \ve_4+\ve_6 \\
       \ve_4+\ve_7 \\
       \ve_4+\ve_8 \\
       \ve_5+\ve_6 \\
       \ve_5+\ve_7 \\
       \ve_5+\ve_8 \\
       \ve_6+\ve_7 \\
       \ve_6+\ve_8 \\
       \ve_7+\ve_8 \\
     \end{array}
   \right)\\
   \end{split},
\begin{split}
B=&\left(
    \begin{array}{c}
      b_i
    \end{array}
  \right)_{i=1}^{28}=
 \left(
     \begin{array}{c}
       -\ve_1+\ve_2 \\
        -\ve_1+\ve_3 \\
      -\ve_1+\ve_4 \\
      -\ve_1+\ve_5 \\
      -\ve_1+\ve_6 \\
      -\ve_1+\ve_7 \\
      -\ve_1+\ve_8 \\
       -\ve_2+\ve_3 \\
       -\ve_2+\ve_4 \\
       -\ve_2+\ve_5\\
       -\ve_2+\ve_6 \\
       -\ve_2+\ve_7\\
       -\ve_2+\ve_8 \\
       -\ve_3+\ve_4 \\
       -\ve_3+\ve_5 \\
       -\ve_3+\ve_6 \\
       -\ve_3+\ve_7\\
       -\ve_3+\ve_8 \\
       -\ve_4+\ve_5 \\
       -\ve_4+\ve_6 \\
      - \ve_4+\ve_7 \\
       -\ve_4+\ve_8 \\
       -\ve_5+\ve_6 \\
       -\ve_5+\ve_7 \\
       -\ve_5+\ve_8 \\
       -\ve_6+\ve_7 \\
       -\ve_6+\ve_8 \\
       -\ve_7+\ve_8 \\
     \end{array}
   \right)\\
   \end{split}
   \ee
\newpage

 \be
   \begin{split}
C=&\left(
    \begin{array}{c}
      c_i \\
    \end{array}
  \right)_{i=1}^{21}=
 \left(
     \begin{array}{c}
      \ve_8-\ve_1-\ve_2+\ve_3+\ve_4+\ve_5+\ve_6+\ve_7 \\
        \ve_8 -\ve_1+\ve_2-\ve_3+\ve_4+\ve_5+\ve_6+\ve_7  \\
      \ve_8 -\ve_1+\ve_2+\ve_3-\ve_4+\ve_5+\ve_6+\ve_7  \\
       \ve_8 -\ve_1+\ve_2+\ve_3+\ve_4-\ve_5+\ve_6+\ve_7  \\
       \ve_8 -\ve_1+\ve_2+\ve_3+\ve_4+\ve_5-\ve_6+\ve_7  \\
       \ve_8 -\ve_1+\ve_2+\ve_3+\ve_4+\ve_5+\ve_6-\ve_7  \\
     \ve_8+\ve_1-\ve_2-\ve_3+\ve_4+\ve_5+\ve_6+\ve_7  \\
        \ve_8 +\ve_1-\ve_2+\ve_3-\ve_4+\ve_5+\ve_6+\ve_7  \\
       \ve_8 +\ve_1-\ve_2+\ve_3+\ve_4-\ve_5+\ve_6+\ve_7 \\
      \ve_8 +\ve_1-\ve_2+\ve_3+\ve_4+\ve_5-\ve_6+\ve_7 \\
      \ve_8 +\ve_1-\ve_2+\ve_3+\ve_4+\ve_5+\ve_6-\ve_7 \\
       \ve_8 +\ve_1+\ve_2-\ve_3-\ve_4+\ve_5+\ve_6+\ve_7 \\
        \ve_8 +\ve_1+\ve_2-\ve_3+\ve_4-\ve_5+\ve_6+\ve_7  \\
        \ve_8 +\ve_1+\ve_2-\ve_3+\ve_4+\ve_5-\ve_6+\ve_7  \\
        \ve_8 +\ve_1+\ve_2+\ve_3+\ve_4+\ve_5 +\ve_6-\ve_7 \\
        \ve_8 +\ve_1+\ve_2+\ve_3-\ve_4-\ve_5 +\ve_6+\ve_7 \\
        \ve_8 +\ve_1+\ve_2+\ve_3-\ve_4+\ve_5 -\ve_6+\ve_7 \\
        \ve_8 +\ve_1+\ve_2+\ve_3-\ve_4+\ve_5 +\ve_6-\ve_7 \\
        \ve_8 +\ve_1+\ve_2+\ve_3+\ve_4-\ve_5 -\ve_6+\ve_7 \\
        \ve_8 +\ve_1+\ve_2+\ve_3+\ve_4-\ve_5 +\ve_6-\ve_7 \\
        \ve_8 +\ve_1+\ve_2+\ve_3+\ve_4+\ve_5 -\ve_6-\ve_7 \\
     \end{array}
   \right) \\
   \end{split}
   \ee

\newpage

 \be
   \begin{split}
D=&\left(
    \begin{array}{c}
      d_i
    \end{array}
 \right)_{i=1}^{35} =
 \frac12\left(
     \begin{array}{c}
        \ve_8-\ve_1-\ve_2-\ve_3-\ve_4+\ve_5+\ve_6+\ve_7  \\
        \ve_8-\ve_1-\ve_2-\ve_3+\ve_4-\ve_5+\ve_6+\ve_7  \\
        \ve_8-\ve_1-\ve_2-\ve_3+\ve_4+\ve_5-\ve_6+\ve_7  \\
        \ve_8-\ve_1-\ve_2-\ve_3+\ve_4+\ve_5+\ve_6-\ve_7  \\
      \ve_8 -\ve_1-\ve_2+\ve_3-\ve_4-\ve_5+\ve_6+\ve_7\\
      \ve_8 -\ve_1-\ve_2+\ve_3-\ve_4+\ve_5-\ve_6+\ve_7\\
      \ve_8 -\ve_1-\ve_2+\ve_3-\ve_4+\ve_5+\ve_6-\ve_7\\
       \ve_8 -\ve_1-\ve_2+\ve_3+\ve_4-\ve_5-\ve_6 +\ve_7 \\
        \ve_8 -\ve_1-\ve_2+\ve_3+\ve_4-\ve_5+\ve_6 -\ve_7 \\
        \ve_8 -\ve_1-\ve_2+\ve_3+\ve_4+\ve_5-\ve_6-\ve_7 \\
        \ve_8 -\ve_1+\ve_2-\ve_3-\ve_4-\ve_5+\ve_6+\ve_7  \\
        \ve_8 -\ve_1+\ve_2-\ve_3-\ve_4+\ve_5-\ve_6+\ve_7  \\
        \ve_8 -\ve_1+\ve_2-\ve_3-\ve_4+\ve_5+\ve_6-\ve_7  \\
        \ve_8 -\ve_1+\ve_2-\ve_3+\ve_4-\ve_5-\ve_6+\ve_7   \\
        \ve_8 -\ve_1+\ve_2-\ve_3+\ve_4-\ve_5+\ve_6-\ve_7   \\
        \ve_8 -\ve_1+\ve_2-\ve_3+\ve_4+\ve_5-\ve_6-\ve_7   \\
        \ve_8 -\ve_1+\ve_2+\ve_3-\ve_4-\ve_5-\ve_6 +\ve_7  \\
        \ve_8 -\ve_1+\ve_2+\ve_3-\ve_4-\ve_5+\ve_6 -\ve_7  \\
         \ve_8 -\ve_1+\ve_2+\ve_3-\ve_4+\ve_5-\ve_6-\ve_7   \\
  \ve_8 -\ve_1+\ve_2+\ve_3+\ve_4-\ve_5-\ve_6-\ve_7   \\
  \ve_8  +\ve_1-\ve_2-\ve_3-\ve_4-\ve_5+\ve_6 +\ve_7 \\
  \ve_8  +\ve_1-\ve_2-\ve_3-\ve_4+\ve_5-\ve_6 +\ve_7 \\
  \ve_8  +\ve_1-\ve_2-\ve_3-\ve_4+\ve_5+\ve_6 -\ve_7 \\
         \ve_8 +\ve_1-\ve_2-\ve_3+\ve_4-\ve_5-\ve_6+\ve_7  \\
         \ve_8 +\ve_1-\ve_2-\ve_3+\ve_4-\ve_5+\ve_6-\ve_7  \\
          \ve_8 +\ve_1-\ve_2-\ve_3+\ve_4+\ve_5-\ve_6 -\ve_7 \\
          \ve_8 +\ve_1-\ve_2+\ve_3-\ve_4-\ve_5-\ve_6 +\ve_7 \\
          \ve_8 +\ve_1-\ve_2+\ve_3-\ve_4-\ve_5+\ve_6 -\ve_7 \\
          \ve_8 +\ve_1-\ve_2+\ve_3-\ve_4+\ve_5-\ve_6 -\ve_7 \\
          \ve_8 +\ve_1-\ve_2+\ve_3+\ve_4-\ve_5-\ve_6 -\ve_7 \\
          \ve_8 +\ve_1+\ve_2-\ve_3-\ve_4-\ve_5-\ve_6+\ve_7 \\
          \ve_8 +\ve_1+\ve_2-\ve_3-\ve_4-\ve_5+\ve_6-\ve_7 \\
          \ve_8 +\ve_1+\ve_2-\ve_3-\ve_4+ \ve_5-\ve_6-\ve_7 \\
          \ve_8 +\ve_1+\ve_2-\ve_3+\ve_4-\ve_5-\ve_6-\ve_7 \\
           \ve_8 +\ve_1+\ve_2+\ve_3-\ve_4-\ve_5-\ve_6-\ve_7 \\
     \end{array}
  \right)  \\
   \end{split}
   \ee
\newpage

    \be
   \begin{split}
E=&\left(
    \begin{array}{c}
      e_i\\
    \end{array}
  \right)_{i=1}^8=
 \frac12\left(
     \begin{array}{c}
     \ve_8+\ve_1+\ve_2+\ve_3+\ve_4+\ve_5+\ve_6+\ve_7\\
      \ve_8+\ve_1-\ve_2-\ve_3-\ve_4-\ve_5-\ve_6-\ve_7\\
        \ve_8-\ve_1+\ve_2-\ve_3-\ve_4-\ve_5-\ve_6 -\ve_7\\
      \ve_8 -\ve_1-\ve_2+\ve_3-\ve_4-\ve_5-\ve_6-\ve_7 \\
       \ve_8 -\ve_1-\ve_2-\ve_3+\ve_4-\ve_5-\ve_6-\ve_7  \\
        \ve_8 -\ve_1-\ve_2-\ve_3-\ve_4+\ve_5-\ve_6 -\ve_7\\
        \ve_8  -\ve_1-\ve_2-\ve_3-\ve_4-\ve_5+\ve_6 -\ve_7\\
        \ve_8  -\ve_1-\ve_2-\ve_3-\ve_4-\ve_5-\ve_6 +\ve_7
     \end{array}
   \right)\\
   \end{split}
   \ee

   Simple positive roots are: $\a_1=e_2, \a_2=a_1, \a_3=b_1,\a_4=b_8, \a_5=b_{14}, \a_6=b_{19}, \a_7=b_{23}, \a_8=b_{26}.$  The matrix for $(\a_i )$  is:
   \be
   g= \left(
      \begin{array}{cccccccc}
        1/2 & -1/2 & -1/2 &-1/2 & -1/2&-1/2   &-1/2 & 1/2\\
        1 & 1 & 0 & 0 & 0 & 0 & 0 & 0 \\
        -1 & 1 & 0 & 0 & 0 & 0 & 0 & 0 \\
        0 & -1 & 1 & 0 & 0 & 0 & 0 & 0 \\
        0 & 0 & -1 & 1 & 0 & 0 & 0 & 0 \\
        0 & 0 & 0 & -1 & 1 & 0 & 0 & 0 \\
        0 & 0 & 0 & 0 & -1 & 1 & 0 & 0 \\
        0 & 0 & 0 & 0 & 0 & -1 & 1 & 0 \\
      \end{array}
   \right)
   \ee
\vspace{100pt}

The coordinates of $(a_i)$ relative to the ordered basis $\{\a_1,\dots,\a_8\}$ are given by the rows of    \be\lf((gg^t)^{-1}gA^t\ri)^t=
   \left(
     \begin{array}{cccccccc}
         0 &   1 &  0   &0 &  0&     0 &        0      &   0\\
         0  &  1     &     0&1 &0&0&0&0\\
         0&1 &0&1 &1 &0&0&0\\
    0 &1 &0 &1 &1 &1 &0&0\\
   0  &1 &    0  &1 &1 &1 &1 &0\\
   0  &1  &   0 &1 &1 &1 &1 &1 \\
    2 &3 &3 &5 &4 &3 &2 &1 \\
         0&1 &1 &1 &0&0&0&0\\
         0&1 &1 &1 &1 &0&0&0\\
    0 &1 &1 &1 &1 &1 &0&0\\
   0  &1 &1 &1 &1 &1 &1 &0\\
   0  &1 &1 &1 &1 &1 &1 &1 \\
 2 &3 &4 &5 &4 &3 &2 &1 \\
 0&1 &1 &2 &1 &0&0&0\\
 0 &1 &1 &2 &1 &1 &0&0\\
   0  &1 &1 &2 &1 &1 &1 &0\\
   0  &1 &1 &2 &1 &1 &1 &1 \\
 2 &3 &4 &6 &4 &3 &2 &1 \\
 0 &1 &1 &2 &2 &1 &0&0\\
   0  &1 &1 &2 &2 &1 &1 &0\\
   0  &1 &1 &2 &2 &1 &1 &1 \\
 2 &3 &4 &6 &5 &3 &2 &1 \\
 0 &1 &1 &2 &2 &2 &1 &0\\
 0 &1 &1 &2 &2 &2 &1 &1 \\
 2 &3 &4 &6 &5 &4 &2 &1 \\
   0  &1 &1 &2 &2 &2 &2 &1 \\
 2 &3 &4 &6 &5 &4 &3 &1 \\
 2 &3 &4 &6 &5 &4 &3 &2 \\
  \end{array}
   \right)
   \ee
\vspace{100pt}

   The coordinates of $(b_i)$ relative the ordered basis $\{\a_1,\dots,\a_8\}$ are given by the rows of   \be\lf((gg^t)^{-1}gB^t\ri)^t=
 \left(
  \begin{array}{cccccccc}
 0 &  0  &1  &  0  &0&0&0&0\\
  0&0&1 &1 &0&0&0&0\\
  0&0&1 &1 &1 &0&0&0\\
 0 &0 &1 &1 &1 &1 &0&0\\
   0  &   0  &1 &1 &1 &1 &1 &0\\
   0  &   0  &1 &1 &1 &1 &1 &1 \\
 2 &2 &4 &5 &4 &3 &2 &1 \\
  0&0 &0&1 &0&0&0&0\\
  0&0 &0&1 &1 &0&0&0\\
 0 &0 &0 &1 &1 &1 &0&0\\
   0  &0  &  0  &1 &1 &1 &1 &0\\
   0  &0   & 0  &1 &1 &1 &1 &1 \\
 2 &2 &3 &5 &4 &3 &2 &1 \\
  0&0&0&0&1 &0&0&0\\
 0 &0 &0 &0 &1 &1 &0&0\\
   0  &   0  &   0  &   0  &1 &1 &1 &0\\
   0  &   0   &  0   &  0  &1 &1 &1 &1 \\
 2 &2 &3 &4 &4 &3 &2 &1 \\
 0 &0 &0 &0 &   0  &1 &0&0\\
   0  &   0 &    0 &    0 &    0  &1 &1 &0\\
   0  &   0   &  0   &  0   &  0  &1 &1 &1 \\
 2 &2 &3 &4 &3 &3 &2 &1 \\
   0  &   0  &   0 &    0 &    0 &    0  &1 &0\\
   0  &   0   &  0   &  0  &   0  &   0  &1 &1 \\
2 &2 &3 &4 &3 &2 &2 &1 \\
 0&0&0&0&0&0  & 0  &1 \\
2 &2 &3 &4 &3 &2 &1 &1 \\
2 &2 &3 &4 &3 &2 &1  &  0  \\
  \end{array}
\right)
\ee
\vspace{100pt}

The coordinates of $(c_i)$ relative the ordered basis $\{\a_1,\dots,\a_8\}$ are given by the rows of   \be\lf((gg^t)^{-1}gC^t\ri)^t=
 \left(
  \begin{array}{cccccccc}
 1 &2 &3 &5 &4 &3 &2 &1\\
 1 &2 &3 &4 &4 &3 &2 &1\\
 1 &2 &3 &4 &3 &3 &2 &1\\
 1 &2 &3 &4 &3 &2 &2 &1\\
 1 &2 &3 &4 &3 &2 &1 &1\\
 1 &2 &3 &4 &3 &2 &1  &  0\\
 1 &2 &2 &4 &4 &3 &2 &1\\
 1 &2 &2 &4 &3 &3 &2 &1\\
 1 &2 &2 &4 &3 &2 &2 &1\\
1 &2 &2 &4 &3 &2 &1 &1\\
1 &2 &2 &4 &3 &2 &1  &  0\\
1 &2 &2 &3 &3 &3 &2 &1\\
1 &2 &2 &3 &3 &2 &2 &1\\
1 &2 &2 &3 &3 &2 &1 &1\\
1 &2 &2 &3 &3 &2 &1 &   0\\
1 &2 &2 &3 &2 &2 &2 &1\\
1 &2 &2 &3 &2 &2 &1 &1\\
1 &2 &2 &3 &2 &2 &1 &   0\\
1 &2 &2 &3 &2 &1 &1 &1\\
1 &2 &2 &3 &2 &1 &1  &  0\\
1 &2 &2 &3 &2 &1   & 0 &    0\\
  \end{array}
\right)
\ee
\vspace{100pt}

The coordinates of $(d_i)$ relative to the ordered basis $\{\a_1,\dots,\a_8\}$ are given by the rows of
\be\lf((gg^t)^{-1}gD^t\ri)^t=
 \left(
  \begin{array}{cccccccc}
    1&1&2&3&3&3&2&1\\
    1&1&2&3&3&2&2&1\\
    1&1&2&3&3&2&1&1\\
    1&1&2&3&3&2&1&0\\
    1&1&2&3&2&2&2&1\\
    1&1&2&3&2&2&1&1\\
    1&1&2&3&2&2&1 &0\\
    1&1&2&3&2&1&1&1\\
    1&1&2&3&2&1&1&0\\
    1&1&2&3&2&1&0&0\\
    1&1&2&2&2&2&2&1\\
    1&1&2&2&2&2&1&1\\
    1&1&2&2&2&2&1&0\\
    1&1&2&2&2&1&1&1\\
    1&1&2&2&2&1&1&0\\
    1&1&2&2&2&1&0&0\\
    1&1&2&2&1&1&1&1\\
    1&1&2&2&1&1&1&0\\
    1&1&2&2&1&1&0&0\\
    1&1&2&2&1&0&0&0\\
    1&1&1&2&2&2&2&1\\
    1&1&1&2&2&2&1&1\\
    1&1&1&2&2&2&1&0\\
    1&1&1&2&2&1&1&1\\
    1&1&1&2&2&1&1&0\\
    1&1&1&2&2&1&0&0\\
    1&1&1&2&1&1&1&1\\
    1&1&1&2&1&1&1&0\\
    1&1&1&2&1&1&0&0\\
    1&1&1&2&1&0&0&0\\
    1&1&1&1&1&1&1&1\\
    1&1&1&1&1&1&1&0\\
    1&1&1&1&1&1&0&0\\
    1&1&1&1&1&0&0&0\\
    1&1&1&1&0&0&0&0 \\
  \end{array}
\right)
\ee
\vspace{50pt}

The coordinates of $(e_i)$ relative the ordered basis $\{\a_1,\dots,\a_8\}$ are given by the rows of
\be\lf((gg^t)^{-1}gE^t\ri)^t=
 \left(
  \begin{array}{cccccccc}
    1  &  3&3&5&4&3&2&1\\
1& 0 &  0&0  & 0&0  & 0  & 0\\
 1   &0&1 &  0  & 0&0&0&0\\
 1&0&1&1&0&0&0&0\\
 1&0&1&1&1&0&0&0\\
 1&0&1&1&1&1&0&0\\
 1&0&1&1&1&1&1&0\\
 1&0&1&1&1&1&1&1
  \end{array}
\right)
\ee

Since $\a$ is a root implies that $|\a|^2=2$, Lemma \ref{l-N-E} is still true in this case.

\subsubsection{The space $(E_8,\a_1)$}

\be
\begin{split}
\Delta_1^+(1)=&\{ c_1,\dots,c_{21}; d_1,\dots,d_{35}; e_1,\dots,e_8\}\\
\Delta_1^+(2)= &\{a_7, a_{13}, a_{18},  a_{22}, a_{25},a_{27}, a_{28}; b_7, b_{13}, b_{18},  b_{22}, b_{25},b_{27}, b_{28}\}\\
\end{split}
\ee
$\Delta^+_1(k)=\emptyset$ for $k\ge 3.$

\bee
\left\{
  \begin{array}{ll}
    \dim=78,\ \Ric=23 g, &   \\
    \text{\rm 4 largest eigenvalues of $M_1$ are 14.1102,
   14.1102,
   14.1102, 23}, & \\
\text{\rm eigenvalues of $M_2$ are at most $14.5$ (using Lemma \ref{l-eigenvalue})}.
  \end{array}
\right.
\eee

 Thus the space has $QB> 0$.

 \subsubsection{The space $(E_8,\a_2)$}
 \be
\begin{split}
\Delta_2^+(1)=&\{ a_1,..., a_6, a_8,...,, a_{12},  a_{14},...,, a_{17}, a_{19},...,a_{21}, a_{23}, a_{24},a_{26}; d_1,\dots,d_{35}\}\\
\Delta_2^+(2)=&   \{b_7, b_{13}, b_{18},  b_{22}, b_{25},b_{27}, b_{28}, c_1,...,c_{21}\}\\
\Delta_2^+(3)=&\{a_7, a_{13}, a_{18},  a_{22}, a_{25},a_{27}, a_{28}; e_1\}\\
\end{split}
\ee
$\Delta^+_2(k)=\emptyset$ for $k\ge 4.$
\bee
\left\{
  \begin{array}{ll}
    \dim=92,\ \Ric=17g, &   \\
    \text{\rm 4 largest eigenvalues of $M_1$ are 13.8336,
   13.8336,
   13.8336,
   17.0000}, & \\
\text{\rm eigenvalues of $M_2$ are at most $15$ (using Lemma \ref{l-eigenvalue})}.
  \end{array}
\right.
\eee
Thus the space has $QB> 0$.

\subsubsection{The space $(E_8,\a_3)$}

\be
\begin{split}
\Delta_3^+(1)=&\{  a_8,..., a_{12},  a_{14},..., a_{17}, a_{19},...,a_{21}, a_{23}, a_{24},a_{26}; b_1,\dots, b_6;\\
&  d_{21},...,d_{35}; e_3,..., e_8\}\\
\Delta_3^+(2)= & \{c_7,...,c_{21};d_1,...,d_{20}\}\\
\Delta_3^+(3)= &\{a_7;   b_{13}, b_{18},  b_{22}, b_{25},b_{27}, b_{28}; c_1,...,c_6; e_1\}\\
\Delta_3^+(4)= & \{a_{13}, a_{18},  a_{22}, a_{25},a_{27}, a_{28}; b_7\}.
\end{split}
\ee
$\Delta^+_3(k)=\emptyset$ for $k\ge 5.$
\bee
\left\{
  \begin{array}{ll}
    \dim=98,\ \Ric=13g, &   \\
    \text{\rm 4 largest eigenvalues of $M_1$ are 11.3117,
   11.3117,
   13.0000,
   16.9627.}.
  \end{array}
\right.
\eee
Thus the space   does not have $QB\ge0$.

\subsubsection{The space $(E_8,\a_4)$}

\be
\begin{split}
\Delta_4^+(1)=&\{ a_2,..., a_{6},  a_{8},..., a_{12}; b_2,...,b_6; b_8,...,b_{12}; d_{31},..., d_{35}; e_4,..., e_8\}\\
\Delta_4^+(2)= &\{a_{14},..., a_{17}, a_{19},..., a_{21},a_{23}, a_{24},  a_{26}; d_{11},..., d_{30}\}\\
\Delta_4^+(3)=
&\{c_{12},...,c_{21}; d_1,...,d_{10}\}\\
\Delta_4^+(4)= &\{b_{18}, b_{22}, b_{25}, b_{27}, b_{28}; c_2,...,c_{11}\}\\
\Delta_4^+(5)= &\{a_7, a_{13}, b_7, b_{13}, c_1,e_1\}\\
\Delta_4^+(6)= &\{a_{18}, a_{22}, a_{25}, a_{27}, a_{28}\}.
\end{split}
\ee
$\Delta^+_4(k)=\emptyset$ for $k\ge 7.$
\bee
\left\{
  \begin{array}{ll}
   \dim=106,\ \Ric=9g, &   \\
    \text{\rm 4 largest eigenvalues of $M_1$ are 9.0798,
   11.2147,
   11.2147,
   12.6168}.
  \end{array}
\right.
\eee
 Thus the   space   does not have $QB\ge0$.

\subsubsection{The space $(E_8,\a_5)$}
\be
\begin{split}
\Delta_5^+(1)=&\{a_3,....,a_6, a_9,...,a_{12},  a_{14},...,a_{17}\\
&b_3,..., b_6,  b_9,...,b_{12}, b_{14},...,b_{17};d_{17},...,d_{20},d_{27},...,d_{34}, e_5,...,e_8\}\\
\Delta_5^+(2)=& \{a_{19},..., a_{21}, a_{23},a_{24},a_{26};c_{16},...,c_{21}\\
&d_5,...,d_{16}, d_{21},...,d_{26}\}\\
\Delta_5^+(3)=&\{b_{22},b_{25},b_{27},b_{28}; c_3,...,c_6, c_8,...,c_{15},d_1,d_2,d_3,d_4\}\\
\Delta_5^+(4)=&\{ a_7, a_{13}, a_{18};  b_{7},b_{13},b_{18}, c_1, c_2,c_7;e_1\}\\
\Delta_5^+(5)=&\{a_{22}, a_{25}, a_{27},a_{28}\}.
\end{split}
\ee
$\Delta^+_5(k)=\emptyset$ for $k\ge 6.$
\bee
\left\{
  \begin{array}{ll}
    \dim=104,\ \Ric=11g, &   \\
    \text{\rm 4 largest eigenvalues of $M_1$ are 11.5575,
   12.0012,
   12.0012,
   12.0012}.
  \end{array}
\right.
\eee

Thus  the space does not have $QB\ge0$.

\subsubsection{The space $(E_8,\a_6)$}
\be
\begin{split}
\Delta_6^+(1)=& \{a_4,..., a_6, a_{10},..., a_{12},  a_{15},..., a_{17}, a_{19},..., a_{21};\\
&b_4,...,b_6, b_{10},..., b_{12},b_{15},..., b_{17}, b_{19},...,b_{21}; c_{19},..., c_{21}\\
& d_8 ,...,d_{10}, d_{14},...,d_{19}, d_{24},...,d_{29}, d_{31},..., d_{33},e_6,...,e_8\}\\
\Delta_6^+(2)=& \{a_{23}, a_{24}, a_{26}; b_{25}, b_{27},b_{28}; c_4,..., c_6, c_9,...,  c_{11},   c_{13},...,c_{18};\\
& d_2,...,d_7, d_{11},..., d_{13}, d_{21},..., d_{23}\}\\
\Delta_6^+(3)=&\{ a_7, a_{13},a_{18}, a_{22}, b_7, b_{13}, b_{18}, b_{22}; c_1,..., c_3 , c_7 , c_8 , c_{12},d_1,e_1\}\\
\Delta_6^+(4)=&\{ a_{25}, a_{27}, a_{28}\}
\end{split}
\ee
$\Delta^+_6(k)=\emptyset$ for $k\ge 5.$
\bee
\left\{
  \begin{array}{ll}
    \dim=97,\ \Ric=14g, &   \\
    \text{\rm 4 largest eigenvalues of $M_1$ are 11.4257,
   14.0000,
   16.0721,
   16.0721}.
  \end{array}
\right.
\eee
Thus the space  does not satisfy $QB\ge0$.

\subsubsection{The space $(E_8,\a_7)$}
\be
\begin{split}
\Delta_7^+(1)=& \{a_5,a_6,a_{11},a_{12},   a_{16},a_{17},a_{20}, a_{21};a_{23}, a_{24};\\
&b_5, b_6, b_{11}, b_{12}, b_{16}, b_{17}, b_{20}, b_{21}, b_{23}, b_{24}, b_{27}, b_{28};\\
&  c_5 , c_6 , c_{10} , c_{11} , c_{14} , c_{15}, c_{17},...,c_{20};\\
& d_3 , d_4 , d_6,...,d_9 , d_{12},...,d_{15} , d_{17} , d_{18}, d_{22},...,d_{25} d_{27} d_{28}, d_{31},d_{32},e_7,e_8\}\\
\Delta_7^+(2)=&\{a_7, a_{13}, a_{18},  a_{22}, a_{25}, a_{26}; b_{7}, b_{13}, b_{18}, b_{22},b_{ 25};\\
&c_1,...,c_4, , c_7 , c_8 , c_9 , c_{12} , c_{13}, c_{16}; d_1, d_2 , d_5, d_{11}, d_{21},e_1\}\\
\Delta_7^+(3)=&\{a_{27}, a_{28}\}
\end{split}
\ee
$\Delta^+_7(k)=\emptyset$ for $k\ge 4.$
\bee
\left\{
  \begin{array}{ll}
    \dim=83,\ \Ric=19g, &   \\
    \text{\rm 4 largest eigenvalues of $M_1$ are 11.4093,
   11.4093,
   19.0000,
   22.1376}.
  \end{array}
\right.
\eee
  Thus the space    does not satisfy $QB\ge0$.

\subsubsection{The space $(E_8,\a_8)$}
\be
\begin{split}
\Delta_8^+(1)= &\{a_6, a_7, a_{12}, a_{13}, a_{17}, a_{18}, a_{21},a_{22}, a_{24},a_{25}, a_{26}, a_{27};\\
 &b_{6}, b_{7}, b_{12}, b_{13}, b_{17}, b_{18}, b_{21}, b_{22}, b_{24}, b_{25}, b_{26},  b_{27}\\
 &c_1,...,c_5, c_7,...,c_{10}, c_{12},..., c_{14}, c_{16} , c_{17} , c_{19}\\
 &d_1,...,d_3, d_5 , d_6 , d_8 , d_{11} , d_{12} , d_{14}, d_{17} , d_{21} , d_{22} , d_{24} , d_{27} , d_{31},e_1,e_8\}\\
\Delta_8^+(2)=&\{ a_{28}\}
\end{split}
\ee
$\Delta^+_8(k)=\emptyset$ for $k\ge 3.$
\bee
\left\{
  \begin{array}{ll}
    \dim=57, \ \Ric=29g, &   \\
    \text{\rm 4 largest eigenvalues of $M_1$ are 11.5000,
   11.5000,
   11.5000,
   29.0000}, & \\
\text{\rm eigenvalues of $M_2$ are at most $11.5$ (using Lemma \ref{l-eigenvalue})}.
  \end{array}
\right.
\eee
  Thus the space has $QB> 0$.

\section{Appendix}\label{appendix}

  We illustrate how we initialize the \K C-space $(G_2,\a_2)$ then calculate bisectional curvatures and estimate the eigenvalues of $M_2$ in MAPLE.   The main formulas used to calculate the curvatures will apply to all the other cases.  Actual  MAPLE code below will be italicized.

\subsection{Initializing $(G_2,\a_2)$} We begin by initializing the root system.

1.      Begin by defining the positive root system in $\R^3$, then define $S$ as the set of all positive and negative roots:
       \vspace{10pt}\vspace{10pt}
\begin{it}

\noindent $>$ a1:=[1,-1,0];  a2:=[-1,0,1]; a3:=[0,-1,1]; b1:=[-2,1,1]; b2:=[1,-2,1]; b3:=[-1,-1,2];

\noindent $>$ S:=\{a1,...,b1,...,-a1,... ,-b1,...\}

 \end{it}
 \vspace{10pt}\vspace{10pt}

2.  By expressing the positive roots in terms of the simple positive roots, identify the corresponding sets $\Delta_2^+(i)$ which in this case are just $\Delta_2^+(1)$ and $\Delta_2^+(2)$ (appearing as $m1, m2$ below).  Refer to the ordered elements of $\Delta_2^+$ by either $C(i)$ or $c(i)$ below, depending on how they are used later on.  The function $g(i)$ is $1,2$ depending on whether $C(i)$ is in $m1$ or  $m2$ respectively.

     \vspace{10pt}\vspace{10pt}
\begin{it}

\noindent $>$ m1:=\{a2,a3,b1,b2\};

\noindent $>$ m2 := \{b3\};

\noindent $>$ A := Matrix([a2, a3, b1, b2, b3]);

\noindent $>$ C:=i $->$  convert(Row(A,i), list)

\noindent $>$ c:=i $->$  Row(A,i)

\noindent $>$ g:=i $->$

if evalb(C(i) in m1) then 1

elif evalb(C(i) in m2) then 2 else 0 end if;

 \end{it}
 \vspace{10pt}\vspace{10pt}

\subsection{bisectional curvature formula and matrix}

 Here we compute the matrix $M1$ using Lemma \ref{l-curvature-formula-2}.

1. We first need to define some basic functions appearing in Lemma \ref{l-curvature-formula-2}.  Below $N(i,j)$ calculates $N_{C(i), C(j)}$ while $T(i,j)N(i,j)$ calculates $\tilde{N}_{C(i), C(j)}$.  The function $Nm(i,j)$ calculates $N_{C(i), -C(j)}$ while $Tm(i,j)Nm(i,j)$ calculates $\tilde{N}_{C(i), -C(j)}$ and the MAPLE codes for these are similar.

       \vspace{10pt}\vspace{10pt}
\begin{it}

\noindent $>$ N:= (i,j) $->$

if evalb(C(i)+C(j) in S and C(i)-2*C(j) in S) then 3

elif evalb(C(i)+C(j) in S and C(i)-C(j) in S) then 2

elif evalb(C(i)+C(j) in S) then 1

else 0 end if;
\vspace{10pt}

\noindent $>$ T:= (i,j) $->$

 if evalb(not(N(i,j)=0))

 then sqrt(c(i).c(i))*sqrt(c(i).c(i))/sqrt((c(i)+c(j)).(c(i)+c(j)))

  else 0 end if;

\vspace{10pt}

  \end{it}
 \vspace{10pt}\vspace{10pt}

2. The matrix for bisectional curvature is given by

       \vspace{10pt}\vspace{10pt}
\begin{it}

\noindent $>$ B:= Matrix(5,5, (i,j) $->$
\vspace{10pt}

 \noindent  if $evalb(i\leq j)$ then

\noindent   1/g(j)*(c(i).c(j)+(1/2)*(g(i)/(g(i)+g(j)))*N(i,j)$^2$*T(i,j)$^2$  )
 \vspace{10pt}

  \noindent  else

\noindent   1/g(i)*(c(j).c(i)+(1/2)*(g(j)/(g(j)+g(i)))*N(j,i)$^2$*T(j,i)$^2$  )
 \vspace{10pt}
end if);
\end{it}

\subsection{general curvature formula}

 First we want to compute general curvatures $R(\alpha \b \beta, \gamma, \b \delta)$ where $\alpha \neq \beta$, $\gamma \neq \delta$ and where $\delta=\alpha-\beta+\gamma$ which we may assume by Lemma \ref{l-curvature-formula-3}.    In the rest of the formulas in this subsection, we identify a triple index $(i,j,k)$ with the quadruple of roots $$\alpha=C(k), \beta=C(i),\gamma=C(j),\delta=C(k)-C(i)+C(j).$$  in particular, given any $(i,j,k)$ we always have  $\delta=\alpha-\beta+\gamma$.

 1. The function $allroots(i,j,k)$ returns $true$ or $false$ depending on whether $\delta$ is a root or not.  Moreover, $gd(i,j,k)=l$ if $\delta$ is in $ml$ and is zero otherwise, while $D(i,j,k)=l$ if $\delta=C(l)$ and is zero otherwise.

           \vspace{10pt}\vspace{10pt}
\begin{it}

\noindent $>$ allroots:=(i,j,k) $->$ evalb(C(k)-C(i)+C(j) in (m1 union m2 union m3))
\vspace{10pt}

\noindent $>$ gd:=(i,j,k) $->$

if C(k)-C(i)+C(j) in m1 then 1

elif C(k)-C(i)+C(j) in m2 then 2

.

.

.

else 0;

\noindent $>$ D:=(i,j,k) $->$

if C(k)-C(i)+C(j) =C(1) then 1

elif C(k)-C(i)+C(j) =C(2) then 2

.

.

.

else 0 end if;

\vspace{10pt}

\vspace{10pt}

 \vspace{10pt}\vspace{10pt}

 \end{it}
2.   Now we define the coefficient functions used in Lemma \ref{l-curvature-formula-3}.  Below, $xi(j,k)$ for example is 1 if  $j<k$ and zero otherwise.

         \vspace{10pt}\vspace{10pt}
\begin{it}

\noindent $>$  xi := (k,j) $->$ if $j < k$ then 1 else 0 end if;

\noindent $>$  delta :=  (k,j) $->$  if $k = j$ then 1 else 0 end if;

\noindent $>$  coeff1 :=  (i,j,k,l) $->$ (k-j)*xi(k, j)-k*l/(i+k);

\noindent  $>$  coeff2 :=   (i,j,k,l) $->$ k*xi(i, j)+l*xi(j, i)+l*delta(i, j)*delta(k, l)+(j-k)*xi(k, j);

\vspace{10pt}
\end{it}
 \vspace{10pt}\vspace{10pt}

3.  Finally, given $(i,j,k)$ we apply the formulas above and Lemma \ref{l-curvature-formula-3} to calculate curvature associated to $\alpha=C(k), \beta=C(i),\gamma=C(j),\delta=C(k)-C(i)+C(j).$  When $\alpha\neq \beta$ and $\delta$ is a root then $Rest(i,j,k)$ is the upper estimate for $|R(\alpha, \b \beta, \gamma, \b \delta)|$ obtained by replacing the $N's$ and their coefficients in Lemma \ref{l-curvature-formula-3} by their absolute values.  Otherwise $Rest(i,j,k)=0$.

           \vspace{10pt}\vspace{10pt}
\begin{it}
  \noindent $>$ Rest:=(i,j,k) $->$  \vspace{10pt}
  \noindent

  if (allroots(i,j,k) and not(C(k) = C(i))  then
  \vspace{10pt}
  \noindent

(1/2)*(1/(sqrt(g(k))*sqrt(g(i))*sqrt(g(j))*sqrt(gd(i, j, k)))*
 \vspace{10pt}

($|$(coeff1(g(k), g(i), g(j), gd(i, j, k))$|$*

T(k,j)*N(k,j)*T(i,D(i,j,k))*N(i,D(i,j,k))

+

$|$(coeff2(g(k), g(i), g(j), gd(i, j, k))$|$

*Tm(k,i)*Nm(k,i)*Tm(j,D(i,j,k))*Nm(j,D(i,j,k)))
 \vspace{10pt}

else 0 end if;
 \end{it}

  \vspace{10pt}

  \vspace{10pt}

\subsection{matrix of non-bisectional curvatures}
Now we calculate the $25\times 25$ matrix $Z$ as defined in \eqref{Z}.

 1. First we all ordered pairs of the form $(C(i), C(j))$ in a list of length $25$ using the two commands below.  For example, $LIST[1]$ returns the pair $[C(1), C(1)]$ while $LIST[1][2]$ corresponds to the second component, $C(1)$, of  $LIST[1]$.

  \vspace{10pt}

\begin{it}
 \noindent $>$  AA:= Matrix(5,5, (i,j) $->$ [C(i), C(j)]

   \vspace{10pt}

 \noindent $>$ LIST:= convert(AA, list)
 \end{it}
   \vspace{10pt}\vspace{10pt}

 2.  Below, $l1(i)=j$ provided $LIST[i][1]=C(j)$.  Similarly, $l2(i)=j$ provided $LIST[i][2]=C(j)$ and its MAPLE code is similar.

  \vspace{10pt}

\begin{it}
 \noindent $>$  l1:= i $->$

 \noindent $>$ if LIST[i][1]=C(1) then 1

  \noindent $>$ if LIST[i][1]=C(2) then 2
   \vspace{10pt}

   .

   .

   .
\vspace{10pt}

\end{it}

 3.  Now we calculate the matrix $Z$  as defined in \eqref{Z}.   Associate  $0\leq i,j \leq 25$ to the quadruple $[A,B,C,D]=[LIST[i][1], LIST[i][2],$ $ LIST[j][1], LIST[j][2]]$.  Now if $A\neq B$, $A=D$ and $B=C$, then $Z_{ij}=|R(ABBA)|$.  If not in the previous case and $A\neq B$ and $A-B=D-C$ then $Z_{ij}$ is the upper estimate for $|R(A,B,C,D)|$ obtained by replacing the $N's$ and their coefficients in Lemma \ref{l-curvature-formula-3} by their absolute values (the minimum appearing in the formula below is justified by the curvature identity $R(A,B,C,D)=R(C,B,A,D)$).  If not in the previous cases then $Z_{ij}$ is zero.

  \vspace{10pt}

\begin{it}
 \noindent $>$  Z:= Matrix(25,25, (i,j) $->$

  \noindent if (evalb(not(LIST[i][1]=LIST[i][2]))  and evalb(LIST[i][2]=LIST[j][1])

  and  evalb(LIST[i][1]=LIST[j][2]))

  \noindent then $|$B[l2(i), l1(i)]$|$

   \vspace{10pt}
  \noindent elif (evalb(not(LIST[i][1]=LIST[i][2]))

   and evalb(LIST[i][1]-LIST[i][2]=LIST[j][2]-LIST[j][1])

  \noindent then evalf(min(Rest(l2(i), l1(j), l1(i)),Rest(l2(i), l1(i), l1(j))))
  \vspace{10pt}

  \noindent else 0 end if);

 \end{it}
   \vspace{10pt}\vspace{10pt}

 Now for the matrix $Z$, the weighted row sums in \eqref{3} of Lemma \ref{weightedrowsumsZ} are given by $S0, S1, S2,..$ in the following commands.
    \vspace{10pt}\vspace{10pt}

\begin{it}

    \noindent $>$b0:=Matrix((25,1), (i,j) $->$1):

       \noindent $>$S0:=max(Z.b0)
   \vspace{10pt}

  \noindent $>$v1:=1/9*Z.b0

    \noindent $>$b1:=Matrix((25,1), (i,j) $->$ min(1, v1[i,1])):

       \noindent $>$S1:=max(Z.b1)
  \vspace{10pt}

  \noindent $>$v2:=1/9* Z.b1

    \noindent $>$b2:=Matrix((25,1), (i,j) $->$ min(1, v2[i,1])):

       \noindent $>$S2:=max(Z.b2)

 .

 .

.

 \end{it}

\subsection{The matrices $B$, $Z$ for $(G_2, \alpha_2)$}
 Below, we give the matrix $B$ of bisectional curvatures and the matrix $Z$ as calculated in MAPLE using the above commands.  For the $25\times 25$ matrix $Z_{AB, CD}$, the pairs $AB$ and $CD$ are ordered into a list of 25 elements as:
 $(C(1), C(1))$, $(C(2), C(1))$, $(C(3), C(1)),...etc.$
 The matrix $Z_1$ gives the first 13 columns of $Z$ while $Z_2$ gives the  next 12.
\vspace{50pt}

$$B= \left[ \begin {array}{ccccc} 2&5/2&3&0&3/2\\ \noalign{\medskip}5/2&2&0
&3&3/2\\ \noalign{\medskip}3&0&6&-3/2&3/2\\ \noalign{\medskip}0&3&-3/2
&6&3/2\\ \noalign{\medskip}3/2&3/2&3/2&3/2&3\end {array} \right]
$$

\newpage

$$Z_1=   \left[ \begin {array}{ccccccccccccc} 0&0&0&0&0&0&0&0&0&0&0&0&0
\\ \noalign{\medskip}0&0&\sqrt {2}\sqrt {6}&0&0&5/2&0&0&0&0&0&0&0
\\ \noalign{\medskip}0&\sqrt {2}\sqrt {6}&0&0&0&0&0&0&3/2&0&3&0&0
\\ \noalign{\medskip}0&0&0&0&0&0&0&3/2&0&0&0&0&0\\ \noalign{\medskip}0
&0&0&0&0&0&0&0&0&0&0&0&0\\ \noalign{\medskip}0&5/2&0&0&0&0&0&0&\sqrt {
2}\sqrt {6}&0&\sqrt {2}\sqrt {6}&0&0\\ \noalign{\medskip}0&0&0&0&0&0&0
&0&0&0&0&0&0\\ \noalign{\medskip}0&0&0&3/2&0&0&0&0&0&0&0&0&0
\\ \noalign{\medskip}0&0&3/2&0&0&\sqrt {2}\sqrt {6}&0&0&0&0&0&0&0
\\ \noalign{\medskip}0&0&0&0&0&0&0&0&0&0&0&0&0\\ \noalign{\medskip}0&0
&3&0&0&\sqrt {2}\sqrt {6}&0&0&0&0&0&0&0\\ \noalign{\medskip}0&0&0&0&0&0
&0&0&0&0&0&0&0\\ \noalign{\medskip}0&0&0&0&0&0&0&0&0&0&0&0&0
\\ \noalign{\medskip}0&0&0&0&0&0&0&0&0&0&0&0&0\\ \noalign{\medskip}0&0
&0&0&0&0&0&0&0&0&0&0&0\\ \noalign{\medskip}0&0&0&0&0&0&0&0&0&0&0&3/2&0
\\ \noalign{\medskip}0&\sqrt {2}\sqrt {6}&0&0&0&0&0&0&3&0&3/2&0&0
\\ \noalign{\medskip}0&0&0&0&0&0&0&0&0&0&0&0&0\\ \noalign{\medskip}0&0
&0&0&0&0&0&0&0&0&0&0&0\\ \noalign{\medskip}0&0&0&0&0&0&0&0&0&0&0&0&0
\\ \noalign{\medskip}0&0&0&0&3/2&0&0&0&0&0&0&0&0\\ \noalign{\medskip}0
&0&0&0&0&0&0&0&0&3/2&0&0&0\\ \noalign{\medskip}0&0&0&0&0&0&0&0&0&0&0&0
&0\\ \noalign{\medskip}0&0&0&0&0&0&0&0&0&0&0&0&0\\ \noalign{\medskip}0
&0&0&0&0&0&0&0&0&0&0&0&0\end {array} \right] $$

\newpage

$$
Z_2= \left[ \begin {array}{cccccccccccc} 0&0&0&0&0&0&0&0&0&0&0&0
\\ \noalign{\medskip}0&0&0&\sqrt {2}\sqrt {6}&0&0&0&0&0&0&0&0
\\ \noalign{\medskip}0&0&0&0&0&0&0&0&0&0&0&0\\ \noalign{\medskip}0&0&0
&0&0&0&0&0&0&0&0&0\\ \noalign{\medskip}0&0&0&0&0&0&0&3/2&0&0&0&0
\\ \noalign{\medskip}0&0&0&0&0&0&0&0&0&0&0&0\\ \noalign{\medskip}0&0&0
&0&0&0&0&0&0&0&0&0\\ \noalign{\medskip}0&0&0&0&0&0&0&0&0&0&0&0
\\ \noalign{\medskip}0&0&0&3&0&0&0&0&0&0&0&0\\ \noalign{\medskip}0&0&0
&0&0&0&0&0&3/2&0&0&0\\ \noalign{\medskip}0&0&0&3/2&0&0&0&0&0&0&0&0
\\ \noalign{\medskip}0&0&3/2&0&0&0&0&0&0&0&0&0\\ \noalign{\medskip}0&0
&0&0&0&0&0&0&0&0&0&0\\ \noalign{\medskip}0&0&0&0&3/2&0&0&0&0&0&0&0
\\ \noalign{\medskip}0&0&0&0&0&0&0&0&0&3/2&0&0\\ \noalign{\medskip}0&0
&0&0&0&0&0&0&0&0&0&0\\ \noalign{\medskip}0&0&0&0&0&0&0&0&0&0&0&0
\\ \noalign{\medskip}3/2&0&0&0&0&0&0&0&0&0&0&0\\ \noalign{\medskip}0&0
&0&0&0&0&0&0&0&0&0&0\\ \noalign{\medskip}0&0&0&0&0&0&0&0&0&0&3/2&0
\\ \noalign{\medskip}0&0&0&0&0&0&0&0&0&0&0&0\\ \noalign{\medskip}0&0&0
&0&0&0&0&0&0&0&0&0\\ \noalign{\medskip}0&3/2&0&0&0&0&0&0&0&0&0&0
\\ \noalign{\medskip}0&0&0&0&0&0&3/2&0&0&0&0&0\\ \noalign{\medskip}0&0
&0&0&0&0&0&0&0&0&0&0\end {array} \right] $$


\begin{thebibliography}{10}
\bibitem{Bando}  Bando, S., {\sl On the classification of three-dimensional compact K\"ahler manifolds of nonnegative bisectional curvature}, J. Differential Geom. \textbf{19} (1984), no. 2, 283-297, MR0755227, Zbl 0547.53034.

\bibitem{Besse} Besse, A. L., {\sl
Einstein manifolds},
  Springer-verlag,  (1987),
MR2371700, Zbl 1147.53001.


 %\bibitem{B} Berger, M., {\sl Sur les vari\'et\'es d'Einstein compactes}, C.R. IIIe R\'eunion Math. Expression latine, Namur 1965, 35-55, MR0238226,
\bibitem{BG} Bishop, R.L., Goldberg, S. I.,  {\sl On the second cohomology group of a \K manifold of
positive curvature}, Proc. Amer. Math. Soc. \textbf{16}, (1965), 119-122, MR0172221, Zbl 0125.39403.
\bibitem{Borel-1}  Borel, A., {\sl K\"ahlerian coset spaces of semi-simple Lie groups}, Proc. Nat. Acad. Sci.
U.S.A., \textbf{ 40} (1954), 1147-1151, MR0077878,
Zbl 0058.16002.


\bibitem{Borel-2}  Borel, A., {\sl On the curvature tensor of Hermitian symmetric manifolds}, Ann. of Math.,
\textbf{71} (1960), 508-521, MR0111059, Zbl 0100.36101.

\bibitem{BH}  Borel, A. and Hirzebruch, F.,{\sl Characteristic classes and homogeneous spaces I}, Amer.
J. Math., \textbf{80} (1958), 458-538, MR0102800,
Zbl 0097.36401.

\bibitem{Bourbaki} Bourbaki, N., {\sl  Lie groups and Lie algebras},
  Springer (2002), MR1890629,
Zbl 1145.17001.

\bibitem{CT}Chau, A., Tam, L.F.,{\sl On quadratic orthogonal bisectional curvature}, J. Diff. Geom., to appear, arXiv:1108.6252.

\bibitem{CT2}Chau, A., Tam, L.F.,{\sl K\"ahler C-spaces of classical type and quadratic bisectional curvature}, arXiv:1202.4542.  \bibitem{Chen} Chen, X. X., {\sl On K\"ahler manifolds with positive orthogonal bisectional curvature}, Adv. Math. \textbf{215} (2007), no. 2, 427-445, MR2355611, Zbl 1131.53038.

\bibitem{Fulton-Harris}  Fulton, W.,   Harris, J.{\sl Representation theory : a first course}, Springer-Verlag  (1991),
MR1153249,
Zbl 0744.22001.




 \bibitem{GK}  Goldberg, S.I., Kobayashi, S., {\sl Holomorphic bisectional curvature},
J. Differential Geom., \textbf{1}, (1967), 225-233, MR0227901, Zbl 0169.53202.
\bibitem{GZ}Gu, H.L., Zhang, Z.H., {\sl An Extension of Mok's Theorem on the Generalized Frankel Conjecture}, Sci. China Math. \textbf{53} (2010), no. 5, 1253-1264, MR2653275, Zbl 1204.53058.

  \bibitem{HowardSmythWu1981}Howard, A., Smyth, B. and Wu, H. {\sl On compact K\"ahler manifolds of nonnegative bisectional curvature, I,} Acta Math. \textbf{147} (1981), 51-56, MR0631087, Zbl 0473.53055.

\bibitem{Itoh}Itoh, M.{\sl On curvature properties of K\"ahler $C$-spaces}, J. Math. Soc. Japan \textbf{30} (1978), no. 1, 39--71, MR0470904, Zbl 0363.53024 .






\bibitem{LiWuZheng} Li, Q., Wu, D. and Zheng, F., {\sl An example of compact K\"ahler manifold with nonnegative quadratic bisectional curvature}, Proc. Amer. Math. Soc., to appear, arXiv:1110.1749.

\bibitem{Mok} Mok, N., {\sl The uniformization theorem for compact \K manifolds of non-
negative bisectional curvature}, J. Differential Geom. \textbf{27} (1988), no. 2, 179-214, MR0925119, Zbl 0642.53071.
\bibitem{M} Mori, S. {\sl Projective manifolds with ample tangent bundles}, Ann. of Math. (2)
\textbf{110} (1979), 593-606, MR0554387, Zbl 0423.14006.



 \bibitem{Serre} Serre, J.-P., {\sl  Complex semisimple Lie algebras}, Springer (2001),
MR1808366,
Zbl 1058.17005.


\bibitem{SY}  Siu, Y. T., Yau, S. T., {\sl Complex \K manifolds of positive bisectional
curvature}, Invent. Math. \textbf{59} (1980), 189-204, MR0577360, Zbl 0442.53056.
\bibitem{Wang}  Wang, H. C., {\sl Closed manifolds with homogeneous complex structures}, Amer. J. Math.,
\textbf{76} (1954), 1-32, MR0066011,
Zbl 0055.16603.





\bibitem{wilking} Wilking, B., {\sl A Lie algebraic approach to Ricci flow invariant curvature conditions and Harnack inequalities }, J. Reine Angew. Math., to appear, arXiv:1011.3561


\bibitem{Wu} Wu, H., {\sl On compact \K manifolds of nonnegative bisectional curvature II},
Acta Math. \textbf{147} (1981), 57-70, MR0631088, Zbl 0473.53056.
\bibitem{WuYauZheng} Wu, D., Yau, S.T., Zheng, F. {\sl A degenerate Monge Amp\`ere equation and the boundary classes of \K cones}, Math. Res. Lett. \textbf{16} (2009), no.2, 365-374, MR2496750, Zbl 1183.32018.

\bibitem{Yau} Yau, S.-T. (ed), {\sl Seminar on differential geometry}, Princeton University Press; Tokyo: University of Tokyo Press, \textbf{1982}, MR0645728, Zbl 0471.00020.

\bibitem{Zheng} Zheng, F., {\sl Private communication}.
    \end{thebibliography}
\end{document}